\documentclass{amsart}

\usepackage[colorlinks=true,linkcolor=magenta,citecolor=magenta]{hyperref}%
\usepackage[english]{babel}
\usepackage[utf8x]{inputenc}
\usepackage[T1]{fontenc}
\usepackage[margin=3.0cm]{geometry}
\usepackage{lipsum}
\usepackage[T1]{fontenc}
\usepackage{float}
\usepackage{amssymb}
\usepackage{amsfonts}
\usepackage{amsmath}
\usepackage{graphicx}
\usepackage{longtable}
\usepackage{dcolumn}
\usepackage{tabularx}
\usepackage{rotating}
\usepackage{xtab}
\usepackage{mathrsfs} 
\usepackage{bbm}
\usepackage{booktabs}
\usepackage{paralist} 
\usepackage{verbatim}
\usepackage{subfiles} 
\theoremstyle{plain}
\newtheorem{Th}{Theorem}[section]
\newtheorem{Lemma}[Th]{Lemma}
\newtheorem{Cor}[Th]{Corollary}
\newtheorem{Prop}[Th]{Proposition}

\theoremstyle{definition}
\newtheorem{Def}[Th]{Definition}
\newtheorem{rem}[Th]{Remark}
\newtheorem{Conj}[Th]{Conjecture}
\newtheorem{Rem}[Th]{Remark}

 \DeclareOldFontCommand{\rm}{\normalfont\rmfamily}{\mathrm}
\newcommand{\C}{{\rm{CR}}}
\newcommand{\FZZ}{{\rm{FZZ}}}
\newcommand{\diam}{{\rm{diam}}}

\newcommand{\disk}{\mathrm{disk}}
\newcommand{\qs}{\mathrm{QS}}
\newcommand{\LF}{\mathrm{LF}}
\newcommand{\sle}{\mathrm{SLE}}
\newcommand{\qd}{\mathrm{QD}}
\newcommand{\qt}{\mathrm{QT}}
\newcommand{\lo}{{\rm{loop}}}

\newcommand{\bubble}{\mathrm{Bubble}}
\newcommand{\bub}{\mathrm{bubble}}

\newcommand{\conf}{\mathrm{conf}}

\newcommand{\weld}{\mathrm{Weld}}
\newcommand{\var}{\mathrm{Var}}
\newcommand{\leb}{\mathrm{Leb}}

\newcommand{\rad}{\mathrm{rad}}

\newcommand{\ms}{\mathrm{MS}}
\newcommand{\saw}{\mathrm{SAW}}
\newcommand{\md}{\mathrm{MD}}
\newcommand{\sab}{\mathrm{SAB}}
\newcommand{\Rad}{\mathrm{Rad}}
\usepackage{graphicx}
\usepackage{url}
\usepackage[colorinlistoftodos]{todonotes}
\numberwithin{equation}{subsection}

\title{The SLE Bubble Measure via Conformal Welding of Quantum Surfaces}
\author{Da Wu}
\begin{document}
\begin{abstract}
We showed that the $\sle_{\kappa}(\rho)$ bubble measure recently constructed by Zhan arises naturally from the conformal welding of two Liouville quantum gravity (LQG) disks. The proof relies on (1) a ``quantum version'' of the limiting construction of the SLE bubble,  (2) the conformal welding between quantum triangles and quantum disks due to Ang, Sun and Yu, and (3) the uniform embedding techniques of Ang, Holden and Sun. As a by-product of our proof, we obtained a decomposition formula of the $\sle_\kappa(\rho)$ bubble measure. Furthermore, we provided two applications of our conformal welding results. First, we computed the moments of the conformal radius of the $\sle_\kappa(\rho)$ bubble on $\mathbb H$ conditioning on surrounding $i$. The second application concerns the bulk-boundary correlation function in the Liouville Conformal Field Theory (LCFT). Within probabilistic frameworks, we derived a formula linking the bulk-boundary correlation function in the LCFT to the joint law of left \& right quantum boundary lengths and the quantum area of the two-pointed quantum disk. This relation is used by Ang, Remy, Sun and Zhu in a concurrent work to verify the formula of two-pointed bulk-boundary correlation function in physics predicted by Hosomichi (2001). 
\end{abstract}
\maketitle
\tableofcontents
\section{Introduction}
The Schramm-Loewner evolution ($\sle$) and Liouville quantum gravity (LQG) are central objects in Random Conformal Geometry and it was shown in \cite{She16} and \cite{DMS14} that SLE curves arise naturally as the interfaces of LQG surfaces under conformal welding. Conformal welding results in \cite{She16,DMS14} mainly focus on the infinite volume LQG surfaces. Recently, Ang, Holden and Sun \cite{Weldisk} showed that conformal welding of finite-volume quantum surfaces called two-pointed quantum disks can give rise to canonical variants of $\sle$ curves with two marked points. Later, it was shown by Ang, Holden and Sun \cite{SLEloopviawelding} that another canonical invariant of $\sle$ called $\sle$ Loop is the natural welding interface of two quantum disks without marked points. The resulting LQG surface is called the \textit{quantum sphere}, which describes the scaling limit of classical planar map models with spherical topology.
\par 
As will be reviewed in Section \ref{Sec: SLE bubble}, the rooted $\sle_\kappa(\rho)$ bubble measure on $\mathbb H$ is an important one parameter family of random Jordan curves constructed by Zhan \cite{SLEbubble} for all $\kappa>0$ and $\rho>-2$. When $\kappa>4$ and $\rho\in (-2,\frac{\kappa}{2}-4]$, the law of the bubble is a probability measure and satisfies conformal invariance property (\cite[Theorem 3.10]{SLEbubble}). When $\rho>(-2)\vee(\frac{\kappa}{2}-4)$, the law of the bubble is a $\sigma$-finite infinite measure and satisfies conformal covariance property (\cite[Theorem 3.16]{SLEbubble}). In both cases, an instance $\eta$ of $\sle_\kappa(\rho)$ bubble is characterized by the following Domain Markov Property (DMP): suppose $\tau$ is a positive stopping time for $\eta$, then conditioning on $\eta[0,\tau]$ and the event that $\eta$ is not completed at $\tau$, the rest of $\eta$ is a chordal $\sle_\kappa(\rho)$ on $\mathbb H\backslash\eta[0,\tau]$ (\cite[Theorem 3.16]{SLEbubble}). Moreover, it was shown that $\sle_\kappa(\rho)$ bubble measure can be viewed as the weak limit of chordal $\sle_\kappa(\rho)$ on $\mathbb H$ from $0$ to $\varepsilon$ as $\varepsilon\to 0^+$ (with force point at $0^-$) after suitable rescaling (\cite[Theorem 3.20]{SLEbubble}).
\par 
On the other hand, it was shown in \cite[Section 4]{FZZ} that a particular $\sle_\kappa(\rho)$ bubble curve can be obtained from  conformally welding two Liouville quantum gravity surfaces of the disk topology. This  was used to derive the \textit{Fateev-Zamolodchikov-Zamolodchikov} (FZZ) formula in Liouville theory, which serves as a crucial input to the proof of the imaginary DOZZ formula for  conformal loop ensemble (CLE) on the Riemann sphere \cite{AS21}. This paper generalizes the conformal welding result in \cite{FZZ} to all $\rho>-2$; see Remark \ref{rem:relation to FZZ} for the precise relation between our result and the one in  \cite{FZZ}.
\par
The rest of the paper is organized as follows. We state our main conformal welding results, including Theorem \ref{Conditional Welding} and Theorem \ref{Main welding theorem}, in Section \ref{sec:main theorem} and Theorem \ref{welding of generalized weight} with its applications in Section \ref{sec:generalization and application}. All the necessary backgrounds on Random Conformal Geometry will be reviewed in Section \ref{sec:preliminaries}. We first prove Theorem \ref{Conditional Welding} in Sections \ref{sec: curve law} and \ref{sec:field law} and then prove Theorem \ref{Main welding theorem} in Section \ref{sec:proof of main based on conditional welding} based on Theorem \ref{Conditional Welding} and the \textit{uniform embedding} of LQG surfaces. Next, we prove Theorem \ref{welding of generalized weight}, which is the generalization of Theorem \ref{Conditional Welding} to the case when the bulk insertion of the quantum surface has generic weight. In Section \ref{sec:computation of conformal radius}, we discuss two applications of Theorem \ref{welding of generalized weight}: The first application concerns the computation of the moments of the conformal radius of the $\sle_\kappa(\rho)$ bubble on $\mathbb H$ conditioning on surrounding $i$; Secondly, we derive a formula linking the bulk-boundary correlation function in LCFT to the joint law of left \& right quantum boundary lengths and the quantum area of the two-pointed quantum disk. Finally, in Section $\ref{sec:open problems and future research}$, we will discuss several conjectures that arise naturally from the contexts of this paper, including a generalization of the $\sle_\kappa(\rho)$ bubble  and the scaling limit of bubble-decorated disk quadrangulations.
 
\subsection{$\sle_\kappa(\rho)$ bubble measures via conformal welding of quantum disks}\label{sec:main theorem}
Let $\bubble_{\mathbb H}(p)$ be the space of rooted simple loops on $\mathbb H$ with root $p\in \mathbb R$. Precisely, an oriented simple closed loop $\eta$ is in $\bubble_{\mathbb H}(p)$ if and only if $p\in \eta,(\eta\backslash\lbrace p\rbrace)\subseteq\overline{\mathbb H}$. Throughout this paper, for $\eta\in \bubble_{\mathbb H}(p)$, let $D_\eta(p)$ be the connected component of $\mathbb H\backslash\eta$ which is encircled by $\eta$ and let $D_\eta(\infty)$ be the domain $\mathbb H\backslash(\eta\cup D_\eta(p))$ containing $\infty$. The point $p$ corresponds to two pseudo boundary marked points $p^-$ and $p^+$ on $D_\eta(\infty)$. Let $\sle_{\kappa,0}^{\bub}(\rho)$ denote the rooted $\sle_\kappa(\rho)$ bubble measure with root $0$ studied in \cite{SLEbubble} (see Definition \ref{Def of rooted SLE bubble}) and this is a $\sigma$-finite infinite measure on the space $\bubble_{\mathbb H}(0)$.
\par
For each $\gamma\in (0,2)$, there is a family of LQG surfaces with disk topology called \textit{quantum disks}. There is also a weight parameter $W>0$ associated with the family of quantum disks. Let $\mathcal M^{\disk}_{0,2}(W)$ denote the two-pointed weight-$W$ quantum disk; both marked points are on the boundary, each with weight $W$ (see Definition \ref{Def: thick quantum disk} and \ref{Def: thin quantum disk} for two regimes in terms of $W$). When $W=2$, the two marked points in quantum disk $\mathcal M_{0,2}^{\disk}(2)$ are \emph{quantum typical} w.r.t. the quantum boundary length measure (\cite[Proposition $A.8$]{Weldisk}) and we denote the $\mathcal M_{0,2}^{\disk}(2)$ by $\qd_{0,2}$. Let $\qd_{0,1}$ and $\qd_{1,1}$ denote the typical quantum disks with one boundary marked point and with one bulk $\&$ one boundary marked point respectively (see Definition \ref{qd with arbitrary marked points} for the class of typical quantum disks and its variants).
\par 
Let $\qd_{0,1}(\ell)$ and $\qd_{1,1}(\ell)$ be the disintegration of $\qd_{0,1}$ and $\qd_{1,1}$ over its quantum boundary length respectively, i.e., $\qd_{0,1} = \int_0^{\infty}\qd_{0,1}(\ell)d\ell$ and $\qd_{1,1} = \int_0^{\infty}\qd_{1,1}(\ell)d\ell$, and both $\qd_{0,1}(\ell)$ and $\qd_{1,1}(\ell)$ should be understood as $\qd_{0,1}$ and $\qd_{1,1}$ restricted to having total boundary length $\ell$ respectively. Similarly, let $\mathcal M^{\disk}_{0,2}(W;\cdot,\ell)$ be the disintegration of $\mathcal M^{\disk}_{0,2}(W)$ over its right boundary, i.e., $\mathcal M^{\disk}_{0,2}(W) = \int_0^{\infty}\mathcal M^{\disk}_{0,2}(W;\cdot,\ell)d\ell$, and the $\mathcal M^{\disk}_{0,2}(W;\cdot,\ell)$ again represents the $\mathcal M^{\disk}_{0,2}(W)$ restricted to having the right boundary length $\ell$. Let $\int_0^\infty\mathcal{M}^{\disk}_{0,2}(W;\cdot,\ell)\times \qd_{0,1}(\ell)d\ell$ be the curve-decorated quantum surface obtained by conformally welding the right boundary of $\mathcal M^{\disk}_{0,2}(W)$ and total boundary of $\qd_{0,1}$. Similarly, $\int_0^\infty \mathcal M^{\disk}_{0,2}(W;\cdot,\ell)\times \qd_{1,1}(\ell) d\ell$ is the quantum surface obtained by welding the right boundary of $\mathcal M^{\disk}_{0,2}(W)$ and the total boundary of $\qd_{1,1}$.
\par 
 In theoretical physics, LQG originated in A. Polyakov's seminal work \cite{Pol81} where he proposed a theory of summation over the space of Riemannian metrics on fixed two dimensional surface. The fundamental building block of his framework is the Liouville conformal field theory (LCFT), which describes the law of the conformal factor of the metric tensor in a surface of fixed complex structure. The LCFT was made rigorous in probability theory in various different topologies; see \cite{DKRV16} and \cite{HRV18} for the case of Riemann sphere and of simply connected domain with boundary respectively, and \cite{DRV16,Rem18,GRV19} for the case of other topologies. 
\par
To be precise, let $P_{\mathbb H}$ be the probability measure corresponding to the law of the free-boundary Gaussian free field (GFF) on $\mathbb H$ normalized to having average zero on the unit circle in upper half plane unit circle $\partial\mathbb D\cap\mathbb H$. The infinite measure $\LF_{\mathbb H}(d\phi)$ is defined by first sampling $(h,\mathbf c)$ according to $P_{\mathbb H}\times \left[e^{-Qc}dc\right]$ and then letting $\phi(z)=h(z)-2Q\log|z|_{+}+\mathbf c$, where $Q=\frac{2}{\gamma}+\frac{\gamma}{2}$ and $|z|_{+} = \max \lbrace |z|,1\rbrace$. We can further define the Liouville field with bulk or/and boundary insertion(s), e.g., $\LF_{\mathbb H}^{(\beta,p)}$ and $\LF_{\mathbb H}^{(\alpha,z),(\beta,p)}$, where $p\in \mathbb R$ and $z\in \mathbb H$. To make sense of $\LF_{\mathbb H}^{(\beta,p)}$, where $p\in \partial\mathbb H$, let $\LF_{\mathbb H}^{(\beta,p)}:=\lim_{\varepsilon\to 0}\varepsilon^{\beta^2/4}e^{\frac{\beta}{2}\phi_\varepsilon(p)}\LF_{\mathbb H}(d\phi)$, $\phi_\varepsilon$ being a suitable regularization at scale $\varepsilon$ of $\phi$. In terms of $\LF_{\mathbb H}^{(\alpha,z),(\beta,p)}$ with $z\in \mathbb H$ and $p\in \partial\mathbb H$, we use the similar limiting procedure. Let $\LF_{\mathbb H}^{(\beta,p),(\alpha,z)}:=\lim_{\varepsilon\to 0}\varepsilon^{\alpha^2/2}e^{\alpha\phi_\varepsilon(z)}\LF_{\mathbb H}^{(\beta,p)}(d\phi)$, $\phi_\varepsilon(z)$ being some suitable renormalization at scale $\varepsilon$. By Cameron-Martin shift (a.k.a. Girsanov's theorem), the $\LF_{\mathbb H}^{(\beta,p)}$ represents a sample from $\LF_{\mathbb H}$ plus a $\beta$-$\log$ singularity at boundary marked point $p$ locally. Similarly, $\LF_{\mathbb H}^{(\alpha,z),(\beta,p)}$ should be viewed as $\LF_{\mathbb H}$ plus one boundary $\beta$-$\log$ singularity at $p$ and one bulk $\alpha$-$\log$ singularity at $z$.
\par
For $q\in \mathbb H$ and $p\in \partial\mathbb H$, let $\bubble_{\mathbb H}(p,q)$ be the space of rooted simple loops on $\mathbb H$ rooted at $p$ and surrounding $q$. Precisely, an oriented simple closed loop $\eta$ is in $\bubble_{\mathbb H}(p,q)$ if and only if $p\in \eta,(\eta\backslash\lbrace p\rbrace)\subseteq\overline{\mathbb H}$ and $q\in D_\eta(p)$. Let $\sle_{\kappa,0}^{\bub}(\rho)[d\eta|{i\in D_\eta(0)}]$ denote the conditional law of $\sle_{\kappa,0}^{\bub}(\rho)$ on surrounding $i$ and this is a probability measure on $\bubble_{\mathbb H}(0,i)$.
\begin{figure}[ht]
	\includegraphics[scale=0.8]{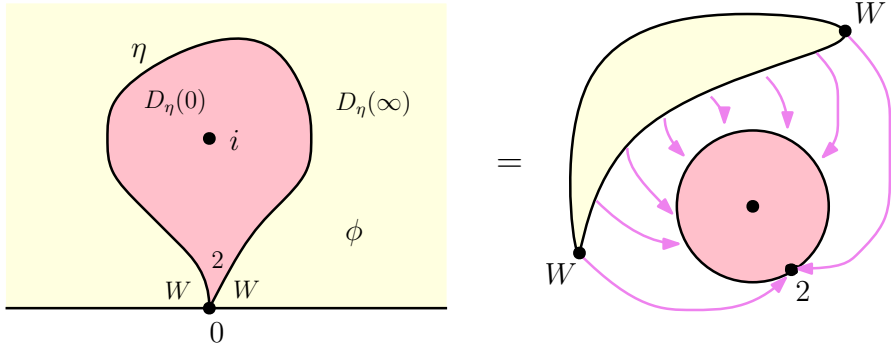}
	\centering
	\caption{Illustration of Theorem \ref{Conditional Welding} when $W\geq\frac{\gamma^2}{2}$: suppose $(\phi,\eta)$ is sampled from $C\cdot\LF_{\mathbb H}^{(\gamma,i),\beta_{2W+2},0)}(d\phi)\times \sle_{\kappa,0}^{\bub}(\rho)[d\eta|{i\in D_\eta(0)}]$ as shown on the left above, then the law of $(D_\eta(0), \phi, i,0)$ and $(D_\eta(\infty), \phi, 0^-, 0^+)$ viewed as a pair of marked quantum surfaces is equal to $\int_0^\infty \mathcal M^{\disk}_{0,2}(W;\cdot,\ell)\times \qd_{1,1}(\ell) d\ell$, i.e., the quantum surface obtained by conformally welding the total boundary of a sample from $\qd_{1,1}$ with the right boundary of a sample from $\mathcal M^{\disk}_{0,2}(W)$.}
	\label{fig: Main_welding_thm}
\end{figure} 
\begin{Th}\label{Conditional Welding}
	Fix $\gamma\in (0,2)$. For $W>0$, let $\rho=W-2$ and $\beta_{2W+2}=\gamma-\frac{2W}{\gamma}$. There exists some constant $C\in (0,\infty)$ such that suppose $(\phi,\eta)$ is sampled from
	\begin{equation}
		C\cdot\LF_{\mathbb H}^{(\gamma,i),(\beta_{2W+2},0)}(d\phi)\times \sle_{\kappa,0}^{\bub}(\rho)[d\eta|{i\in D_\eta(0)}],
	\end{equation}
	 then the law of $(D_\eta(0), \phi, i,0)$ and $(D_\eta(\infty), \phi, 0^-, 0^+)$ viewed as a pair of marked quantum surfaces is equal to
	\begin{equation}\label{qs:conditionalwelding}
		\int_0^\infty \mathcal M^{\disk}_{0,2}(W;\cdot,\ell)\times \qd_{1,1}(\ell) d\ell.
	\end{equation}
\end{Th}
\begin{Rem}\label{rem:relation to FZZ}
    In \cite{FZZ}, the authors considered the same type of conformal welding with $W=\frac{\gamma^2}{2}-2$ (\cite[Theorem 4.1]{FZZ}). This particular conformal welding result was used to obtained the so-called FZZ formula proposed in \cite{FZZ00}. However in \cite[Theorem 4.1]{FZZ}, the law of the welding interface was not explicitly specified. Here in the above Theorem \ref{Conditional Welding}, we generalized the \cite[Theorem 4.1]{FZZ} to all $W>0$, and furthermore identified the law of the welding interface to be the $\sle_\kappa(W-2)$ bubble constructed in \cite{SLEbubble}.
\end{Rem}
The proof of Theorem \ref{Conditional Welding} is separated into two parts. In Section \ref{sec: curve law}, we show that the law of welding interface of curve-decorated quantum surface (\ref{qs:conditionalwelding}) is the $\sle_{\kappa,0}^{\bub}(\rho)$ conditioning on surrounding $i$ and moreover, it is independent of the underlying random field. To identify the law of the welding interface, we essentially use the ``quantum version'' of the limiting construction of the $\sle_\kappa(\rho)$ bubble; see Corollary \ref{conditional convergence lemma} for the statement on the Euclidean case. More precisely, we first consider the conformal welding of $\mathcal M^{\disk}_{0,2}(W)$ and $\qd_{1,2}$, i.e., the typical quantum disk with two boundary and one bulk marked points, whose welding interface is the chordal $\sle_\kappa(\rho)$ conditioning on passing to the left of some fixed point in $\mathbb H$ (Lemma \ref{Lem:02+12}). Then conditioning on the quantum boundary length of $\qd_{1,2}$ between two boundary marked points shrinks to zero, we can construct a coupling with (\ref{qs:conditionalwelding}). Under such coupling, these two welding interfaces will match with high probability (Lemma \ref{Key coupling lemma}). The independence of curve with the underlying random field follows from the coupling argument and Corollary \ref{conditional convergence lemma} on the deterministic convergence of chordal $\sle_\kappa(\rho)$. 
\par The proof of the law of the underlying random field after conformal welding of two quantum disks, i.e., the quantum surface (\ref{qs:conditionalwelding}), is done in two steps. In Section \ref{sec:field law}, we first consider (\ref{qs:conditionalwelding}) when $0<W<\frac{\gamma^2}{2}$, i.e., when the two-pointed disk is thin. By Lemma \ref{Lemma decomp of thin disk},  the thin quantum disk of weight $W$ with one additional typical boundary marked point can be viewed as the concatenation of three independent disks: two thin disks of weight $W$ and one thick disk of weight $\gamma^2-W$ with one typical boundary marked point. Therefore, we can first sample one typical boundary marked point on $\mathcal M^{\disk}_{0,2}(W)$ and then sample two typical boundary marked points on $\qd_{1,1}(\gamma,\alpha)$, i.e., the quantum disk with one generic boundary insertion (Definition \ref{Def: general QD}). The field law after conformally welding $\mathcal M^{\disk}_{2,\bullet}(W)$ and $\qd_{1,3}(\gamma,\alpha)$ can be derived from conformal welding results for quantum triangles in \cite{QT22}. After de-weighting all the additional marked points, we solve the case when $0<W<\frac{\gamma^2}{2}$. To extend to the full range $W>0$, we inductively weld thin disks outside $\qd_{1,1}(\gamma,\alpha)$. By Theorem \ref{Disk welding theorem}, a thick disk can be obtained by welding  multiple thin disks. This concludes the outline of the proof of Theorem \ref{Conditional Welding}. 
\par 
Next, we use the techniques of \emph{uniform embedding} of quantum surfaces from \cite{InteofSLE} to remove the bulk insertion in Theorem \ref{Conditional Welding} so that the welding interface is the $\sle_{\kappa}(\rho)$ bubble without conditioning. In order to introduce Theorem \ref{Main welding theorem}, we quickly recall the setups of the uniform embedding of upper half plane $\mathbb H$. Let $\conf(\mathbb H)$ be the group of conformal automorphisms of $\mathbb H$ where the group multiplication $\cdot$ is the function composition $f\cdot g = f\circ g$. Let $\mathbf m_{\mathbb H}$ be a \emph{Haar measure} on $\conf(\mathbb H)$, which is both left and right invariant. Suppose $\mathfrak{f}$ is sampled from $\mathbf m_{\mathbb H}$ and $\phi\in C^\infty_0(\mathbb H)'$, i.e., $\phi$ is a generalized function, then we call the random function
\begin{equation}\label{eqn:equiv trans rule}
	\mathfrak{f}\bullet_\gamma\phi=\phi\circ \mathfrak{f}^{-1}+Q|\log(\mathfrak{f}^{-1})'|
\end{equation}
 the \textit{uniform embedding} of $(\mathbb H,\phi)$ via $\mathbf{m}_{\mathbb H}$. By invariance property of Haar measure, the law of $\mathfrak{f}\bullet_\gamma\phi$ only depends on $(\mathbb H,\phi)$ as quantum surface. We write $\mathbf{m}_{\mathbb H}\ltimes \left(\int_0^\infty\mathcal{M}^{\disk}_{0,2}(W;\cdot,\ell)\times \qd_{0,1}(\ell)d\ell\right)$ as the law of $(\mathfrak f\bullet_{\gamma} h,f(\eta), f(r))$, where $(\mathbb H,h,\eta,r)$ is an embedding of a sample from curve-decorated quantum surface $\int_0^\infty\mathcal{M}^{\disk}_{0,2}(W;\cdot,\ell)\times \qd_{0,1}(\ell)d\ell$, and $\mathfrak f$ is sampled independently from $\mathbf m_{\mathbb H}$. Notice that here the $\mathbf m_{\mathbb H}$ does not fix our boundary marked point $r$, which initially is the root of $\eta$.
 \par 
 The equation (\ref{eqn:equiv trans rule}) also provides a natural equivalence relation $\sim_\gamma$ over curve-decorated quantum surfaces; two curve-decorated quantum surfaces $(D_1,\phi_1,\eta_1,\omega_1,\ldots,\omega_n)$ with $\omega_i\in D_1\cup\partial D_1$ and $(D_2,\phi_2,\eta_2,z_1\ldots,z_n)$ with $z_i\in D_2\cup\partial D_2$ are equivalent as quantum surfaces, denoted by
 \begin{equation}
     (D_1,\phi_1,\eta_1,\omega_1,\ldots,\omega_n)\sim_\gamma (D_2,\phi_2,\eta_2,z_1\ldots,z_n),
 \end{equation}
if there is a conformal map $\psi:D_1\to D_2$ such that $\phi_2=\psi\bullet_\gamma \phi_1$, $\eta_2=\psi(\eta_1)$, and $\psi(\omega_i)=z_i,1\leq i\leq n$. 
\par
 We can also consider the case when the marked points are fixed under the action of Haar measure. For fixed $p\in \partial\mathbb H$, let $\conf(\mathbb H,p)$ be the subgroup of $\conf(\mathbb H)$ fixing $p$ and let $\mathbf m_{\mathbb H,p}$ be a Haar measure on $\conf(\mathbb H,p)$. The curve-decorated quantum surface $\int_0^\infty \mathcal M^{\disk}_{0,2}(W;\cdot,\ell)\times \qd_{0,1}(\ell) d\ell$ can be identified as a measure on the product space $\left(C_0^{\infty}(\mathbb H)'/\conf(\mathbb H,p)\right)\times \bubble_{\mathbb H}(p)$. Therefore, the measure $\mathbf{m}_{\mathbb H,p}\ltimes\left(\int_0^\infty \mathcal M^{\disk}_{0,2}(W;\cdot,\ell)\times \qd_{0,1}(\ell) d\ell\right)$ can be defined in the exact same way as $\mathbf{m}_{\mathbb H}\ltimes \left(\int_0^\infty\mathcal{M}^{\disk}_{0,2}(W;\cdot,\ell)\times \qd_{0,1}(\ell)d\ell\right)$ for fixed $p\in \partial\mathbb H$. 
 \par
 For any fixed $p\in \mathbb H$, let $\sle_{\kappa,p}^{\bub}(\rho)$ denote the $\sle_\kappa(\rho)$ bubble measure rooted at $p\in \mathbb R$. It is easily defined as the image of $\sle_{\kappa,0}^{\bub}(\rho)$ under the shifting map $f_p:z\mapsto z+p$. 
 \begin{figure}[H]
	\includegraphics[scale=0.9]{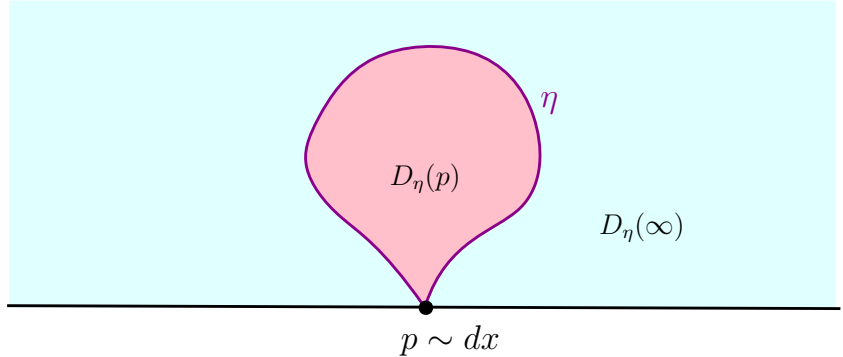}
	\centering
	\caption{Illustration of the welding equation (\ref{eqn:ue in main}) in Theorem \ref{Main welding theorem}: first sample a root point $p$ according to Lebesgue measure $dx$ on $\mathbb R$, then sample $(\phi,\eta)$ according to the product measure $\LF^{(\beta_{2W+2},p)}_{\mathbb H}(d\phi)\times\sle_{\kappa,p}^{\bub}(W-2)(d\eta)$. The resulting quantum surface $(\mathbb H,\phi,\eta,p)/\sim_\gamma$ has the law of $C\int_0^\infty\mathcal{M}^{\disk}_{0,2}(W;\cdot,\ell)\times \qd_{0,1}(\ell)d\ell$ after uniform embedding.}
	\label{fig:main_welding}
\end{figure}
\begin{Th}\label{Main welding theorem}
Fix $\gamma\in (0,2)$. For $W>0$, let $\rho=W-2$ and $\beta_{2W+2}=\gamma-\frac{2W}{\gamma}$. There exists some constant $C\in (0,\infty)$ such that 
\begin{equation}\label{eqn:ue in main}
    \mathbf{m}_{\mathbb H}\ltimes \left(\int_0^\infty\mathcal{M}^{\disk}_{0,2}(W;\cdot,\ell)\times \qd_{0,1}(\ell)d\ell\right) = C\cdot\LF^{(\beta_{2W+2},p)}_{\mathbb H}(d\phi)\times\sle_{\kappa,p}^{\bub}(\rho)(d\eta)dp,
 \end{equation}
 where $\mathbf{m}_{\mathbb H}$ is a Haar measure on $\conf(\mathbb H)$, i.e., the group of conformal automorphisms of $\mathbb H$. Furthermore, there exists some constant $C\in (0,\infty)$ such that 
\begin{equation}\label{eqn:ue in main fixed root}
    \mathbf{m}_{\mathbb H,0}\ltimes\left(\int_0^\infty \mathcal M^{\disk}_{0,2}(W;\cdot,\ell)\times \qd_{0,1}(\ell) d\ell\right) = C\cdot \LF_{\mathbb H}^{(\beta_{2W+2},0)}(d\phi)\times\sle_{\kappa,0}^{\bub}(\rho)(d\eta),
\end{equation}
where $\mathbf{m}_{\mathbb H,0}$ is a Haar measure on $\conf(\mathbb H,0)$, i.e., the group of conformal automorphisms of $\mathbb H$ fixing $0$.
\end{Th} 
The proof of Theorem \ref{Main welding theorem} is presented in Section \ref{sec:proof of main based on conditional welding}. The equation (\ref{eqn:ue in main fixed root}) should be viewed as the disintegration of equation (\ref{eqn:ue in main}) over its boundary root point. Unlike the case of Theorem \ref{Conditional Welding}, where there are two marked points:one in the bulk and one on the boundary, there is only one marked point in curve-decorated quantum surface $\int_0^\infty\mathcal{M}^{\disk}_{0,2}(W;\cdot,\ell)\times \qd_{0,1}(\ell)d\ell$. Therefore, we do not have enough marked points to fix a conformal structure of $\mathbb H$. In this case, the LCFT describes the law of quantum surface $\int_0^\infty\mathcal{M}^{\disk}_{0,2}(W;\cdot,\ell)\times \qd_{0,1}(\ell)d\ell$ after uniform embedding, whereas in Theorem \ref{Conditional Welding}, the LCFT describes the law of the quantum surface (\ref{qs:conditionalwelding}) under a fixed embedding.
\par
Another way of stating Theorem \ref{Main welding theorem} without using uniform embedding is to fix a particular embedding on the right hand side of equations (\ref{eqn:ue in main}) and (\ref{eqn:ue in main fixed root}). For instance, we can first sample $(\phi,\eta)$ from $C\cdot\LF_{\mathbb H}^{(\beta_{2W+2},0)}(d\phi)\times \sle_{\kappa,0}^{\bub}(\rho)(d\eta)$ and then fix the embedding by requiring $\nu_{\phi}(0,1)=\nu_{\phi}(1,\infty) = \nu_{\phi}(\infty,0)$, i.e., the quantum boundary lengths between $0,1$ and $\infty$ are all equal. By doing this, the law of $(D_\eta(0), \phi, 0)$ and $(D_\eta(\infty), \phi, 0^-, 0^+)$ viewed as a pair of marked quantum surfaces is equal to $\int_0^\infty \mathcal M^{\disk}_{0,2}(W;\cdot,\ell)\times \qd_{0,1}(\ell) d\ell.$
\par
As a by-product of the uniform embedding, we also obtain the following decomposition formula (Lemma \ref{rela bet SLE and SLE tilt} and Corollary \ref{Cor:root p equiv}) on the rooted $\sle$ bubble measure $\sle^{\bub}_{\kappa,p}(\rho)$:
\begin{align}\label{eqn:decomp of bubble in the main}
\begin{split}
    &\sle_{\kappa,p}^{\bub}(\rho)(d\eta)= C\cdot\frac{1}{|D_\eta(p)|}\int_{\mathbb H} |q-p|^{W-\frac{2W(W+2)}{\gamma^2}}\left(\Im q\right)^{\frac{W(W+2)}{\gamma^2}-\frac{W}{2}}\sle_{\kappa,p}^{\bub}(\rho)[d\eta|{q\in D_\eta(p)}] d^2q,
\end{split}
\end{align}
where $C\in (0,\infty)$, $|D_\eta(p)|$ is the Euclidean area of $D_\eta(p)$, $\kappa=\gamma^2$, and $\rho = W-2$. The (\ref{eqn:decomp of bubble in the main}) also tells us that 
\begin{equation}\label{eqn:law of bubble surrounds q}
	\sle^{\bub}_{\kappa,p}(\rho)[q\in D_\eta(p)] \propto |q-p|^{W-\frac{2W(W+2)}{\kappa}}(\Im q)^{-\frac{W}{2}+\frac{W(W+2)}{\kappa}}.
\end{equation}
In other words, for fixed $p\in \mathbb R$, the ``probability'' that $\sle_{\kappa,p}^{\bub}(\rho)$ surrounds $q$ is proportional to $|q-p|^{W-\frac{2W(W+2)}{\kappa}}(\Im q)^{-\frac{W}{2}+\frac{W(W+2)}{\kappa}}$. As we will see in Section \ref{sec:proof of main based on conditional welding}, it is the Haar measure together with ``uniform symmetries'' of the underlying Liouville field, or more concretely, the conformal covariance property of the LCFT, that give us equation (\ref{eqn:law of bubble surrounds q}). The equation (\ref{eqn:decomp of bubble in the main}) provides a concrete relationship between the ordinary infinite bubble measure $\sle^{\bub}_{\kappa,p}(\rho)$ and the probability measure $\sle^{\bub}_{\kappa,p}(\rho)[d\eta|i\in D_\eta(p)]$ after conditioning, which builds the bridge between Theorem \ref{Main welding theorem} and Theorem \ref{Conditional Welding}.
\begin{Rem}[Scaling limits of random planar maps decorated by self-avoiding bubbles]
	Motivated by \cite[Theorem 1.2]{SLEloopviawelding}, we conjecture that the scaling limit of the quadrangulated disk decorated by the self-avoiding discrete bubble converges in law to one-pointed quantum disk decorated by $\sle$ bubble, i.e., the $\int_0^\infty\mathcal{M}^{\disk}_{0,2}(2;\cdot,\ell)\times \qd_{0,1}(\ell)d\ell$ in Theorem \ref{Main welding theorem}, for $\kappa=\gamma^2=\frac{8}{3}$ in the so-called \emph{Gromov-Hausdorff-Prokhorov-Uniform topology} (GHPU topology). For the precise definition of GHPU topology, see \cite[Subsection 2.6]{SLEloopviawelding}. The precise conjectures regarding the scaling limit of bubble-decorated quadrangulated disks will be presented in Subsection \ref{sec:conj on scaling limits}.
\end{Rem}
\subsection{SLE bubble zippers with a generic insertion and applications}\label{sec:generalization and application}
\subsubsection{Moments of the conformal radius of $\sle_\kappa(\rho)$ bubbles}
Next, we consider the generalization of Theorem \ref{Conditional Welding} to the case when the bulk insertion of $\qd_{1,1}$ has generic weight. To generalize Theorem \ref{Conditional Welding}, we first define the twisted $\sle_\kappa(\rho)$ bubble measure on $\bubble_{\mathbb H}(0,i)$ corresponding to weight-$\alpha$ bulk insertions of the quantum disk. Given $\eta\in \bubble_{\mathbb H}(0,i)$, let $\psi_\eta:\mathbb H\to D_\eta(0)$ be the unique conformal map fixing $i$ and $0$. Let $\mathbf m$ denote the probability law of $\sle_{\kappa,0}^{\bub}(\rho)[d\eta|{i\in D_\eta(0)}]$ as in Theorem $\ref{Conditional Welding}$ and $\Delta_\alpha:=\frac{\alpha}{2}(Q-\frac{\alpha}{2})$ is known as the \textit{scaling dimension}. Define $\mathbf m_\alpha$ to be the \emph{non-probability} measure on $\bubble_{\mathbb H}(0,i)$ such that 
\begin{equation}\label{twisted SLE bubble: RN derivative}
	\frac{d\mathbf m_\alpha}{d\mathbf m}(\eta)=|\psi'_\eta(i)|^{2\Delta_\alpha-2}.
\end{equation}
Fix $p\in \mathbb R, q\in \mathbb H$ and let $\LF_{\mathbb H}^{(\beta,p),(\alpha,q)}(\ell)$ be the disintegration of $\LF_{\mathbb H}^{(\beta,p),(\alpha,q)}$ over its total boundary length, i.e., $\LF_{\mathbb H}^{(\beta,p),(\alpha,q)} = \int_0^\infty \LF_{\mathbb H}^{(\beta,p),(\alpha,q)}(\ell)d\ell$. Like before, the measure $\LF_{\mathbb H}^{(\beta,p),(\alpha,q)}(\ell)$ represents the Liouville field $\LF_{\mathbb H}^{(\beta,p),(\alpha,q)}$ restricted to having total boundary length $\ell$. The quantum surface $\qd_{1,1}(\alpha,\gamma)$ is the simple generalization of $\qd_{1,1}$ and has the LCFT description of $\LF_{\mathbb H}^{(\alpha,i),(\gamma,0)}$ under the particular embedding $(\mathbb H,\phi,0,i)$; see Definition \ref{def:general qd}. Again, $\qd_{1,1}(\alpha,\gamma;\ell)$ is the disintegration of $\qd_{1,1}(\alpha,\gamma)$ over its total boundary length, i.e., $\qd_{1,1}(\alpha,\gamma) = \int_0^{\infty}\qd_{1,1}(\alpha,\gamma;\ell)d\ell$. We generalize Theorem \ref{Conditional Welding} to Theorem \ref{welding of generalized weight} in order to compute the moments of conformal radius of the $\sle_\kappa(\rho)$ bubble conditioning on surrounding $i$.
\begin{Th}\label{welding of generalized weight}
For $\alpha\in\mathbb R$ and $W>0$, there exists some constant $C_W\in (0,\infty)$ such that the following holds: Suppose $(\phi,\eta)$ is sampled from $\LF_{\mathbb H}^{(\beta_{2W+2},0),(\alpha,i)}(1)\times \mathbf m_\alpha$, then the law of $(D_\eta(0),\phi,i,0)$ and $(D_\eta(\infty),\phi,0^-,0^+)$ viewed as a pair of marked quantum surfaces is given by $C_{W}\cdot\int_0^\infty \qd_{1,1}(\alpha,\gamma;\ell)\times \mathcal M^{\disk}_{0,2}(W;1,\ell)d\ell$. In other words, 
 \begin{equation}\label{general welding equation}
 	\LF_{\mathbb H}^{(\beta_{2W+2},0),(\alpha,i)}(1)\times \mathbf m_\alpha=C_{W}\cdot\int_0^\infty \qd_{1,1}(\alpha,\gamma;\ell)\times \mathcal M^{\disk}_{0,2}(W;1,\ell)d\ell.
 \end{equation}	
\end{Th}
For technical convenience, we restrict the total boundary length of the curve-decorated quantum surface (\ref{general welding equation}) to $1$. 
For simply connected domain $D_\eta(0)$, $\psi_\eta^{-1}$ is the conformal map from $D_\eta(0)$ to $\mathbb H$ that fixes $0$ and $1$. Let $g(z)=i\frac{z-i}{z+i}$ be the uniformizing map from $\mathbb H$ to $\mathbb D$ and let $\varphi_\eta:D_\eta(0)\mapsto \mathbb D$ be such that $\varphi_\eta := g\circ \psi_\eta^{-1}$. Notice that $\varphi_\eta$ maps $i$ to $0$ and $0$ to $1$ respectively. Under our setups, the conformal radius of $D_\eta(0)$ viewed from $i$, denoted by $\Rad(D_\eta(0),i)$, is defined as $\frac{1}{|\varphi'_\eta(i)|}$, i.e., 
\begin{equation}\label{def:cr}
    \Rad(D_\eta(0),i):=\frac{1}{|\varphi'_\eta(i)|}.
\end{equation}
Notice that our definition of conformal radius (\ref{def:cr}) differs slightly with the classical literature of complex analysis, where the conformal map is chosen so that it maps $i$ to $0$ and its derivative at $i$ is in $\mathbb R_{+}$. By simple computation, 
\begin{align}
\begin{split}
    \varphi_\eta'(i)=\left[g\circ \psi_\eta^{-1}\right]'(i) = g'(\psi_{\eta}^{-1}(i))\cdot(\psi_\eta^{-1})'(i) = g'(i)\cdot\frac{1}{\psi_\eta'(i)}.
\end{split}
\end{align}
Therefore, 
\begin{equation}
    \Rad(D_\eta(0),i)=\frac{1}{|\varphi'_\eta(i)|} = \frac{|\psi_\eta'(i)|}{|g'(i)|} = 2 |\psi_\eta'(i)|.
\end{equation}
When $\eta$ is sampled from $\sle_{\kappa,0}^{\bub}(\rho)[d\eta|i\in D_\eta(0)]$, we are interested in the moments of conformal radius $\Rad(D_\eta(0),i)$. Specifically, we want to compute $\mathbb E\left[\Rad(D_\eta(0),i)^{2\Delta_\alpha-2} \right]$, which is the same as $2^{2\Delta_\alpha-2}\cdot\mathbb E\left[|\psi'_\eta(i)|^{2\Delta_\alpha-2} \right]$. To clear up additional constant in the conformal welding equation (\ref{general welding equation}), we further define the \textit{renormalized moments of conformal radius} $\C(\alpha,W)$ to be
\begin{equation}\label{Def:renormalized CR intro}
	\C(\alpha,W):=\frac{\Rad(D_\eta(0),i)}{2^{2\Delta_\alpha-2}\cdot C_W} =\frac{\mathbb E\left[|\psi'_\eta(i)|^{2\Delta_\alpha-2} \right]}{C_W}. 
\end{equation}
Throughout this paper, with a slight abuse of notation, when we talk about ``the conformal radius of $\sle_{\kappa,0}^{\bub}(\rho)[d\eta|{i\in D_\eta(0)}]$'', we really mean the conformal radius of the random simply connected domain $D_\eta(0)$ viewed from $i$ when $\eta$ is sampled from probability measure $\sle_{\kappa,0}^{\bub}(\rho)[d\eta|{i\in D_\eta(0)}].$
\begin{Prop}[Moments of conformal radius of $\sle_\kappa$ bubbles conditioning on surrounding $i$]\label{main prop: conformal radius}
	Fix $\kappa\in (0,4)$, $W=2,\rho=0$ and $\frac{\gamma}{2}<\alpha<Q+\frac{2}{\gamma}$. Suppose $\eta$ is sampled from $\sle_{\kappa,0}^{\bub}[d\eta|i\in D_\eta(0)]$, then we have 
	\begin{equation}
		\mathbb E\left[|\psi'_\eta(i)|^{2\Delta_\alpha-2} \right]=\frac{\Gamma(\frac{2\alpha}{\gamma})\Gamma(\frac{8}{\kappa}-\frac{2\alpha}{\gamma}+1)}{\Gamma(\frac{8}{\kappa}-1)}.
	\end{equation}
 Consequently, 
 \begin{equation}
     \mathbb E\left[\Rad(D_\eta(0),i)^{2\Delta_\alpha-2} \right] = 2^{2\Delta_\alpha-2}\cdot \frac{\Gamma(\frac{2\alpha}{\gamma})\Gamma(\frac{8}{\kappa}-\frac{2\alpha}{\gamma}+1)}{\Gamma(\frac{8}{\kappa}-1)}.
 \end{equation}
\end{Prop} 
Moments of the conformal radius of the general $\sle^{\bub}_{\kappa,0}(\rho)[d\eta|i\in D_\eta(0)]$ bubbles are computed in Proposition \ref{conformal radius: general case}. The key ingradients of the computation are the functin $\overline{G}(\alpha,\beta)$ and the \textit{Liouville reflection coefficient} $R(\beta,\mu_1,\mu_2)$ in \cite{RZ22,InteofSLE}, which describe the quantum boundary length laws of the two-pointed disk and the disk with one bulk and one boundary marked points, respectively. 
\subsubsection{The bulk-boundary correlation function in the LCFT}\label{subsubsec: two-pointed correlation function}
 As an another important application of Theorem \ref{welding of generalized weight}, we derived a formula for the bulk-boundary correlation function in the LCFT within probabilistic frameworks. In theoretical physics, the LCFT is defined by the formal path integral. The most basic observable of Liouville theory is the correlation function with $N$ bulk marked points $z_i\in \mathbb H$ with weights $\alpha_i\in \mathbb R$ and $M$ boundary marked points $s_j\in \mathbb R$ with weights $\beta_j$. Precisely, for bulk insertions $(z_i)_{1\leq i\leq N}$ with weights $(\alpha_i)_{1\leq i\leq N}$ and boundary insertions $(s_j)_{1\leq j\leq M}$ with weights $(\beta_j)_{1\leq j\leq M}$, the correlation function in the LCFT at these points is defined using the following formal path integral: 
\begin{equation}\label{eqn: path integral}
	\left\langle\prod_{i=1}^N e^{\alpha_i\phi(z_i)}\prod_{j=1}^M e^{\frac{\beta_j}{2}\phi(s_j)}\right\rangle_{\mu,\mu_\partial}=\int_{X:\mathbb H\to \mathbb R} DX\prod_{i=1}^N e^{\alpha_iX(z_i)}\prod_{j=1}^M e^{\frac{\beta_j}{2}X(s_j)}e^{-S_{\mu,\mu_\partial}^L(X)}, 
\end{equation}
where $DX$ is the formal uniform measure on infinite dimensional function space and $S_{\mu,\mu_\partial}^L(X)$ is the \textit{Liouville action functional} given by 
\begin{equation}\label{eqn: action functional}
	S_{\mu,\mu_\partial}^L(X):=\frac{1}{4\pi}\int_{\mathbb H} \left(|\nabla_g X|^2+Q R_g X + 4\pi\mu e^{\gamma X} \right)d\lambda_g+\frac{1}{2\pi}\int_{\mathbb R}\left(QK_g X+2\pi\mu_\partial e^{\frac{\gamma}{2}X}\right)d\lambda_{\partial g}. 
\end{equation}
For background Riemannian metric $g$ on $\mathbb H$, $\nabla_g, R_g, K_g,d\lambda_g,d\lambda_{\partial g}$ stand for the gradient, Ricci curvature, Geodesic curvature, volume form and line segment respectively. The subscripts $\mu,\mu_\partial$ emphasize the fact that both $\mu$ and $\mu_\partial$ are positive.
\par
As a conformal field theory, the bulk correlation function $\left\langle e^{\alpha\phi(z)} \right\rangle_{\mu,\mu_\partial}$ of LCFT takes the following form:
\begin{equation}
	\left\langle e^{\alpha\phi(z)} \right\rangle_{\mu,\mu_\partial}=\frac{U(\alpha)}{|\Im z|^{2\Delta_\alpha}}\qquad\text{for $z\in \mathbb H$},
\end{equation}
where $U(\alpha)$ is known as the \textit{structure constant} and $\Delta_\alpha=\frac{\alpha}{2}(Q-\frac{\alpha}{2})$ is called the \textit{scaling dimension} as mentioned before. In \cite{FZZ00}, the following formula for $U(\alpha)$ was proposed:
\begin{equation}\label{eqn:UFZZ}
	U_{\FZZ}(\alpha):=\frac{4}{\gamma} 2^{-\frac{\alpha^2}{2}}\left(\frac{\pi\mu}{2^{\gamma\alpha}}\frac{\Gamma\left(\frac{\gamma^2}{4}\right)}{\Gamma(1-\frac{\gamma^2}{4})}\right)^{\frac{Q-\alpha}{\gamma}}\Gamma\left(\frac{\gamma\alpha}{2}-\frac{\gamma^2}{4}\right)\Gamma\left(\frac{2\alpha}{\gamma}-\frac{4}{\gamma^2}-1\right)\cos\left((\alpha-Q)\pi s \right),
\end{equation}
where the parameter $s$ is defined through the following ratio of cosmological constants $\frac{\mu_\partial}{\sqrt{\mu}}$: 
\begin{equation*}
	\cos\frac{\pi\gamma s}{2}=\frac{\mu_\partial}{\sqrt{\mu}}\sqrt{\sin\frac{\pi\gamma^2}{4}},\qquad\text{with}\ \begin{cases}
		s\in [0,\frac{1}{\gamma}) \qquad &\text{if}\qquad \frac{\mu_\partial ^2}{\mu}\sin\frac{\pi\gamma^2}{4}\leq 1,\\
		s\in i[0,+\infty)\qquad & \text{if}\qquad\frac{\mu_\partial^2}{\mu}\sin\frac{\pi\gamma^2}{4}\geq 1.
	\end{cases}
\end{equation*}
In \cite{FZZ}, the (\ref{eqn:UFZZ}) was proved within rigorous probability theory frameworks. From now on, for measure $M$ on the space of distributions, let $M[f]:=\int f(\phi) M(d\phi)$. For $\gamma\in (0,2)$ and $\mu,\mu_\partial >0$, let
\begin{equation}
	\left\langle e^{\alpha\phi(z)}\right\rangle_{\mu,\mu_\partial}:=\LF_{\mathbb H}^{(\alpha,z)}\left[e^{-\mu\mu_{\phi}(\mathbb H)-\mu_\partial \nu_{\phi}(\mathbb R)}-1\right],\qquad\text{for $z\in \mathbb H$,}
\end{equation}
where 
\begin{equation*}
	\mu_{\phi}(\mathbb H)=\lim_{\varepsilon\to 0}\varepsilon^{\frac{\gamma^2}{2}}\int_{\mathbb H}e^{\gamma\phi_\varepsilon(z)}d^2z\qquad\text{and}\qquad\nu_\phi(\mathbb R)=\lim_{\varepsilon\to 0}\varepsilon^{\frac{\gamma^2}{4}}\int_{\mathbb R}e^{\frac{\gamma}{2}\phi_\varepsilon(z)}dz.
\end{equation*}
Since $|\Im z|^{2\Delta_\alpha}\left\langle e^{\alpha\phi(z)}\right\rangle_{\mu,\mu_\partial}$ does not depend on $z\in \mathbb H$, define $U(\alpha):=\left\langle e^{\alpha\phi(i)}\right\rangle_{\mu,\mu_\partial}$.
\begin{Th}[{\cite[Theorem 1.1]{FZZ}}]
	For $\gamma\in (0,2),\alpha\in (\frac{2}{\gamma},Q)$ and $\mu,\mu_\partial >0$, we have $U(\alpha)=U_{\FZZ}(\alpha)$.
\end{Th}
The above theorem is the first step towards rigorously solving the boundary LCFT. In this paper, we consider the bulk-boundary correlation in the LCFT. For $z\in \mathbb H$ and $s\in \mathbb R$, by the conformal invariance property, the bulk-boundary correlation function in the LCFT takes the following form: 
\begin{equation}
	\left\langle e^{\alpha\phi(z)}e^{\frac{\beta}{2}\phi(s)}\right\rangle_{\mu,\mu_\partial}=\frac{G_{\mu,\mu_\partial}(\alpha,\beta)}{|\Im z|^{2\Delta_\alpha-\Delta_\beta}|z-s|^{2\Delta_\beta}}.
\end{equation}
Within probabilistic frameworks, define
\begin{equation}
	\left\langle e^{\alpha\phi(z)}e^{\frac{\beta}{2}\phi(s)}\right\rangle_{\mu,\mu_\partial}:=\LF_{\mathbb H}^{(\alpha,z),(\beta,s)}\left[e^{-\mu\mu_{\phi}(\mathbb H)-\mu_\partial \nu_{\phi}(\mathbb R)} \right]
\end{equation}
and 
\begin{equation}
	G_{\mu,\mu_\partial}(\alpha,\beta):= \LF_{\mathbb H}^{(\alpha,i),(\beta,0)}\left[e^{-\mu\mu_{\phi}(\mathbb H)-\mu_\partial \nu_{\phi}(\mathbb R)} \right].
\end{equation}
Notice that $|\Im z|^{2\Delta_\alpha-\Delta_\beta}|z-s|^{2\Delta_\beta}\left\langle e^{\alpha\phi(z)}e^{\frac{\beta}{2}\phi(s)}\right\rangle_{\mu,\mu_\partial}$ does not depend on $z$ and $s$ and the function $G_{\mu,\mu_\partial}(\alpha,\beta)$ is called the \textit{structure constant} in the boundary Liouville theory.
\par
So far in the literature, all the exact formulas in LCFT except FZZ (\ref{eqn:UFZZ}) have been derived by BPZ equations and the corresponding operator product expansion \cite{BPZ84}, including \cite{KRV20} for the DOZZ formula and \cite{Rem20,RZ20,RZ22} for different cases of boundary Liouville correlation functions with $\mu=0$ and $\mu_\partial>0$; see also discussions in \cite[Section 1.1]{FZZ}. In this paper, from Theorem \ref{welding of generalized weight}, we derive a formula linking the bulk-boundary correlation function to the joint law of left \& right quantum boundary lengths and quantum area of $\mathcal M^{\disk}_{0,2}(W)$ when $0<W<\frac{\gamma^2}{2}$.
\begin{Prop}[Bulk-boundary correlation function in the LCFT]\label{Prop:bulk-boundary}
	Fix $\gamma\in (0,2)$ and $\mu,\mu_\partial>0$. When $\beta_{2W+2}$ and $\alpha$ satisfy $0<\beta_{2W+2}<\gamma$ and $Q-\frac{\beta_{2W+2}}{2}<\alpha<Q$, we have 
	\begin{align}\label{eqn:two-pointed corr with joint law of weight W disk}
	\begin{split}
	G_{\mu,\mu_\partial}(\alpha,\beta_{2W+2}) &=\C(\alpha,W)^{-1}\frac{2}{\gamma}2^{-\frac{\alpha^2}{2}}\overline U_0(\alpha)\frac{2}{\Gamma(\frac{2}{\gamma}(Q-\alpha))}\left(\frac{1}{2}\sqrt{\frac{\mu}{\sin(\pi\gamma^2/4)}}\right)^{\frac{2}{\gamma}(Q-\alpha)}\times\\
 &\mathcal M_{0,2}^{\disk}(W)\left[e^{-\mu_\partial R_W-\mu A_W}\cdot K_{\frac{2}{\gamma}(Q-\alpha)}\left(L_W\sqrt{\frac{\mu}{\sin(\pi\gamma^2/4)}} \right) \right],
	\end{split}
\end{align}
where $\beta_{2W+2}=\gamma-\frac{2W}{\gamma}$, $L_W,R_W$ and $A_W$ denote the left, right quantum boundary length and quantum area of $\mathcal M^{\disk}_{0,2}(W)$ respectively. The $\C(\alpha,W)$ is the renormalized moments of the conformal radius defined in (\ref{Def:renormalized CR intro}) and takes an explicit formula (\ref{formula of C(alpha, W)}). The $\overline U_0(\alpha)$ is defined in Theorem \ref{Th: U0alpha} and takes the following explicit formula: 
\begin{align}
		\overline U_0(\alpha)=\left(\frac{2^{-\frac{\gamma\alpha}{2}}2\pi}{\Gamma(1-\frac{\gamma^2}{4})}\right)^{\frac{2}{\gamma}(Q-\alpha)}\Gamma\left(\frac{\gamma\alpha}{2}-\frac{\gamma^2}{4}\right)\qquad\text{for all $\alpha>\frac{\gamma}{2}$.}
\end{align}
The $K_\nu(x)$ is the modified Bessel function of second kind. Precisely, 
\begin{equation*}
	K_\nu(x):=\int_0^\infty e^{-x\cosh t}\cosh(\nu t)dt\qquad\text{for $x>0$ and $\nu\in \mathbb R$}.
\end{equation*}
\end{Prop}
The condition $0<\beta_{2W+2}<\gamma$ in Proposition \ref{Prop:bulk-boundary} is equivalent to $0<W<\frac{\gamma^2}{2}$, i.e., the case when the two-pointed quantum disk is thin.  By \cite[(3.5),(3.6),(3.7)]{HRV18}, the \textit{Seiberg bounds} correspond to 
        \begin{equation}\label{eqn:Seiberg bound}
            \alpha<Q, \quad \beta_{2W+2}<Q\quad\text{and}\quad\alpha+\frac{1}{2}\beta_{2W+2}>Q,
        \end{equation}
        which hold if and only if 
        \begin{equation}
           0< G_{\mu,\mu_\partial}(\alpha,\beta_{2W+2})= \LF_{\mathbb H}^{(\alpha,i),(\beta_{2W+2},0)}\left[e^{-\mu\mu_{\phi}(\mathbb H)-\mu_\partial \nu_{\phi}(\mathbb R)} \right] <\infty.
        \end{equation}
        Notice that the range of $\alpha$ and $\beta_{2W+2}$ in Proposition \ref{Prop:bulk-boundary} are strictly contained in (\ref{eqn:Seiberg bound}), and therefore the $G_{\mu,\mu_\partial}(\alpha,\beta_{2W+2})$ in (\ref{eqn:two-pointed corr with joint law of weight W disk}) is nontrivial.
\begin{Rem}
An explicit formula for the quantity 
    \begin{equation}
        \mathcal M_{0,2}^{\disk}(W)\left[e^{-\mu_\partial R_W-\mu A_W}\cdot K_{\frac{2}{\gamma}(Q-\alpha)}\left(L_W\sqrt{\frac{\mu}{\sin(\pi\gamma^2/4)}} \right) \right]
    \end{equation}
    in (\ref{eqn:two-pointed corr with joint law of weight W disk}) is derived in the concurrent work of \cite[Lemma 4.4]{ARSZ23}. Combined with Proposition~\ref{Prop:bulk-boundary}, this verifies the formula for $G_{\mu,\mu_\partial}(\alpha,\beta) $  proposed by Hosomichi \cite{KH01} in physics; see \cite[Theorem 1.2]{ARSZ23} for more details.
\end{Rem} 
 \subsection{Acknowledgements}
 This paper is part of the author's Ph.D. thesis written at University of Pennsylvania. The author would like to thank Xin Sun for many helpful discussions. The author also wants to thank Dapeng Zhan for explaining the constructions of $\sle_\kappa(\rho)$ bubbles via radial Bessel processes, and Morris Ang and Zijie Zhuang for the careful reading of the early draft of this paper.
\section{Preliminaries}\label{sec:preliminaries}
\subsection{Notations and basic setups}
Throughout this paper, $\gamma\in (0,2)$ is the LQG coupling constant. Moreover, 
\begin{equation*}
    Q=\frac{2}{\gamma}+\frac{\gamma}{2}\qquad\text{and}\qquad \kappa = \gamma^2.
\end{equation*}
For weight $W\in \mathbb R$, $\beta_W$ is always a function of $W$ with $\beta_W = Q+\frac{\gamma}{2}-\frac{W}{\gamma}=\gamma+\frac{2-W}{\gamma}$. We will work with planar domains in $\mathbb C$ including the upper half plane $\mathbb H = \lbrace z\in \mathbb C:\Im(z)>0\rbrace$, horizontal strip $\mathcal{S}=\mathbb R\times (0,\pi)$ and unit disk $\mathbb D = \lbrace z\in \mathbb C: |z|\leq 1\rbrace$. For a domain $D\subset\mathbb{C}$, we denote its boundary by $\partial D$. For instance, $\partial\mathbb H = \overline{\mathbb R}=\mathbb R\cup\lbrace \infty\rbrace$, $\partial\mathcal{S} = \lbrace z\in \mathbb C: \Im(z) = 0 \ \text{or} \ 1 \rbrace\cup\lbrace \pm\infty\rbrace$ and $\partial\mathbb D=\lbrace z:|z|=1\rbrace$.
\par
We will frequently consider non-probability measure and extend the terminology of probability theory to this setting. More specifically, suppose $M$ is a measure on a measurable space $(\Omega,\mathcal F)$ with $M(\Omega)$ not necessarily $1$ and $X$ is a $\mathcal F$-measurable function, then we say $(\Omega,\mathcal F)$ is a sample space and $X$ is a random variable. We call the pushforward $M_X = X_*M$ the \textit{law} of $X$ and we say that $X$ is sampled from $M_X$. We also write 
\begin{equation*}
	M_X[f]:=\int f(x) M_X(dx).
\end{equation*}
Weighting the law of $X$ by $f(X)$ corresponds to working with measure $d\widetilde{M_X}$ with Randon-Nikodym derivative $\frac{d\widetilde{M_X}}{dM_X}=f$. For some event $E\in \mathcal F$ with $0<M[E]<\infty$, let $M[\cdot|E]$ denote the probability measure $\frac{M[E\cap \cdot]}{M[E]}$ over the measure space $(E,\mathcal F_E)$ with $\mathcal F_E = \lbrace A\cap E:A\in \mathcal F \rbrace$. For a finite positive measure $M$, we denote its total mass by $|M|$ and let $M^{\#} = |M|^{-1} M$ denote the corresponding probability measure.
\par
Let $g$ be a smooth metric on $\mathbb H$ such that the metric completion of $(\mathbb H, g)$ is a compact Riemannian manifold. Let $H^1(\mathbb H,g)$ be the standard Sobolev space with norm defined by 
\begin{equation*}
	|h|_{H^1(\mathbb H,g)}:=\left(\int_{\mathbb H}|\nabla h(z)|^2+|h(z)|^2 g(z) d^2z \right)^{1/2}.
\end{equation*}
Let $H^{-1}(\mathbb H,g)$ be its dual space, which is defined as the completion of the set of smooth functions on $\mathbb H$ with respect to the following norm:
\begin{equation*}
	|f|_{H^{-1}(\mathbb H,g)}:=\sup_{h\in H^1(\mathbb H,g),|h|_{H^1(\mathbb H,g)\leq 1}}\left|\int_{\mathbb H} f(z)h(z)g(z)d^2z \right|.
\end{equation*}
Here we remark that $H^{-1}(\mathbb H)$ is a polish space and its topology does not depend on the choice of $g$. Throughout this paper, all the random functions considered are in $H^{-1}(\mathbb H)$. 
\subsection{$\sle_{\kappa}(\rho)$ bubble measures}\label{Sec: SLE bubble}
In this section, we review the rooted $\sle_{\kappa}(\rho)$ bubble measure constructed by Zhan in \cite{SLEbubble}. It was constructed on $\mathbb{H}$ for all $\kappa>0$ and $\rho>-2$. Throughout this paper, we only consider the case when $0<\kappa<4$ and $\rho>(\frac{\kappa}{2}-4)\vee(-2)$. In this case, the law of the bubble is a $\sigma$-finite infinite measure and satisfies conformal covariance property (\cite[Theorem $3.16$]{SLEbubble}). As mentioned before, an $\sle_{\kappa}(\rho)$ bubble $\eta$ is characterized by the following \textit{Domain Markov Property}: let $\tau$ be a positive stopping time for $\eta$, then conditioning on the part of $\eta$ before $\tau$ and the event that $\eta$ is not complete at the time $\tau$, the part of $\eta$ after $\tau$ is an $\sle_{\kappa}(\rho)$ curve from $\eta(\tau)$ to the root of $\eta$ in a connected component of $\mathbb{H}\backslash \eta[0,\tau]$. To proceed, we first review the chordal $\sle_{\kappa}(\underline{\rho})$ process on $\mathbb H$.
\subsubsection{Chordal $\sle_{\kappa}(\underline{\rho})$ processes}
In this subsection, we review the basic construction of chordal $\sle_{\kappa}(\underline{\rho})$ process. First, we introduce some notations and terminologies. Let $(E,d_E)$ be a metric space and let $C([0,\widehat{T}],E)$ be the space of continuous functions from $[0,\widehat{T})$ to $E$. Let 
\begin{equation*}
    \Sigma^E = \bigcup_{0<\widehat{T}\leq \infty} C([0,\widehat{T}], E).
\end{equation*}
For each $f\in \Sigma^E$, the lifetime $\widehat{T_f}$ of $f$ is the extended number in $(0,\infty]$ such that $[0,\widehat{T_f})$ is the domain of $f$. Let $\mathbb H=\lbrace z\in \mathbb C: \Im z>0\rbrace$ be the open upper half plane. A set $K\subset\mathbb H$ is called an $\mathbb H$-hull if $K$ is bounded and $\mathbb H\backslash K$ is a simply connected domain. For each $\mathbb H$-hull $K$, there is a unique conformal map $g_K$ from $\mathbb H\backslash K$ onto $\mathbb H$ such that $g_K(z)-z = O(1/z)$ as $z\to \infty$. The number hcap$(K):=\lim_{z\to \infty}z(g_K(z)-z)$ is called $\mathbb H$-capacity of $K$, which satisfies hcap$(\emptyset)=0$ and hcap$(K)>0$ if $K\neq\emptyset$. Let 
\begin{equation}
    \rad_\omega(K):=\sup\left\lbrace |z-\omega|:z\in K\cup\lbrace \omega\rbrace \right\rbrace
\end{equation}
for $\omega\in \mathbb C$ and $K\subset\mathbb C$. For $W\in \Sigma^{\mathbb R}$, the chordal Loewner equation driven by $W$ is the following differential equation in $\mathbb C$:
\begin{equation*}
    \partial_t g_t(z) = \frac{2}{g_t(z)-W_t}
\end{equation*}
with $0\leq t<\widehat{T_W}$ and $g_0(z)=z$. For each $z\in \mathbb C$, let $\tau_z^*$ be the biggest extended number in $[0,\widehat{T_W}]$ such that the solution $t\mapsto g_t(z)$ exists on $[0,\tau^*_z)$. For $0\leq t< \widehat{T_W}$, let $K_t = \lbrace z\in \mathbb H: \tau_z^*\leq t\rbrace$ and $H_t = \mathbb H\backslash K_t$. It turns out that each $K_t$ is an $\mathbb H$-hull with hcap$(K_t)=2t$ and $g_t = g_{K_t}$. We call $g_t$ and $K_t$ the chordal Loewner maps and hulls, respectively. 
\par
Now we review the definition of multi-force-point $\sle_{\kappa}(\underline{\rho})$ process. Here, all the force points lie on the boundary. Let $\kappa>0$ and $\underline{\rho} = (\rho_1,\ldots,\rho_m)\in \mathbb R^m$. Let $\omega\in \mathbb R$ and $v_1,\ldots,v_m$ be such that 
\begin{equation}
    \sum_{j:v_j =\omega^+} \rho_j > -2 \qquad\text{and}\qquad \sum_{j:v_j = \omega^{-}}\rho_j > -2.
\end{equation}
Consider the following system of SDE:
\begin{align}\label{SLE_kappa(rho) SDE}
\begin{split}
    dW_t &= \sum_{j=1}^m \mathbbm{1}_{\lbrace W_t\neq V_t^j\rbrace} \frac{\rho_j}{W_t-V_t^j}dt + \sqrt{\kappa}dB_t,\qquad W_0 = \omega;\\
    dV_t^j &= \mathbbm{1}_{\lbrace W_t\neq V_t^j\rbrace}\frac{2}{V_t^j-W_t}dt,\qquad V_0^j = v_j, \qquad 1\leq j\leq m.
\end{split}
\end{align}
If some $v_j=\infty$, then $V_t^j$ is $\infty$, and $\frac{1}{V_t^j-W_t^j}$ is $0$. It is known that a weak solution of the system (\ref{SLE_kappa(rho) SDE}), in the integral sense, exists and is unique in law, and the $W_t$ in the solution a.s. generates a Loewner curve $\eta$, which we call $\sle_{\kappa}(\underline{\rho})$ curve starts from $\omega$ with force points $\underline{v} = (v_1,\ldots,v_m)$. The $V_t^j$ is called the \textit{force point process} started from $v_j$. 
\subsubsection{$\sle_{\kappa}(\rho)$ bubbles as the weak limit of chordal $\sle_{\kappa}(\rho)$}\label{subsec: bubble as limit}
In this section, we review the main constructions of rooted $\sle_{\kappa}(\rho)$ measures in \cite{SLEbubble}. To do this, we first introduce some basic notations and terminologies. Let $f\in \Sigma^E$. For a continuous and strictly increasing function $\theta$ on $[0,\widehat{T_f})$ with $\theta(0)=0$, the function $g:=f\circ\theta^{-1}\in \Sigma^E$ is called the time-change of $f$ via $\theta$, and we write $f\sim g$. Let $\widetilde{\Sigma^E}:=\Sigma^E/\sim$ and an element of $\widetilde{\Sigma^E}$, denoted by $[f]$, where $f\in \Sigma^E$, is called an MTC (module time-changes) function or curve. Throughout this paper, all the curves considered are MTC curve. Therefore, we will simply write $f$ instead of $[f]$ without confusion. The $\widetilde{\Sigma^E}$ is a metric space with the distance defined by 
\begin{equation}\label{metric on curve}
    d_{\widetilde{\Sigma^E}}:=\inf\left\lbrace \sup\lbrace d_E(f'(t),g'(t)):0\leq t < \widehat{T_{f'}}\rbrace: f'\in [f],g'\in [g], \widehat{T_{f'}}=\widehat{T_{g'}}  \right\rbrace.
\end{equation}
An element $f\in \Sigma^E$ is called a rooted loop if 
\begin{equation*}
    \lim_{t\to \widehat{T_f}}f(t) = f(0),
\end{equation*}
and $f(0)$ is called its root. If $f\in\Sigma^E$ is called a rooted loop, then $[f]\in\widetilde{\Sigma}^E$ is called a rooted MTC loop. Notice that all the elements in $\bubble_{\mathbb H}(p)$ are MTC loops.
\par
By \cite{SLEbubble}, the rooted $\sle_{\kappa}(\rho)$ bubble is constructed as the weak limit of chordal $\sle_{\kappa}(\rho)$ measures after rescaling. We use $\xrightarrow{w}$ to denote the \textit{weak convergence}. Recall that for bounded measures $\mu_n,n\in \mathbb{N}$, and $\mu$ defined on some metric space $E$, $\mu_n \xrightarrow{w} \mu$ if and only if for any $f \in C_b(E,\mathbb R)$, $\mu_n(f)\xrightarrow{w}\mu(f)$. For general simply connected domain $(D,a,b)$, let $\sle^{D}_{\kappa,(a,c)\to b}(\rho)$ denote the chordal $\sle_\kappa(\rho)$ process on $D$ from $a$ to $b$ with force point $c$. In this paper, $c\in \lbrace a^+,a^-\rbrace$ mostly. 
\begin{figure}[H]
	\includegraphics[scale = 0.8]{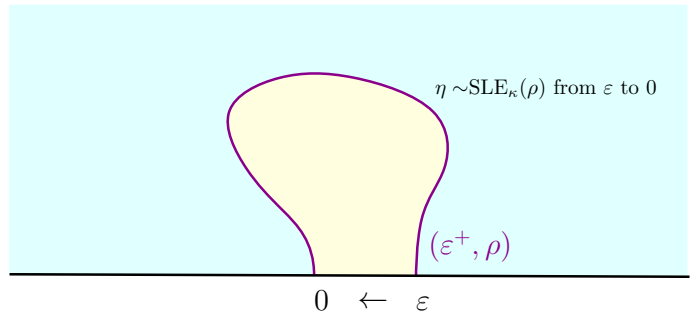}
	\centering
	\caption{Illustration of Theorem \ref{Thm on weak limit}: $\sle^{\bub}_{\kappa,0}(\rho)$ as the weak limit of chordal $\sle_\kappa(\rho)$ with suitable rescaling.}
	\label{fig:SLEbubble_convergence}
\end{figure}
\begin{Th}[{\cite[Theorem $3.20$]{SLEbubble}}]\label{Thm on weak limit}
    Let $0<\kappa<4$ and $\rho>-2$.  There exists a non-zero $\sigma$-finite measure $\sle_{\kappa,0}^{\bub}(\rho)$ on $\bubble_{\mathbb H}(0)$ such that the following holds: For any fixed $S>0$, let $E_S = \lbrace \eta:\rad_0(\eta)>S\rbrace$. Then as $\varepsilon\to 0^+$,
    \begin{equation}\label{equiv: weak limit}
        \varepsilon^{\frac{(\rho+2)(\kappa-8-2\rho)}{2\kappa}}\mathbbm{1}_{E_S}\sle^{\mathbb H}_{\kappa,(\varepsilon;\varepsilon^+)\to 0}(\rho) \xrightarrow{w} \mathbbm{1}_{E_S}\sle_{\kappa,0}^{\bub}(\rho)
    \end{equation}
    in the space $\widetilde{\Sigma}^{\mathbb C}$ with distance defined by (\ref{metric on curve}).
\end{Th} 
\begin{rem}
	Notice that in \cite[Theorem 3.20]{InteofSLE}, the author considered $\sle^{\mathbb H}_{\kappa,(r,r^+)\to -r}(\rho)$ for $r>0$ as the limiting sequence of measures. To get (\ref{equiv: weak limit}), we first apply the shift map $f_r:\mathbb H\to \mathbb H$ such that $f_r(z) = z+r$ then let $\varepsilon = 2r$.
\end{rem}
\begin{Def}[Rooted $\sle_{\kappa}(\rho)$ bubble measures]\label{Def of rooted SLE bubble}
    For $0<\kappa<4$ and $\rho>-2$, we define the weak limit $\sle_{\kappa,0}^{\bub}(\rho)$ in Theorem \ref{Thm on weak limit} as the rooted $\sle_{\kappa}(\rho)$ bubble measure with root $0$. More generally, for any $p\in \partial\mathbb H$, let $f_p:\mathbb H\to \mathbb H$ be such that $f_p(z) = z+p$ and define
    \begin{equation*}
        \sle_{\kappa,p}^{\bub}(\rho):= f_p(\sle_{\kappa,0}^{\bub}(\rho)).
    \end{equation*}
    If $\rho=0$, then we omit the existence of $\rho$ and write $\sle_{\kappa,p}^{\bub}$ for fixed $p\in \partial\mathbb H$.
\end{Def}
\begin{Cor}\label{conditional convergence lemma}
    Let $\widetilde{E_{i,0}}$ be the set of curves on $\overline{\mathbb H}$ starting from some point on $[0,\infty]$, ending at $0$, and surrounding $i$.  Under the same settings as Theorem \ref{Thm on weak limit}, we have
    \begin{equation}\label{condition convergence}
        \sle^{\mathbb H}_{\kappa,(\varepsilon;\varepsilon^+)\to 0}(\rho)[d\eta |\widetilde{E_{i,0}}]\xrightarrow{w} \sle_{\kappa,0}^{\bub}(\rho)[d\eta |\widetilde{E_{i,0}}]\qquad\text{as $\varepsilon\to 0^{+}$}
    \end{equation}
     in the metric space $\widetilde{\Sigma}^{\mathbb C}$ with distance defined by (\ref{metric on curve}).
\end{Cor}
\begin{proof}
    Let $E_1 = \lbrace \eta:\rad_0(\eta)>1\rbrace$. It is clear that $\widetilde{E_{i,0}}\subset E_{1}$. Moreover, $\widetilde{E_{i,0}}$ is open in $\widetilde{\Sigma}^{\mathbb C}$ and $\partial \widetilde{E_{i,0}}$ contains the curves that end at $0$ and pass through $i$. For $0<\delta<1$, let $E_\delta = \lbrace \eta:\rad_0(\eta)\geq\delta\rbrace$ and $\tau_\delta=\inf_{t>0} \lbrace t:\rad_0(\eta[0,t])=\delta\rbrace$ be the first time that $\sle_{\kappa,0}^{\bub}(\rho)$ curve has radius $\delta$ under capacity parametrization. For any $\eta\in E_\delta$, let $\eta_{\delta}=\eta[0,\tau_\delta]$. For any fixed instance of $\eta_{\delta}$, let $\widetilde{\partial E_{i,\eta_{\delta}}}$ be the set of curves from $\eta(\tau_\delta)$ to $0$ on $\mathbb H\backslash \eta_\delta$ that pass through $i$. By Domain Markov Property of $\sle_{\kappa,0}^{\bub}(\rho)$ stated in \cite[Theorem 3.16]{SLEbubble}, we have that 
    \begin{align}
        \begin{split}
            \sle_{\kappa,0}^{\bub}(\rho)[\partial\widetilde{E_{i,0}}] = \int_{E_\delta}\sle_{\kappa,(\eta_{\delta},v(\eta_\delta))\to 0}^{\mathbb H\backslash \eta_{\delta}}(\rho)[\widetilde{\partial E_{i,\eta_{\delta}}}]\sle_{k,0}^{\bub}(\rho)(d\eta_\delta).
        \end{split}
    \end{align}
By \cite[Theorem 3.20]{SLEbubble}, $\sle_{\kappa,0}^{\bub}(\rho)[E_\delta]>0$. Moreover, it is well-known that when $0<\kappa<4$, the probability that chordal $\sle_\kappa(\rho)$ passes through a fixed interior point is zero (see, for instance, \cite{DecompofSLE}). Therefore, $\sle_{\kappa,0}^{\bub}(\rho)[\partial\widetilde{E_{i,0}}]=0$. By (\ref{equiv: weak limit}) and \cite[(F3)]{SLEbubble},
    \begin{equation}
        \varepsilon^{\frac{(\rho+2)(\kappa-8-2\rho)}{2\kappa}}\mathbbm{1}_{E_1}\mathbbm{1}_{\widetilde{E_{i,0}}} \sle^{\mathbb H}_{\kappa,(\varepsilon;\varepsilon^+)\to 0}(\rho) \xrightarrow{w} \mathbbm{1}_{E_1}\mathbbm{1}_{\widetilde{E_{i,0}}} \sle_{\kappa,0}^{\bub}(\rho).
    \end{equation}
    Equivalently,
    \begin{equation}\label{equiv: weak limit2}
        \varepsilon^{\frac{(\rho+2)(\kappa-8-2\rho)}{2\kappa}}\mathbbm{1}_{\widetilde{E_{i,0}}} \sle^{\mathbb H}_{\kappa,(\varepsilon;\varepsilon^+)\to 0}(\rho) \xrightarrow{w} \mathbbm{1}_{\widetilde{E_{i,0}}} \sle_{\kappa,0}^{\bub}(\rho).
    \end{equation}
    In order to prove (\ref{condition convergence}), it remains to show that $0<\sle_{\kappa,0}^{\bub}(\rho)[\widetilde{E_{i,0}}]<\infty$. By \cite[Theorem 3.16]{SLEbubble},
    \begin{equation}
        \sle_{\kappa,0}^{\bub}(\rho)[\widetilde{E_{i,0}}]\leq \sle_{\kappa,0}^{\bub}(\rho)[E_1]<\infty.
    \end{equation}
    For any $\eta\in E_\delta$, let $\eta_{\delta}=\eta[0,\tau_\delta]$. For any fixed instance of $\eta_{\delta}$, let $\widetilde{E_{i,\eta_{\delta}}}$ denote the set of curves on $\mathbb H\backslash \eta_\delta$ from $\eta(\tau_\delta)$ to $0$  that surround $i$. Again, by Domain Markov Property of $\sle_{\kappa,0}^{\bub}(\rho)$ (\cite[Theorem 3.16]{SLEbubble}), 
    \begin{align}
        \begin{split}
            \sle_{\kappa,0}^{\bub}(\rho)[\widetilde{E_{i,0}}] = \int_{E_\delta}\sle_{\kappa,(\eta_{\delta},v(\eta_\delta))\to 0}^{\mathbb H\backslash \eta_{\delta}}(\rho)[\widetilde{E_{i,\eta_{\delta}}}]\sle_{k,0}^{\bub}(\rho)(d\eta_\delta),
        \end{split}
    \end{align}
    where the force point $v(\eta_{\delta})$ is defined in \cite[(3.17)]{SLEbubble}. For each instance of $\eta_{\delta}$, we claim that
    \begin{equation}\label{positivity of surrounding prob}
    	\sle_{\kappa,(\eta_{\delta},v(\eta_\delta))\to 0}^{\mathbb H\backslash \eta_{\delta}}(\rho)[\widetilde{E_{i,\eta_{\delta}}}]>0.
    \end{equation}
    Assume otherwise, i.e., $\sle_{\kappa,(\eta_{\delta},v(\eta_\delta))\to 0}^{\mathbb H\backslash \eta_{\delta}}(\rho)[\widetilde{E_{i,\eta_{\delta}}}]=0$. By conformal invariance property of chordal $\sle_\kappa(\rho)$, we only need to consider the $\sle_{\kappa}(\rho)$ on $\mathbb H$ from $0$ to $\infty$ conditional on passing to the left of $i$. By scaling property of chordal $\sle_{\kappa}(\rho)$, the probability that $\sle_\kappa(\rho)$ conditional on passing to the left of $ai,a>0$ is zero, i.e., $\sle_\kappa(\rho)$ will almost surely stay to the right of positive imaginary axis. This is impossible and leads to a contradiction. 
     Therefore, $\sle_{\kappa,0}^{\bub}(\rho)[\widetilde{E_{i,0}}]>0$ and this completes the proof.
\end{proof}
\subsection{The Liouville Conformal Field Theory}\label{sec:LCFT}
In this section, we review key results of Liouville Conformal Field Theory on $\mathbb H$.
\subsubsection{Definitions of the LCFT}
To start, let $h$ be the centered Gaussian process on $\mathbb H$ with covariance kernel given by 
\begin{equation*}
	\mathbb E[h(x)h(y)]=G_{\mathbb H}(x,y):= \log\frac{1}{|x-y||x-\overline y|}+2\log |x|_++2\log|y|_{+},
\end{equation*}
where $|x|_+=\max(|x|,1)$. Notice that $h\in H^{-1}(\mathbb H)$ and for test functions $f,g\in H^1(\mathbb H)$, $(h,f)$ and $(h,g)$ are centred Gaussian variables with covariance given by 
\begin{equation*}
	\mathbb E[(h,f),(h,g)]=\iint f(x)G_{\mathbb H}(x,y)g(y) d^2x d^2y.
\end{equation*}
Let $P_{\mathbb H}$ denote the law of $h$. For smooth test functions $f$ and $g$ with mean $0$ on $\mathbb H$, i.e., 
\begin{equation*}
	\int_{\mathbb H} f(z)d^2z=\int_{\mathbb H} g(z) d^2z=0,
\end{equation*}
 we have that 
 \begin{equation*}
 	\mathbb E[(h,f),(h,g)]=\frac{1}{2\pi}\int_{\mathbb H}\nabla f(z)\cdot\nabla g(z) d^2z.
 \end{equation*}
 Notice that this characterizes the free boundary Gaussian free field, which is defined modulo an additive constant. We can fix a particular instance of field $h$ by requiring the average around the upper half plane unit circle to be zero. 
 \par
 Given a function $f\in H^{-1}(\mathbb H)$, let $f_\varepsilon(z)$ be the circular average of $f(z)$ over $\partial B(z,\varepsilon)\cap\mathbb H$. Suppose $h$ is sampled from $P_{\mathbb H}$, then we can define the random measures 
\begin{equation*}
 	\mu_h:=\lim_{\varepsilon\to 0}\varepsilon^{\gamma^2/2}e^{\gamma h_\varepsilon(z)} d^2z\qquad\text{and}\qquad \nu_h:=\lim_{\varepsilon\to 0}\varepsilon^{\gamma^2/4}e^{\gamma h_\varepsilon(z)/2}dz,
\end{equation*}
where convergence holds almost surely. We call $\mu_h$ the \textit{quantum area measure} and $\nu_h$ the \textit{quantum boundary length measure}.
\begin{Def}[{\cite[Definition 2.14]{InteofSLE}}]
	Let $(h,\mathbf c)$ be sampled from $P_{\mathbb H}\times [e^{-Qc}dc]$ on the product space $H^{-1}(\mathbb H)\times\mathbb R$. Let $\phi(z)=h(z)-2Q\log|z|_{+}+\mathbf c$ and let $\LF_{\mathbb H}$ denote the law of $\phi(z)$ on $H^{-1}(\mathbb H)$. We call the sample from $\LF_{\mathbb H}$ the \textit{Liouville field}.
\end{Def}
\begin{Lemma}[{\cite[Lemma 2.2]{FZZ}}]
    For $\alpha \in \mathbb R$ and $z_0\in \mathbb H$, the limit
    \begin{equation*}
        \LF_{\mathbb H}^{(\alpha,z_0)}:=\lim_{\varepsilon\to 0}\varepsilon^{\alpha^2/2}e^{\alpha\phi_\varepsilon(z_0)}\LF_{\mathbb H}(d\phi)
    \end{equation*}
    exists in the vague topology. Moreover, sample $(h,\mathbf{c})$ from $(2\Im z_0)^{-\alpha^2/2}|z_0|_{+}^{-2\alpha(Q-\alpha)}P_{\mathbb H}\times\left[e^{(\alpha-Q)c}dc \right]$ and let
    \begin{equation*}
        \phi(z) = h(z) - 2Q\log|z|_{+} + \alpha G_{\mathbb H}(z,z_0)+\mathbf c\qquad\text{for $z\in \mathbb H$},
    \end{equation*}
    then the law of $\phi$ is given by $\LF_{\mathbb H}^{(\alpha,z_0)}$. We call $\LF_{\mathbb H}^{(\alpha,z_0)}$ the Liouville field on $\mathbb H$ with $\alpha$-insertion at $z$.
\end{Lemma}
Next, we introduce the definition of Liouville field with multiple boundary insertions. The following definition is the combination of \cite[Definition 2.15]{InteofSLE} and \cite[Definition 2.17]{InteofSLE}:
\begin{Def}\label{LCFT with boundary insertions}
    Let $(\beta_{i},s_i)\in \mathbb R\times \partial\mathbb H$ for $i=1,\ldots,m$, where $m\geq 0$ and $s_i$ are pairwise distinct. Let $(h,\mathbf c)$ be sampled from $C_{\mathbb H}^{(\beta_i,s_i)_i}P_{\mathbb H}\times \left[e^{(\frac{1}{2}\sum_i\beta_{i}-Q)c}dc\right]$, where 
    \begin{equation*}
        C_{\mathbb H}^{(\beta_i,s_i)_i}=
        \begin{cases}
            \prod_{i=1}^m |s_i|_{+}^{-\beta_{i} (Q-\frac{\beta_i}{2})} e^{\sum_{j=i+1}^m\frac{\beta_i\beta_j}{4}G_{\mathbb H}(s_i,s_j)} &\qquad\text{if $s_1\neq \infty$,}\\
             \prod_{i=2}^m |s_i|_{+}^{-\beta_{i} (Q-\frac{\beta_i}{2}-\frac{\beta_1}{2})} e^{\sum_{j=i+1}^m\frac{\beta_i\beta_j}{4}G_{\mathbb H}(s_i,s_j)} &\qquad\text{if $s_1=\infty$.}
        \end{cases}
    \end{equation*}
    Let 
    \begin{equation*}
        \phi(z) = 
        \begin{cases}
            h(z) - 2Q\log|z|_{+} + \sum_{i=1}^m\frac{\beta_i}{2}G_{\mathbb H}(z,s_i)+\mathbf c &\qquad\text{if $s_1\neq \infty$,}\\
          h(z) + (\beta_1-2Q)\log|z|_{+} + \sum_{i = 2}^m \frac{\beta_i}{2} G_{\mathbb H}(z,s_i)+\mathbf c &\qquad\text{if $s_1=\infty$.}
        \end{cases}
    \end{equation*}
    We write $\LF_{\mathbb H}^{(\beta_i,s_i)_i}$ for the law of $\phi$ and call a sample from $\LF_{\mathbb H}^{(\beta_i,s_i)_i}$ the Liouville field on $\mathbb H$ with boundary insertions $(\beta_{i},s_i)_{1\leq i\leq m}$.
\end{Def}
 \begin{Lemma}[{\cite[Lemma 2.18]{InteofSLE}}]
 	We have the following convergence in the vague topology of measures on $H^{-1}(\mathbb H)$:
 	\begin{equation*}
 		\lim_{r\to +\infty}r^{\beta(Q-\frac{\beta}{2})}\LF_{\mathbb H}^{(\beta,r),(\beta_i,s_i)_i}=\LF_{\mathbb H}^{(\beta,\infty),(\beta_i,s_i)_i}.
 	\end{equation*}
 \end{Lemma}
  \begin{Def}\label{Def one bulk one boundary}
	Let $(\alpha,q)\in \mathbb R\times \mathbb H$ and let $(\beta_i,p_i)\in \mathbb R\times\partial\mathbb H$ for $1\leq i\leq m$. Suppose $(h,\mathbf c)$ is sampled from $C_{\mathbb H}^{(\beta_i,p_i)_i,(\alpha,q)} P_{\mathbb H}\times \left[e^{\left(\frac{1}{2}\sum_i\beta_i+\alpha-Q\right)c}dc \right]$, where 
	\begin{align*}
		C_{\mathbb H}^{(\beta_i,p_i)_i,(\alpha,q)}=
  \begin{cases}
      \prod_{i=1}^m |p_i|_{+}^{-\beta_{i} (Q-\frac{\beta_i}{2})} e^{\sum_{j=i+1}^m\frac{\beta_i\beta_j}{4}G_{\mathbb H}(p_i,p_j)}(2\Im q)^{-\frac{\alpha^2}{2}}|q|_{+}^{-2\alpha(Q-\alpha)}&\qquad\text{if $p_1\neq \infty$,}\\
      \prod_{i=2}^m |p_i|_{+}^{-\beta_{i} (Q-\frac{\beta_i}{2}-\frac{\beta_1}{2})} e^{\sum_{j=i+1}^m\frac{\beta_i\beta_j}{4}G_{\mathbb H}(p_i,p_j)}(2\Im q)^{-\frac{\alpha^2}{2}}|q|_{+}^{-2\alpha(Q-\alpha)}&\qquad\text{if $p_1=\infty$.}
  \end{cases}
	\end{align*}
	Let 
 \begin{align*}
     \phi(z) = 
     \begin{cases}
         h(z)-2Q\log|z|_{+}+\alpha G_{\mathbb H}(z,q)+\sum_{i=1}^m\frac{\beta_i}{2}G_{\mathbb H}(z,p_i)+\mathbf c &\qquad\text{if $p_1\neq \infty$,}\\
         h(z)+(\beta_1-2Q)\log|z|_{+}+\alpha G_{\mathbb H}(z,q)+\sum_{i=2}^m\frac{\beta_i}{2}G_{\mathbb H}(z,p_i)+\mathbf c &\qquad\text{if $p_1= \infty$.}
     \end{cases}
 \end{align*}
 We denote the law of $\phi(z)$ on $H^{-1}(\mathbb H)$ by $\LF_{\mathbb H}^{(\beta_i,p_i)_i,(\alpha,q)}$.
\end{Def}
Finally, we recall the definition of the LCFT on horizontal strip $\mathcal S=\mathbb R\times(0,\pi)$. It is essentially the same procedure as defining LCFT on $\mathbb H$. Let 
\begin{equation*}
	G_{\mathcal S}(z,w)=-\log|e^z-e^w|-\log|e^z-e^{\overline w}|+\max(2\Re z,0)+\max(2\Re w,0)
\end{equation*}
be the Green function on $\mathcal S$. 
\begin{Def}[{\cite[Definition 2.19]{InteofSLE}}]
	Let $(h,\mathbf c)$ be sampled from $C_{\mathcal S}^{(\beta,\pm\infty),(\beta_3,s_3)}P_{\mathcal S}\times \left[e^{(\beta+\frac{\beta_3}{2}-Q)c}dc\right]$, where $\beta\in \mathbb R$ and $(\beta_3,s_3)\in \mathbb R\times \partial\mathcal S$, and 
	\begin{equation*}
		C_{\mathcal S}^{(\beta,\pm\infty),(\beta_3,s_3)}=e^{(-\frac{\beta_3}{2}(Q-\frac{\beta_3}{2})+\frac{\beta\beta_3}{2})|\Re s_3|}.
	\end{equation*}
	Let $\phi(z)=h(z)-(Q-\beta)|\Re z|+\frac{\beta_3}{2}G_{\mathcal S}(z,s_3)+\mathbf c$ and we denote the law of $\phi(z)$ on $H^{-1}(\mathbb H)$ by $\LF_{\mathcal S}^{(\beta,\pm\infty),(\beta_3,s_3)}$. 
\end{Def}
\subsubsection{Conformal symmetries of LCFT}
Let $\conf(\mathbb H)$ be the group of conformal automorphisms of $\mathbb H$ where group multiplication $\cdot$ is the function composition $f\cdot g = f\circ g$.
\begin{Prop}[{\cite[Proposition $2.16$]{InteofSLE}}]
	For $\beta\in \mathbb R$, let $\Delta_\beta=\frac{\beta}{2}(Q-\frac{\beta}{2})$. Let $f\in \conf(\mathbb H)$ and $(\beta_i,s_i)\in \mathbb R\times \partial\mathbb H$ with $f(s_i)\neq \infty$ for all $1\leq i\leq m$. Then $\LF_{\mathbb H}=f_*(\LF_{\mathbb H})$ and 
	\begin{equation*}
		\LF_{\mathbb H}^{(\beta_i,f(s_i))_i}=\prod_{i=1}^m |f'(s_i)|^{-\Delta_{\beta_i}}f_*\left(\LF_{\mathbb H}^{(\beta_i,s_i)_i}\right).
	\end{equation*}
\end{Prop}
\begin{Prop}\label{pushforward of LCFT}
For $1\leq i\leq m$ and $1\leq j\leq n$, let $(\alpha_i,z_i)\in \mathbb R\times\mathbb H$ and $(\beta_j,s_j)\in \mathbb R\times \partial\mathbb H$ with $f(s_j)\neq\infty$ for all $1\leq j\leq n$. Let $f\in \conf(\mathbb H)$ and we have 
\begin{equation*}
	\LF_{\mathbb H}^{(\alpha_i,f(z_i))_i,(\beta_j,f(s_j))_j}=\prod_{i=1}^m \prod_{j=1}^n |f'(z_i)|^{-2\Delta_{\alpha_i}}|f'(s_j)|^{-\Delta_{\beta_j}} \LF_{\mathbb H}^{(\alpha_i,z_i)_i,(\beta_j,s_j)_j}.
\end{equation*}
\end{Prop}
\begin{proof}
    The proof is exactly the same as that of \cite[Proposition 2.9]{InteofSLE}, which describes the case in $\widehat{\mathbb C}$ instead of $\mathbb H$.
\end{proof}
\begin{Lemma}[{\cite[Lemma 3.14]{FZZ}}]
	Let $\alpha\in \mathbb R$ and $u\in \mathcal S$ with $\Re(u)=0$, then we have 
	\begin{equation*}
		\exp_* \LF_{\mathcal S}^{(\alpha,u)}=\LF_{\mathbb H}^{(\alpha,e^u)}.
	\end{equation*}
\end{Lemma}
\begin{Lemma}[{\cite[Lemma 2.20]{InteofSLE}}]
	Let $\beta\in \mathbb R$ and $(\beta_3,s_3)\in \mathbb R\times \partial\mathcal S$, then we have 
	\begin{equation*}
		\LF_{\mathbb H}^{(\beta,\infty),(\beta,0),(\beta_3,e^{s_3})}=e^{\frac{\beta_3^2}{4}\Re s_3}\exp_* \LF_{\mathcal S}^{(\beta,\pm\infty),(\beta_3,s_3)}.
	\end{equation*}
	Similarly, if $\beta_1,\beta_2,\beta_3\in \mathbb R$ and $f\in \conf(\mathbb H)$ satisfies $f(0)=0,f(1)=1,$ and $f(-1)=\infty$, then 
	\begin{equation*}
		\LF_{\mathbb H}^{(\beta_1,\infty),(\beta_2,0),(\beta_3,1)}=2^{\Delta_{\beta_1}-\Delta_{\beta_2}+\Delta_{\beta_3}}\cdot f_*\LF_{\mathbb H}^{(\beta_1,-1),(\beta_2,0),(\beta_3,1)}.
	\end{equation*}
\end{Lemma}
\subsection{Quantum disks}\label{sec:quantum disks}
\subsubsection{Quantum surfaces}
Let $\mathcal D\mathcal H=\lbrace (D,h): D\subset \mathbb C\ \text{open,  $h\in C_0^\infty(D)'$} \rbrace$. We define equivalence relation on $\mathcal D\mathcal H$ by letting $(D,h)\sim (\widetilde D,\widetilde h)$ if there is a conformal map $\psi:D\to \widetilde D$ such that $\widetilde h=\psi\bullet_\gamma h$, where 
\begin{equation}\label{quantum equivalence relation}
	\psi\bullet_\gamma h:=h\circ \psi^{-1}+Q\log|(\psi^{-1})'|.
\end{equation}
A \textit{quantum surface} is an equivalence class of pairs $(D,h)\in \mathcal D\mathcal H$ under the equivalence relation $\sim_\gamma$. An \textit{embedding} of a quantum surface is a choice of representative $(D,h)$. We can also consider quantum surfaces with marked points $(D,h,z_1,\ldots,z_m,\omega_1,\ldots,\omega_n)$ where $z_i\in D$ and $\omega_j\in \partial D$. We say 
\begin{equation*}
	(D,h,z_1,\ldots,z_m,\omega_1,\ldots,\omega_n)\sim_\gamma (\widetilde D,\widetilde h,\widetilde{z_1},\ldots,\widetilde{z_m},\widetilde{\omega_1},\ldots,\widetilde{\omega_n})
\end{equation*}
if there is a conformal map $\psi:D\to \widetilde D$ such that $\widetilde h=\psi\bullet_\gamma h$ and $\psi(z_i)=\widetilde{z_i},\psi(\omega_j)=\widetilde{\omega_j}$. Let $\mathscr D_{m,n}$ denote the set of equivalence class of such tuples under $\sim_\gamma$ and  let $\mathscr D=\mathscr D_{0,0}$ for simplicity. We use $(\ref{quantum equivalence relation})$ to define the equivalence relation because $\gamma$-LQG quantum area and $\gamma$-LQG quantum length measure is invariant under pushforward $\bullet_\gamma$. Since we will mainly work with $\mathbb H$, we view the set $\mathscr D_{m,n}$ as the quotient space 
\begin{equation*}
	\lbrace(\mathbb H,h,z_1,\ldots,z_m,\omega_1,\ldots,\omega_n):h \ \text{is a distribution on $\mathbb H$}, z_1\ldots,z_m\in \mathbb H, \omega_1,\ldots,\omega_n\in \overline{\mathbb R} \rbrace/\sim_\gamma.
\end{equation*}
The Borel $\sigma$-algebra of $\mathscr D_{m,n}$ is induced by the Borel sigma algebra on $H^{-1}(\mathbb H)$. 
\subsubsection{Quantum Disks}
We recall the definitions of two-pointed quantum disk introduced in \cite{Weldisk}. It is a family of measures on $\mathscr D_{0,2}$. It is initially defined on the horizontal strip $\mathcal S=\mathbb R\times (0,\pi)$. Let $\exp:\mathcal S\to \mathbb H$ be the exponential map $z\mapsto e^z$ and let $h_{\mathcal S}=h_{\mathbb H}\circ \exp$ where $h_{\mathbb H}$ is sampled from $P_{\mathbb H}$. We call $h_{\mathcal S}$ the \textit{free boundary GFF} on $\mathcal S$. It is known that $h_{\mathcal S}$ can be written as the sum of $h^c$ and $h^\ell$ where $h^c$ is constant on $u+[0,i\pi],u\in \mathbb R$ and $h^\ell$ has mean zero on all such vertical lines. We call $h^\ell$ the \textit{lateral component} of free boundary GFF.
\begin{Def}[Thick quantum disk]\label{Def: thick quantum disk}
	Let $W\geq\frac{\gamma^2}{2}$, and let $\beta=Q+\frac{\gamma}{2}-\frac{W}{\gamma}$. Let 
	\begin{equation*}
		Y_t=
		\begin{cases}
			B_{2t}-(Q-\beta)t &\qquad\text{if $t\geq 0$,}\\
			\widetilde{B_{-2t}}+(Q-\beta)t &\qquad\text{if $t<0$,}
		\end{cases}
	\end{equation*}
	where $(B_s)_{s\geq 0}, (\widetilde B_{s})_{s\geq 0}$ are independent standard Brownian motions conditional on $B_{2s}-(Q-\beta)s<0$ and $\widetilde{B_{2s}}-(Q-\beta)s<0$ for all $s>0$. Let $h^1(z)=Y_t$ for all $z$ with $\Re(z)=t$. Let $h^2(z)$ be the lateral component of free boundary GFF on $\mathcal S$ and let $\mathbf c$ be sampled from $\frac{\gamma}{2}e^{(\beta-Q)c}dc$ independent of $h^1$ and $h^2$. Let $\widehat h(z)=h^1(z)+h^2(z)$ and let $\phi(z)=\widehat h(z) +\mathbf c$. Let $\mathcal M^{\disk}_{0,2}(W)$ denote the infinite measure on $\mathscr D_{0,2}$ describing the law of $(\mathcal S,\phi,-\infty,+\infty)$. We call a sample from $\mathcal M^{\disk}_{0,2}(W)$ a weight-$W$ quantum disk.
\end{Def}
\begin{Th}[{\cite[Theorem 2.22]{InteofSLE}}]
    Fix $W\geq\frac{\gamma^2}{2}$ and $\beta_W = \gamma + \frac{2-W}{\gamma}$. If we independently sample $T$ from $\leb_{\mathbb R}$ and $(\mathcal S,\phi,+\infty,-\infty)$ from $\mathcal M^{\disk}_{0,2}(W)$, then the law of $\widetilde{\phi}:=\phi(\cdot+T)$ is $\frac{\gamma}{2(Q-\beta_W)^2}$ is $\frac{\gamma}{2(Q-\beta_W)^2}\LF_{\mathcal S}^{(\beta_W,\pm\infty)}$.
\end{Th}
\begin{Def}\label{Def one addition boundary marked point}
    For $W\geq\frac{\gamma^2}{2}$, we first sample $(\mathcal S,\phi,+\infty,-\infty)$ from $\nu_\phi(\mathbb R)\mathcal M^{\disk}_{0,2}(W)[d\phi]$, then sample $s\in \mathbb R$ according to the probability measure proportional to $\nu_\phi|_{\mathbb R}$. We denote the law of the surface $(\mathcal S,\phi,+\infty,-\infty)/\sim_\gamma$ has the same law as $\mathcal M^{\disk}_{2,\bullet}(W)$.
\end{Def}
\begin{Def}
    Fix $W\geq\frac{\gamma^2}{2}$ and let $\alpha\in \mathbb R$. Let $\mathcal M^{\disk}_{2,\bullet}(W;\alpha)$ denote the law of $(\mathcal S,\phi,\pm\infty,0)/\sim_\gamma$ with $\phi$ sampled from $\frac{\gamma}{2(Q-\beta)^2}\LF_{\mathcal S}^{(\beta,\pm\infty),(\alpha,0)}$. 
\end{Def}
\begin{Def}[Thin quantum disk]\label{Def: thin quantum disk}
	Let $0<W<\frac{\gamma^2}{2}$ and define the infinite measure $\mathcal M^{\disk}_{0,2}(W)$ on two-pointed beaded surfaces as follows: first take $T$ according to $(1-\frac{2}{\gamma^2}W)^{-2} \leb_{\mathbb R_+}$, then sample a Poisson point process $\lbrace (u,\mathcal D_u)\rbrace$ according to $\leb_{[0,T]}\times \mathcal M^{\disk}_{0,2}(\gamma^2-W)$ and concatenate the $\mathcal D_u$ according to ordering induced by $u$. 
\end{Def}
\begin{Def}
    For $W\in (0,\frac{\gamma^2}{2})$ and $\alpha\in \mathbb R$, let $(S_1,S_2,S_3)$ be sampled from 
    \begin{equation*}
        \left(1-\frac{2}{\gamma^2}W\right)^2\mathcal M^{\disk}_{0,2}(W)\times \mathcal M^{\disk}_{2,\bullet}(\gamma^2-W;\alpha)\times\mathcal M^{\disk}_{0,2}(W) 
    \end{equation*}
    and $S$ is the concatenation of the three surfaces. We define the infinite measure $\mathcal M^{\disk}_{2,\bullet}(W;\alpha)$ to be the law of $S$.
\end{Def}
When $W=2$, two marked points of $\mathcal M^{\disk}_{0,2}(2)$ are typical with respect to the quantum boundary length measure, see \cite[Proposition $A.8$]{Weldisk}.
\begin{Def}[Typical quantum disks]\label{qd with arbitrary marked points}
	Let $(\mathcal S,\phi,-\infty,+\infty)$ be an embedding of a sample from $\mathcal M^{\disk}_{0,2}(2)$. Let $A=\mu_\phi(\mathcal S)$ denote the total quantum area and $L=\nu_\phi(\partial\mathcal S)$ denote the total quantum boundary length. Let $\qd$ denote the law of $(\mathcal S,\phi)$ under reweighted measure $L^{-2}\mathcal M^{\disk}_{0,2}(2)$, viewed as a measure on $\mathscr D$ by forgetting two marked points. For non-negative integers $m$ and $n$, let $(\mathcal S,\phi)$ be a sample from $A^mL^n\qd$, then independently sample $z_1,\ldots,z_m$ and $\omega_1,\ldots,\omega_n$ according to $\mu_\phi^{\#}$ and $\nu_\phi^{\#}$, respectively. Let $\qd_{m,n}$ denote the law of $(\mathcal S,\phi,z_1,\ldots,z_m,\omega_1,\ldots,\omega_n)$ viewed as a measure on $\mathscr D_{m,n}$. We call a sample from $\qd_{m,n}$ quantum disk with $m$ bulk and $n$ boundary marked points.
\end{Def}
\subsubsection{Conformal welding of quantum disks}
The following theorem describes the conformal welding of $n$ quantum disks. Notice that the weight $W$ is linearly added when performing the welding operation. 
\begin{Th}[{\cite[Theorem 2.2]{Weldisk}}]\label{Disk welding theorem}
    Fix $W_1,\ldots,W_n>0$ and $W=W_1+\ldots+W_n$. There exists a constant $C=C_{W_1,\ldots,W_n}\in (0,\infty)$ such that for all $\ell,r>0$, the identity 
    \begin{align}
        \begin{split}
            &\mathcal{M}^{\disk}_{0,2}(W;\ell,r)\otimes \mathcal P^{\disk}(W_1,\ldots,W_n)\\
            &=C\iiint_0^\infty \mathcal{M}^{\disk}_{0,2}(W_1;\ell,\ell_1)\times \mathcal{M}^{\disk}_{0,2}(W_2;\ell_1,\ell_2)\times\ldots\times \mathcal{M}^{\disk}_{0,2}(W_n;\ell_{n-1},r)d\ell_1 \ldots d\ell_{n-1}
        \end{split}
    \end{align}
    holds as measures on the space of curve-decorated quantum surfaces. The measure $\mathcal P^{\disk}(W_1,\ldots,W_n)$ is defined in \cite[Definition 2.25]{Weldisk} on tuple of curves $(\eta_1,\ldots,\eta_{n-1})$ in a domain $(D,x,y)$. It was defined by the following induction procedure: first sample $\eta_{n-1}$ from $\sle_{\kappa}(W_1+\ldots+W_{n-1}-2;W_n-2)$ then $(\eta_1,\ldots,\eta_{n-2})$ from $\mathcal P^{\disk}(W_1,\ldots,W_{n-1})$ on connected component $(D',x',y')$ on the left of $D\backslash\eta_{n-1}$ where $x'$ and $y'$ are the first and the last point hit by $\eta_{n-1}$.
\end{Th}
\section{Law of welding interface via a limiting procedure}\label{sec: curve law}
In this section, we prove Proposition \ref{Curve Law Prop}. In words, we show that under the same setup as Theorem $\ref{Conditional Welding}$, the law of the welding interface is $\sle_\kappa(\rho)$ bubble measure conditioning on surrounding $i$. 
\begin{Prop}\label{Curve Law Prop}
Fix $\gamma\in (0,2)$. For $W>0$, let $\rho = W-2$. Let $(\mathbb H,\phi,\eta,0,i)$ be an embedding of the quantum surface
\begin{equation}
		\int_0^\infty \mathcal M^{\disk}_{0,2}(W;\cdot,\ell)\times \qd_{1,1}(\ell) d\ell. 
\end{equation}
Let $M_{\phi}$ denote the marginal law of $\phi$ in $(\mathbb H,\phi,\eta,0,i)$, then  $(\phi,\eta)$ has the law of $M_\phi\times \sle_{\kappa,0}^{\bub}(\rho)[\cdot|i\in D_\eta(0)]$.
\end{Prop}
\subsection{The LCFT description of three-pointed quantum disks}
We start with the definition of two-pointed quantum disk with one additional typical bulk insertion. 
\begin{Def}[{\cite[Definition 3.10]{FZZ}}]
	For $W\geq\frac{\gamma^2}{2}$, recall the definition of thick quantum disk $\mathcal M^{\disk}_{0,2}(W)$ from Definition \ref{Def: thick quantum disk}. Sample $\phi$ on $H^{-1}(\mathbb H)$ such that $(\mathbb H,\phi,0,\infty)$ is an embedding of $\mathcal M^{\disk}_{0,2}(W)$. Let $L$ denote the law of $\phi$ and let $(\phi,z)$ be sampled from $L(d\phi)\mu_{\phi}(dz^2)$. We write $\mathcal M^{\disk}_{1,2}(W)$ for the law $(\mathbb H,\phi,z,0,\infty)$ viewed as a marked quantum surface.
\end{Def} 
\begin{Lemma}\label{LCFTthree}
For $\gamma\in (0,2)$ and $W\in \mathbb R$, let $\beta_W=\gamma+\frac{2-W}{\gamma}$. Suppose $(\phi,\mathbf x)$ is sampled from $\LF_{\mathbb H}^{(\gamma,i),(\beta_W,\infty),(\beta_W,\mathbf x)}\times dx$, then the law of $(\mathbb H,\phi,\eta,i,\infty,\mathbf x)$ as a marked quantum surface is equal to $\frac{2(Q-\beta_W)^2}{\gamma}\mathcal{M}^{\disk}_{1,2}(W)$.
\end{Lemma}
\begin{proof}
By \cite[Lemma $3.12$]{FZZ}, if $\mathcal{M}^{\disk}_{1,2}(W)$ is embedded as $(\mathcal{S},\phi,i\theta,+\infty,-\infty)$, then $(\phi,\theta)$ has the law of 
\begin{equation}
    \frac{\gamma}{2(Q-\beta_W)^2}\LF_{\mathcal{S}}^{(\beta_W,\pm\infty),(\gamma,i\theta)}(d\phi)\mathbbm{1}_{\widetilde{ \theta}\in (0,\pi)}d\widetilde{\theta}.
\end{equation}
Fix $\theta\in (0,\pi)$ and let $\exp:\mathcal{S}\to \mathbb H$ be the map $z\mapsto e^z$. By \cite[Lemma $3.14$]{FZZ} and \cite[Lemma $2.20$]{InteofSLE}, we have
\begin{equation*}
    \exp_*\left( \LF_{\mathcal{S}}^{(\beta_W,\pm\infty),(\gamma,i\theta)}\right) = \LF_{\mathbb H}^{(\gamma, e^{i\theta}),(\beta_W,\infty),(\beta_W,0)}.
\end{equation*}
Let $f_\theta(z)=\frac{z}{\sin\theta}-\cot\theta$, which sends $e^{i\theta}\mapsto i$, $\infty\mapsto\infty$, and $0\mapsto \mathbf{x}=-\cot\theta$. By \cite[Proposition $2.16$]{InteofSLE}, for any $r\in \mathbb R$, we have
\begin{equation*}
    \LF_{\mathbb H}^{(\gamma,i),(\beta_W,-\cot\theta),(\beta_W,\frac{r}{\sin\theta}-\cot\theta)} = (\sin\theta)^{2\Delta_\gamma+2\Delta_\beta} (f_\theta)_* \LF_{\mathbb H}^{(\gamma,e^{i\theta}),(\beta_W,0),(\beta_W,r)}, 
\end{equation*}
where $\Delta_\alpha=\frac{\alpha}{2}(Q-\frac{\alpha}{2})$. After multiplying both sides by $\left(\frac{r}{\sin\theta}-\cot\theta\right)^{2\Delta_{\beta_W}}$, we have 
	\begin{align*}
		&\left(\frac{r}{\sin\theta}-\cot\theta\right)^{2\Delta_{\beta_W}} \LF_{\mathbb H}^{(\gamma,i),(\beta_W,-\cot\theta), (\beta_W, \frac{r}{\sin\theta}-\cot\theta)}\\
		&=(\sin\theta)^{2\Delta_\gamma+2\Delta_{\beta_W}}\left(\frac{1}{\sin\theta}-\frac{\cot\theta}{r}\right)^{2\Delta_{\beta_W}} \left(f_\theta\right)_* \left(r^{2\Delta_{\beta_W}} \LF_{\mathbb H}^{(\gamma, e^{\theta i}), (\beta_W, 0),(\beta_W,r)} \right).
	\end{align*}
	By \cite[Lemma $2.18$]{InteofSLE}, taking limit as $r\to \infty$ yields
	\begin{equation*}
	   \frac{1}{(\sin\theta)^2} \LF_{\mathbb H}^{(\gamma,i),(\beta_W,-\cot\theta), (\beta_W, \infty)}= \left(f_{\theta}\right)_{*}\LF_{\mathbb H}^{(\gamma, e^{\theta i}), (\beta_W,0), (\beta_W, \infty)}.
	\end{equation*}
	Here the convergence is in the vague topology. When $\theta$ is sampled from $\mathbbm{1}_{(0,\pi)}(\widetilde{\theta}) d\widetilde{\theta}$, we have  
	\begin{equation*}
	    \frac{1}{(\sin\theta)^2} \LF_{\mathbb H}^{(\gamma,i),(\beta_W,-\cot\theta), (\beta_W, \infty)}=\LF_{\mathbb H}^{(\gamma,i),(\beta_W,\infty),(\beta_W,\mathbf x)}\times dx
	\end{equation*}
 by change of variables $\mathbf x = -\cot\theta$. This completes the proof.
\end{proof}
A direct consequence of \cite[Theorem $2.2$]{Weldisk} is the following: 
\begin{Th}\label{Conformal welding of two disks}
	Let $(\mathbb H, \phi, 0,\infty)$ be the embedding of a sample from $\mathcal M^{\disk}_{0,2}(W+2)$. Let $\eta$ be sampled from $\sle_\kappa(W-2,0)$ on $(\mathbb H,0,\infty)$ independent of $\phi$, then 
	\begin{equation}
		\mathcal M_{0,2}^{\disk} (W+2)\otimes \sle_{\kappa}(W-2,0)=C_{W,2}\int_{0}^{\infty} \mathcal M_{0,2}^{\disk}(W,\cdot,\ell)\times \mathcal M^{\disk}_{0,2}(2,\ell,\cdot)d\ell
	\end{equation}
	for some constant $C_{W,2}\in (0,\infty)$. 
\end{Th}
For $W>0$, let $\beta_{W+2} = \gamma-\frac{W}{\gamma}$. Let $(\phi,\mathbf x)$ be sampled from $\LF_{\mathbb H}^{(\gamma,i),(\beta_{W+2},\infty),(\beta_{W+2},\mathbf x)}\times dx$ and let $\eta$ be sampled from the chordal $\sle^{\mathbb H}_{\kappa,(\mathbf{x};\mathbf{x}^-)\to \infty}(W-2)$. Denote $\nu_{\phi}(a,b)$ the quantum boundary length of $(a,b)$ with respect to the random field $\phi$. Fix $\delta\in (0,\frac{1}{2})$ and let $M_{\delta}$ denote the law of $(\phi,\mathbf x, \eta)$ restricted to the event that $\nu_{\phi}(\mathbf x,\infty)\in (\delta, 2\delta)$, $\nu_{\phi}(\mathbb R)\in (1,2)$ and $i$ is to the right of $\eta$. Let $M_\delta^{\#} = \frac{1}{|M_\delta|}M_\delta$ be the corresponding probability measure.
\begin{figure}[H]
	\includegraphics[scale = 0.8]{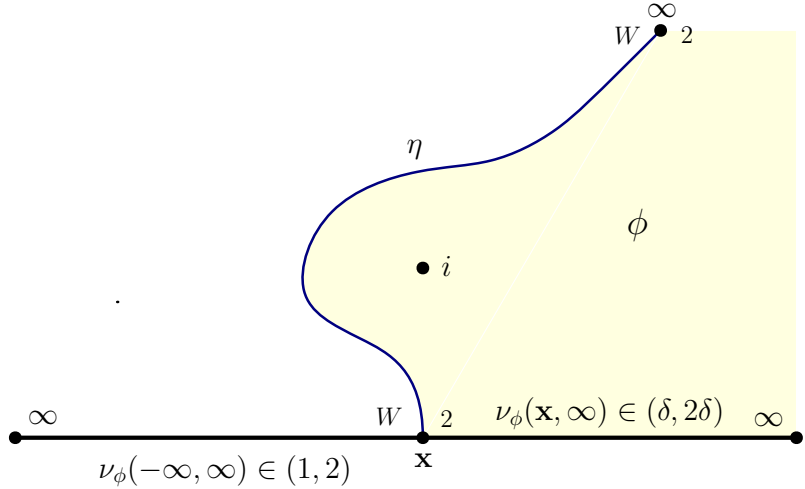}
	\centering
	\caption{Illustration of $M_\delta$: first sample $(\phi,\mathbf x)$ from $\LF_{\mathbb H}^{(\gamma,i),(\beta_{W+2},\infty),(\beta_{W+2},\mathbf x)}\times dx$ and then sample $\eta$ according to $\sle^{\mathbb H}_{\kappa,(\mathbf{x};\mathbf{x}^-)\to \infty}(W-2)$. The $M_\delta$ is the restriction $(\phi,\mathbf x,\eta)$ to the event that $\nu_{\phi}(\mathbf x,\infty)\in (\delta, 2\delta)$, $\nu_{\phi}(-\infty,\infty)\in (1,2)$ and $i$ is to the right of $\eta$.}
	\label{fig:M_delta}
\end{figure}
\begin{Lemma}\label{Lem:02+12}
	Fix $W>0$. There exists some constant $C\in (0,\infty)$ such that for each $\delta\in (0,\frac{1}{2})$, if $(\phi,\mathbf x, \eta)$ is sampled from $M_{\delta}$, then the law of marked quantum surface $(\mathbb H, \phi,\eta, i,\mathbf x, \infty)$ is 
	\begin{equation}\label{welding equation}
		C\cdot\int_{\delta}^{2\delta}\int_{1-\ell'}^{2-\ell'}\int_{0}^{\infty} \mathcal M_{0,2}^{\disk}(W; \ell,\ell_1)\times \mathcal M^{\disk}_{1,2}(2,\ell_1,\ell')d\ell_1 d\ell d\ell'.
	\end{equation}
\end{Lemma}
\begin{proof}
	By Lemma $\ref{LCFTthree}$, if we sample $(\phi,\mathbf x)$ from $\LF_{\mathbb H}^{(\gamma,i),(\beta_{W+2},\infty),(\beta_{W+2},\mathbf x)}\times dx$, then $(\mathbb H, \phi, i, \mathbf x, \infty)$ viewed as a marked quantum surface has the law of $C\cdot\mathcal M_{1,2}^{\disk}(W+2)$ for some constant $C\in (0,\infty)$. Furthermore, if we sample $\eta$ from $\sle^{\mathbb H}_{\kappa,(\mathbf{x};\mathbf{x}^-)\to \infty}(W-2)$ conditional on $i$ is to the right of $\eta$, then by Theorem $\ref{Conformal welding of two disks}$, the quantum surface $(\mathbb H, \phi, \eta, i,\mathbf x, \infty)$ has the law of 
\begin{equation}
	C\cdot\int_{0}^{\infty}\int_{0}^{\infty}\int_{0}^{\infty}\mathcal M_{0,2}^{\disk}(W; \ell, \ell_1)\times \mathcal M_{1,2}^{\disk}(2;\ell_1,\ell')d\ell d\ell_1 d\ell'.
\end{equation}
Conditioning on $\nu_{\phi}(\mathbf x,\infty)\in (\delta,2\delta)$ and $\nu_{\phi}(\mathbb R)\in (1,2)$ gives the desired result. 
\end{proof}
\subsection{Proof of Proposition \ref{Curve Law Prop} via coupling}
Fix $W>0$. Sample a pair of quantum surfaces $(\mathcal{D}_1,\mathcal{D}_2)$ from
\begin{equation}
    \int_1^2 \int_0^\infty \mathcal M^{\disk}_{0,2}(W;a,p)\times \qd_{1,1}(p)dpda
\end{equation}
and let $\mathcal D_1\oplus\mathcal D_2$ be the curve-decorated quantum surface obtained by conformally welding the right boundary of $\mathcal D_1$ and total boundary of $\mathcal D_2$. Notice that $\mathcal D_1\oplus\mathcal D_2$ has a interior marked point and a boundary marked point. Let $(\mathbb D, \phi_{\mathbb D},\eta_{\mathbb D},0,i)$ be the unique embedding of $\mathcal D_1\oplus \mathcal D_2$ on $(\mathbb D, 0, i)$ and let $f:\mathbb H\to \mathbb D$ be the conformal map with $f(i)=0$ and $f(\infty)=i$. Denote $M_{\mathbb D}$ the joint law of $(\mathbb D, \phi_{\mathbb D},\eta_{\mathbb D},0,i)$ and let $M_{\mathbb D}^{\#}=\frac{1}{|M_{\mathbb D}|}M_{\mathbb D}$ be the probability measure obtained from $M_{\mathbb D}$.
\begin{figure}[H]
	\includegraphics[scale = 0.8]{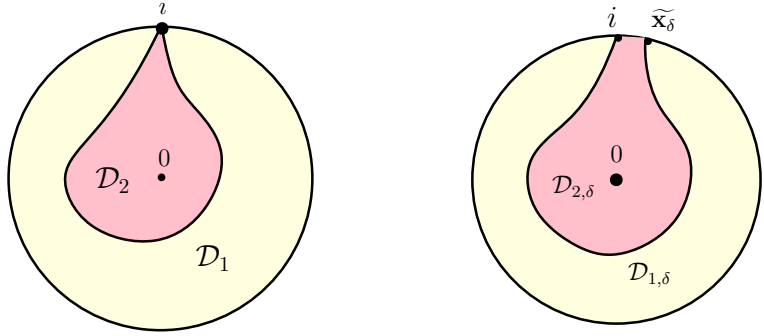}
	\centering
	\caption{\textbf{Left:} $(\phi_{\mathbb D},\eta_{\mathbb D})$ from $M_{\mathbb D}^{\#}$ is obtained by embedding $\mathcal D_1\oplus\mathcal D_2$ into $(\mathbb D,0,i)$. \textbf{Right:} $(\phi^\delta,\eta^\delta)$ from $M_\delta^{\#}$ is obtained by embedding $\mathcal D_{1,\delta}\oplus \mathcal D_{2,\delta}$ into $(\mathbb D,0,i)$.}
	\label{fig:Two_disks}
\end{figure}
Next, we recall the definition of $M_\delta^{\#}$. For $0<\gamma<2$ and $W>0$, let $\beta_{W+2}=\gamma-\frac{W}{\gamma}$. Sample $(\phi,\mathbf x)$ from $\LF_{\mathbb H}^{(\gamma,i),(\beta_{W+2},\infty),(\beta_{W+2} ,\mathbf x)}\times dx$ and let $\eta$ be sampled from $\sle^{\mathbb H}_{\kappa,(\mathbf{x};\mathbf{x}^-)\to \infty}(W-2)$. Fix $\delta\in (0,\frac{1}{2})$ and let $M_{\delta}$ be the law of $(\phi,\mathbf x, \eta)$ restricted to the event that $\nu_{\phi}(\mathbf x,\infty)\in (\delta, 2\delta)$, $\nu_{\phi}(\mathbb R)\in (1,2)$ and $i$ is to the right of $\eta$. Let $M_\delta^{\#}=\frac{1}{|M_\delta|}M_\delta$ be the corresponding probability measure.
\par
Sample $(\phi,\mathbf x,\eta)$ from $M_\delta^{\#}$ and let $\mathcal D_{1,\delta}$ and $\mathcal D_{2,\delta}$ be the two components such that $(\mathbb H, \phi, \eta, i,\mathbf x)$ is the embedding of the surface $\mathcal D_{1\,\delta}\oplus \mathcal D_{2,\delta}$ after conformal welding. Let $\phi^{\delta}=\phi\circ f^{-1}+\log |(f^{-1})'|$ and $\eta^\delta=f\circ\eta$ be such that $(\mathbb D, \phi^\delta,\eta^\delta, 0,i)$ is the embedding of $\mathcal D_{1,\delta}\oplus \mathcal D_{2,\delta}$. Here $\eta^{\delta}$ is the welding interface between $\mathcal D_{1,\delta}$ and $\mathcal D_{2,\delta}$. Let $\widetilde {\mathbf x_\delta}=f(\mathbf x)$ be the image of $\mathbf x$ under $f$.
\begin{Lemma}\label{Key coupling lemma}
	There exists a coupling between $M_{\mathbb D}^{\#}$ and $M_\delta^{\#}$ such that the followings hold: There exist random simply connected domains $U_\delta$ and $\widetilde{U_{\delta}}\subset \mathbb D$ and a conformal map $g_{\delta}:\widetilde{U_\delta}\to U_{\delta}$ satisfying the following properties: With probability $1-o_{\delta}(1)$, we have
	\begin{enumerate}
	    \item $\phi_{\mathbb D}(z)=	\phi^{\delta}\circ g_{\delta}(z)+Q\log |g_\delta'(z)|, \qquad\text{for $z\in\widetilde{U_\delta}$}.$
	    \item $\diam(\mathbb D\backslash U_\delta)=o_\delta(1) \ \text{and} \ \diam(\mathbb D\backslash \widetilde{U_\delta})=o_\delta(1).$
	    \item $|\widetilde{\mathbf x_\delta}-i|=o_\delta(1).$
	    \item $\sup_{z\in K}|g_\delta(z)-z|=o_\delta(1), \ \ \text{for any compact set $K\subset \mathbb D$}.$
	\end{enumerate}
\end{Lemma}
In order to prove Lemma $\ref{Key coupling lemma}$, we need the following two basic coupling results on the quantum disk. The first one is on $\qd_{1,1}$. Suppose $\mathcal D $ as a quantum surface has the law of $ \qd_{1,1}$ and it has emebdding $(\mathbb H,\phi, i,-1)$. Let $\mathcal D^\varepsilon:=\left(\mathbb H_\varepsilon, \phi,i,-1,-1-2\varepsilon\right)$, where $\mathbb H_\varepsilon=\mathbb H\backslash B_\varepsilon(-1-\varepsilon)$ with $B_\varepsilon(-1-\varepsilon)=\left\lbrace z\in \mathbb C: |z+1+\varepsilon|\leq \varepsilon \right\rbrace$. 
\begin{Lemma}[{\cite[Lemma $5.17$]{FZZ}}]\label{Coupling of $1,1$}
	For $\varepsilon>0$ and $\ell>0$, suppose $\mathcal D$ and $\widetilde{\mathcal D}$ are sampled from $\qd_{1,1}(\ell)^{\#}$ and $\qd_{1,1}(\widetilde \ell)^{\#}$ respectively, then the law of $\widetilde{\mathcal D^\varepsilon}$ converges in total variation distance to $\mathcal D^\varepsilon$ as $\widetilde \ell\to \ell$. 
\end{Lemma}
The second coupling result is on $\mathcal M^{\disk}_{0,2}(W)$. Suppose $\mathcal D$ is sampled from $\mathcal M^{\disk}_{0,2}(W)$ and it has embedding $(\mathbb D, \phi,-i,i )$. With a slight abuse of notation, let $\mathcal D^\varepsilon:=(\mathbb D_\varepsilon,\phi, \alpha_\varepsilon, \alpha'_\varepsilon,i,-i)$, where $B_\varepsilon(i)=\left\lbrace z\in \mathbb C: |z-i|\leq \varepsilon \right\rbrace$, $\mathbb D_\varepsilon =\mathbb D\backslash B_\varepsilon(i)$, and $\lbrace \alpha_\varepsilon, \alpha'_\varepsilon\rbrace=\partial\mathbb D\cap \partial B_\varepsilon(i)$.
\begin{Lemma}\label{Coupling of $0,2$}
	Fix $W>0$. For $\varepsilon,\ell,r,\widetilde \ell,\widetilde r >0$, suppose $\mathcal D$ and $\widetilde{\mathcal D}$ are sampled from $\mathcal M^{\disk}_{0,2}(W;\ell,r)^{\#}$ and $\mathcal M^{\disk}_{0,2}(W;\widetilde \ell,\widetilde r)^{\#}$ respectively, then $\widetilde{\mathcal D^\varepsilon}$ converges in total variation distance to $\mathcal D$ as $(\widetilde\ell,\widetilde r)\to (\ell,r)$.
\end{Lemma}
\begin{proof}
    The proof follows directly from \cite[Proposition $2.23$]{Weldisk}.
\end{proof}
\begin{Lemma}\label{quantum length converges}
   Suppose $(\phi,\mathbf x,\eta)$ is sampled from $M_\delta^{\#}$ and let $A=\nu_{\phi}(-\infty, \mathbf x)$, $B=\nu_{\phi}(\mathbf x,\infty)$ and $P=\nu_{\phi}(\eta)$, then as $\delta\to 0$, $B$ converges to $0$ in probability and the $M_\delta^{\#}$-law of $(A,P)$ converges in total variation distance to a probability measure on $(1,2)\times (0,\infty)$ whose density function is proportional to 
   \begin{equation}
   	f_{W}(a,p)p^{-\frac{4}{\gamma^2}+1}dadp,
   \end{equation}
  where $f_{W}(a,p)=|\mathcal M^{\disk}_{0,2}(W;a,p)|$. 
\end{Lemma}
\begin{proof}
   By Proposition $5.1$ and \cite[Lemma $3.3$]{InteofSLE}, we have
   \begin{equation}
   	|\mathcal M^{\disk}_{0,2}(W;\ell,r)|=f_W(\ell,r)\qquad\text{and}\qquad |\mathcal M^{\disk}_{1,2}(2;\ell,r)|=C(\ell+r)^{-\frac{4}{\gamma^2}+1}.
   \end{equation}
   By $(\ref{welding equation})$, the $M_\delta^{\#}$-law of $(A,P,B)$ is a probability measure on the space
   \begin{equation*}
   	S_\delta=\left\lbrace (a,p,b)\in (0,\infty)^{3}: b\in (\delta,2\delta),a+b\in (1,2) \right\rbrace, 
   \end{equation*}
   whose density function is proportional to 
   \begin{equation*}
   	m(a,p,b)=f_W(a,p)(p+b)^{-\frac{4}{\gamma^2}+1}.
   \end{equation*}
   Therefore, we have
   \begin{equation*}
   	|M_\delta|=\int_{S_\delta}m(a,p,b)dadpdb.
   \end{equation*}
   By definition of $M_\delta^{\#}$, for any $\varepsilon>0$, we have $\lim_{\delta\to 0}M_\delta^{\#}[B>\varepsilon]=0$. As $\delta\to 0$, the limiting $M_\delta^{\#}$-law of $(A,P)$ is a probability measure on $(1,2)\times (0,\infty)$ whose density function is proportional to $f_W(a,p)p^{-\frac{4}{\gamma^2}+1}$. This completes the proof.
\end{proof}
\begin{figure}[H]
	\includegraphics[scale = 0.8]{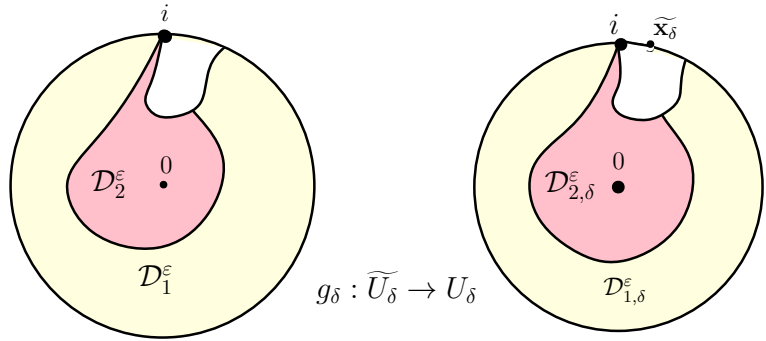}
	\centering
	\caption{We can couple $M_{\mathbb D}^{\#}$ and $M_{\delta}^{\#}$ so that the light green and pink quantum surfaces agree with high probability. The domain $\widetilde{U_\delta}$ is the interior of $\overline{\mathcal D_1^\varepsilon\cup \mathcal D_2^\varepsilon}$ in the embedding of $\mathcal D_1\oplus\mathcal D_2$ and $U_\delta$ is the interior of $\overline{\mathcal D_{1,\delta}^\varepsilon\cup \mathcal D_{2,\delta}^\varepsilon}$ in the embedding of $\mathcal D_{1,\delta}\oplus\mathcal D_{2,\delta}$.}
	\label{fig: coupling}
\end{figure}
\begin{proof}[Proof of Lemma \ref{Key coupling lemma}]
	Recall the definition of marked quantum surfaces $\mathcal D_1$ and $\mathcal D_2$ embedded as $(\mathbb D, \phi_{\mathbb D},\eta_{\mathbb D},0,i)$. Let $\widetilde A$ and $\widetilde P$ be the left and right boundary length of $\mathcal D_1$ respectively. The law of $(\widetilde A, \widetilde P)$ is the probability measure on $[1,2]\times (0,\infty)$ proportional to
	\begin{equation}
		\left|\mathcal M^{\disk}_{0,2}(W;a,p)\right||\qd_{1,1}(p)|\propto f_W(a,p)p^{-\frac{4}{\gamma^2}+1}. 
	\end{equation} 
	Conditioning on $(\widetilde A,\widetilde P)$, the joint law of $(\mathcal D_1,\mathcal D_2)$ is $\mathcal M^{\disk}_{0,2}(W;\widetilde A,\widetilde P)^{\#}\times \qd_{1,1}(\widetilde P)^{\#}$. 
	\par
	Next, let $A_\delta$ and $P_\delta$ be the left and right boundary of $D_{1,\delta}$ respectively and let $B_\delta$ be the right boundary of $D_{2,\delta}$. By Lemma $\ref{quantum length converges}$, as $\delta\to 0$, $M_\delta^{\#}$-law of $(A_\delta,P_\delta)$ converges in law to $(\widetilde A,\widetilde P)$ and $B_\delta\to 0$ in probability. Therefore, we can couple $M_\delta^{\#}$ and $M_{\mathbb D}^{\#}$ so that $(A_\delta,P_\delta)=(\widetilde A,\widetilde P)$ with probability $1-o_{\delta}(1)$. By Lemma $\ref{Coupling of $1,1$}$ and $\ref{Coupling of $0,2$}$, there exists a coupling between $(\mathcal D_1^\varepsilon,\mathcal D_2^\varepsilon)$ and $(\mathcal D_{1,\delta}^\varepsilon,\mathcal D_{2,\delta}^\varepsilon)$ such that 
	\begin{equation}\label{coupling equation}
		\lim_{\delta\to 0} \mathbb P\left[(\mathcal D_1^\varepsilon,\mathcal D_2^\varepsilon)=(\mathcal D_{1,\delta}^\varepsilon,\mathcal D_{2,\delta}^\varepsilon)\right]=1
	\end{equation}
	for some $\varepsilon=o_\delta(1)$ with sufficiently slow decay. Let $\widetilde{U_\delta}$ denote the interior of $\overline{\mathcal D_1^\varepsilon\cup \mathcal D_2^\varepsilon}$ in the embedding of $\mathcal D_1\oplus\mathcal D_2$ and $U_\delta$ denote the interior of $\overline{\mathcal D_{1,\delta}^\varepsilon\cup \mathcal D_{2,\delta}^\varepsilon}$ in the embedding of $\mathcal D_{1,\delta}\oplus\mathcal D_{2,\delta}$. By conformal welding, the marked quantum surfaces $(\widetilde {U_\delta},\phi_{\mathbb D},0,i^{-})$ and $(U_\delta,\phi^\delta,0,i^{-})$ agree with probability $1-o_\delta(1)$. On this high probability event, there exists a unique conformal map $g_\delta:\widetilde{U_\delta}\to U_\delta$ such that $\phi_{\mathbb D}=\phi^\delta\circ g_\delta+Q\log|g_\delta'|$ with $g_\delta(0)=0$ and $g_\delta(i^-)=i^-$. 
\par
Notice that the random simply connected domain $\widetilde {U_\delta}$ is completely determined by $M_{\mathbb D}^{\#}$. Almost surely under $M_{\mathbb D}^{\#}$, the $\lbrace \overline{\mathbb D} \backslash \widetilde {U_\delta}\rbrace_\delta$ is a sequence of shrinking compact sets in the euclidean sense, i.e., $\diam(\overline{\mathbb D} \backslash \widetilde {U_\delta})=o_\delta(1)$ and $\bigcap_{\delta>0}\overline{\mathbb D}\backslash \widetilde{U_\delta}=\lbrace i\rbrace$. By the coupling between $M_{\mathbb D}^{\#}$ and $M_\delta^{\#}$, we know that $\diam(\overline{\mathbb D} \backslash \widetilde {U_\delta})=o_\delta(1)$ with probability $1-o_\delta(1)$. Notice that $\diam(\overline{\mathbb D}\backslash \widetilde {U_\delta})=0$ if and only if the harmonic measure of $\mathbb D\backslash \widetilde{U_\delta}$ viewed from $0$ in $\widetilde{U_\delta}$ tends to $0$ as $\delta\to 0$. Therefore, in our coupling, with probability $1-o_\delta(1)$, the harmonic measure of $\mathbb D\backslash \widetilde{U_\delta}$ viewed from $0$ in $\widetilde{U_\delta}$ is $o_\delta(1)$. Since the harmonic measure is conformally invariant and by $(\ref{coupling equation})$, with probability $1-o_\delta(1)$, harmonic measure of $\mathbb D\backslash U_\delta$ viewed from $0$ in $U_\delta$ is also $o_\delta(1)$. Hence, we have $\diam(\mathbb D\backslash U_\delta)=o_\delta(1)$ with probability $1-o_\delta(1)$. This proves $(2)$ in Lemma $\ref{Key coupling lemma}$.
\par
By construction, we know that $\widetilde{\mathbf x_\delta}\in \mathbb D\backslash U_\delta$ and $|\widetilde{\mathbf x_\delta}-i|\leq \diam\left(\mathbb D\backslash U_\delta \right) $. The above argument directly implies that  $|\widetilde {\mathbf x_\delta}-i|=o_\delta(1)$ with probability $1-o_\delta(1)$. Therefore $(3)$ is also proved. 
\par
Finally, by $(\ref{coupling equation})$, we have that $g_\delta(0)=0$, $g_\delta(i^{-})=i^{-}$, $\diam(\mathbb D\backslash U_\delta)=o_\delta(1)$, and $\diam(\mathbb D\backslash\widetilde{U_\delta})=o_\delta(1)$ with probability $1-o_\delta(1)$, the standard conformal distortion estimates imply $(4)$. 
\end{proof}

\begin{proof}[Proof of Proposition \ref{Curve Law Prop}]
For the convenience of readers, we first recall the definition and basic setup regarding  $M^{\#}_{\delta}$ on $\mathbb H$: For $W>0$, let $\beta_{W+2} = \gamma-\frac{W}{\gamma}$. Sample $(\phi,\mathbf x)$ from $\LF_{\mathbb H}^{(\gamma,i),(\beta_{W+2},\infty),(\beta_{W+2},\mathbf x)}\times dx$ and let $\eta$ be sampled from $\sle^{\mathbb H}_{\kappa,(\mathbf{x};\mathbf{x}^-)\to \infty}(W-2)$. Fix $\delta\in (0,\frac{1}{2})$ and let $M^{\#}_{\delta}$ be the probability law of $(\phi,\mathbf x, \eta)$ restricted to the event that $\nu_{\phi}(\mathbf x,\infty)\in (\delta, 2\delta)$, $\nu_{\phi}(\mathbb R)\in (1,2)$ and $i$ is to the right of $\eta$. Sample $(\phi,\mathbf x,\eta)$ from $M_\delta^{\#}$ and let $\mathcal D_{1,\delta}$ and $\mathcal D_{2,\delta}$ be the two components such that $(\mathbb H, \phi, \eta, i,\mathbf x)$ is the embedding of the conformally welded surface $\mathcal D_{1\,\delta}\oplus \mathcal D_{2,\delta}$.
\par
We first prove the results on $(\mathbb D,0,i)$ instead of $(\mathbb H,i,\infty)$. Let $f:\mathbb H\to \mathbb D$ be the conformal map such that $f(i)=0$ and $f(\infty)=i$. In the end, since both $M_\delta^{\#}$ and $M_{\mathbb D}^{\#}$ are probability laws, we can pull back all the results via $f^{-1}$. Let $\phi^{\delta}=\phi\circ f^{-1}+\log |(f^{-1})'|$ and $\eta^\delta=f\circ\eta$ be such that $(\mathbb D, \phi^\delta,\eta^\delta, 0,i)$ is an embedding of $\mathcal D_{1,\delta}\oplus \mathcal D_{2,\delta}$. Let $\widetilde {\mathbf x_\delta}=f(\mathbf x)$ be the image of $\mathbf{x}$ under $f$. Here $\eta^{\delta}$ represents the welding interface between $\mathcal D_{1,\delta}$ and $\mathcal D_{2,\delta}$.
\par
By Lemma \ref{Key coupling lemma}, there exists a coupling between $M_{\mathbb D}^{\#}$ and $M_\delta^{\#}$ such that 
\begin{equation}
		\lim_{\delta\to 0} \mathbb P\left[(\mathcal D_1^\varepsilon,\mathcal D_2^\varepsilon)=(\mathcal D_{1,\delta}^\varepsilon,\mathcal D_{2,\delta}^\varepsilon)\right]=1
	\end{equation}
for some $\varepsilon=o_\delta(1)$ with sufficiently slow decay (this is (\ref{coupling equation})).
Moreover, let $U_\delta$ be the interior of $\overline{\mathcal D_{1,\delta}^\varepsilon\cup \mathcal D_{2,\delta}^\varepsilon}\subset\mathbb D$ and let $\widetilde{U_{\delta}}$ be the interior of $\overline{\mathcal D_{1}^\varepsilon\cup \mathcal D_{2}^\varepsilon}\subset\mathbb D$. Then there exists a unique conformal map $g_{\delta}:\widetilde{U_\delta}\to U_{\delta}$ such that with probability $1-o_\delta(1)$,  $|\widetilde{\mathbf x_\delta}-i|=o_{\delta}(1)$ and $\sup_{z\in K}|g_\delta(z)-z|=o_\delta(1)$ for any compact set $K\subset\mathbb D$. Take $K=\overline{\mathcal{D}_1}\subseteq\mathbb D$ and by definition of $M_{\mathbb D}^{\#}$, $\eta_{\mathbb D}\subseteq\partial\overline{\mathcal{D}_1}$. The image of $\eta_{\mathbb D}$ under $g_\delta$ is $\eta^\delta\subset\partial\overline{\mathcal{D}_{1,\delta}}$. Since $\sup_{z}|g_\delta(z) - z| = o_\delta(1)$, there exist parametrizations $p_{\delta}:[0,1]\to \eta^\delta$ and $p_{\mathbb D}:[0,1]\to \eta_{\mathbb D}$ such that $|g_\delta( p_{\mathbb D}(t))-p_{\mathbb D}(t)|=|p_\delta(t)-p_{\mathbb D}(t)|=o_\delta(1)$ for all $t\in [0,1]$. Hence, under such coupling between $M^{\#}_{\mathbb D}$ and $M^{\#}_{\delta}$, with probability $1-o_\delta(1)$, there exist parametrizations $p_\delta$ and $p_{\mathbb D}$ of $\eta^\delta$ and $\eta_{\mathbb D}$ respectively, such that $\sup_{t\in [0,1]}|p_\delta(t)-p_{\mathbb D}(t)|=o_\delta(1)$, which implies the topology of convergence under coupling is the same as (\ref{metric on curve}).
\par
Next, by Lemma \ref{Key coupling lemma}, $|\widetilde{\mathbf x_\delta}-i|=o_{\delta}(1)$ with probability $1-o_\delta(1)$, and for any instance of $\widetilde{\mathbf x_\delta}$, $\eta^\delta$ has the law of $\sle^{\mathbb D}_{\kappa,(\widetilde{\mathbf x_\delta};\widetilde{\mathbf x_\delta}^+)\to i}(W-2)[\cdot|0\in \mathcal{D}_{2,\delta}]$. By Corollary \ref{conditional convergence lemma}, for any deterministic sequence $\mathbf{x}_\delta$ on $\partial\mathbb D$ that converges to $i$ in euclidean distance as $\delta\to 0$, 
\begin{equation}
    \sle^{\mathbb D}_{\kappa,(\mathbf{x}_\delta;\mathbf{x}_\delta^+)\to i}(W-2)[\cdot|0\in \mathcal{D}_{2,\delta}]\xrightarrow{w} \sle_{\kappa,i}^{\bub}(W-2)[\cdot|0\in \mathcal{D}_2]
\end{equation}
in the distance (\ref{metric on curve}). Hence , under $M_{\mathbb D}^{\#}$, $\eta_{\mathbb D}$ is independent of $\phi_{\mathbb D}$ and has the law of $\sle_{\kappa,i}^{\bub}(W-2)[\cdot|0\in \mathcal{D}_2]$.
By pulling back all the results above on $\mathbb D$ to $\mathbb H$ via $f^{-1}$, we have that
\begin{equation}\label{restricted integration}
   \int_1^2 \phi(\ell)d\ell \times \sle_{\kappa,0}^{\bub}(W-2)[\cdot|{i\in D_\eta(0)}] = \int_1^2\int_0^\infty \mathcal{M}^{\disk}_{0,2}(W;\ell,r)\times \qd_{1,1}(r) dr d\ell   
\end{equation}
for some unknown Liouville field $\phi$. Finally, by the identical scaling argument in the proof of \cite[Theorem $4.1$]{FZZ}, the integration on $[1,2]$ in (\ref{restricted integration}) can be replaced by $(0,\infty)$. This completes the proof.
\end{proof}
\section{Law of field after conformal welding via induction}\label{sec:field law}
\subsection{Preliminaries on quantum triangles}
Our derivation of field law relies heavily on the conformal welding of quantum triangle with quantum disk. In this section, we recall the definition of quantum triangle and review the welding theorem between quantum triangle and quantum disk (\cite{QT22}).
\begin{Def}[Thick quantum triangle, {\cite[Definition 2.17]{QT22}}]\label{Thick quantum triangle}
    For $W_1,W_2,W_3>\frac{\gamma^2}{2}$, set $\beta_{i}=\gamma+\frac{2-W_i}{\gamma}<Q$ for $i=1,2,3,$ and let $\LF_{\mathcal S}^{(\beta_1,+\infty),(\beta_2,-\infty),(\beta_3,0)}$ be the Liouville field on $\mathcal S$ with insertion $\beta_1,\beta_2,\beta_3$ at $+\infty,-\infty$ and $0$, respectively. Let $\phi$ be sampled from 
    \begin{equation*}
        \frac{1}{(Q-\beta_1)(Q-\beta_2)(Q-\beta_3)}\LF_{\mathcal S}^{(\beta_1,+\infty),(\beta_2,-\infty),(\beta_3,0)}.
    \end{equation*}
    Define $\qt(W_1,W_2,W_3)$ to be the law of the three-pointed quantum surface $(\mathcal S,\phi,+\infty,-\infty,0)/\sim_{\gamma}$ and we call a sample from $\qt(W_1,W_2,W_3)$ a quantum triangle of weight $(W_1,W_2,W_3)$.
\end{Def}
One can also define the conditional law of quantum disks/triangles on fixed boundary length. This is again done by disintegration.
\begin{Def}[{\cite[Definition 2.26]{QT22}}]\label{QT with fixed boundary length}
Fix $W_1,W_2,W_3>\frac{\gamma^2}{2}$. Let $\beta_{i} = \gamma + \frac{2-W_i}{\gamma}$ and $\overline{\beta}=\beta_{1}+\beta_{2}+\beta_{3}$. Sample $h$ from $P_{\mathbb H}$ and set 
\begin{equation*}
    \widetilde{h(z)} = h(z) + (\overline{\beta}-2Q)\log|z|_{+} -\beta_{1}\log|z|-\beta_{2}\log|z-1|.
\end{equation*}
Fix $\ell>0$ and let $L_{12} = \nu_{\widetilde{h}}([0,1])$. We define $\qt(W_1,W_2,W_3;\ell)$, the quantum triangles of weights $W_1,W_2,W_3$ with left boundary length $\ell$, to be the law of $\widetilde{h}+\frac{2}{\gamma}\log \frac{\ell}{L_{12}}$ under the reweighted measure $\frac{2}{\gamma}\frac{\ell^{\frac{1}{\gamma}(\overline{\beta}-2Q)-1}}{L_{12}^{\frac{1}{\gamma}(\overline{\beta}-2Q)}}P_{\mathbb H}(dh)$. The same thing holds if we replace $L_{12} = \nu_{\widetilde{h}}([0,1])$ by $L_{13} = \nu_{\widetilde{h}}((-\infty,0])$ or $L_{23} = \nu_{\widetilde{h}}([1,+\infty))$.
\end{Def}
\begin{Lemma}[{\cite[Lemma 2.27]{QT22}}]
    In the same settings of Definition \ref{QT with fixed boundary length}, the sample from $\qt(W_1,W_2,W_3;\ell)$ has left boundary length $\ell$, and we have 
    \begin{equation}\label{disint of QT}
        \qt(W_1,W_2,W_3) = \int_{0}^\infty \qt(W_1,W_2,W_3;\ell)d\ell .
    \end{equation}
\end{Lemma}
Let $\sle_{\kappa}(\rho_{-};\rho_{+},\rho_1)$ be the law of a chordal $\sle_{\kappa}$ on $\mathbb H$ from $0$ to $\infty$ with force points $0^{-},0^{+},1$, with corresponding weights $\rho_{-},\rho_{+},\rho_1$ respectively. Moreover, suppose $\eta$ is a curve from $0$ to $\infty$ on $\mathbb H$ that does not touch $1$. Let $D_\eta$ be the connected component of $\mathbb H\backslash \eta$ containing $1$ and $\psi_\eta$ is the unique conformal map from the component $D_\eta$ to $\mathbb H$ fixing $1$ and sending the first (resp. last) point on $\partial D_\eta$ hit by $\eta$ to $0$ (resp. $\infty$). Define the measure $\widetilde{\sle_{\kappa}}(\rho_{-};\rho_{+},\rho_1;\alpha)$ on curves from $0$ to $\infty$ on $\mathbb H$ as follows:
\begin{equation}
    \frac{d\widetilde{\sle_{\kappa}}(\rho_{-};\rho_{+},\rho_1;\alpha)}{d\sle_{\kappa}(\rho_{-};\rho_{+},\rho_1)}(\eta) = \psi_{\eta}'(1)^\alpha.
\end{equation}
\begin{Th}[{\cite[Theorem 1.2]{QT22}}]\label{qt+qd,v2}
    Suppose $W,W_1,W_2,W_3>0$ and $\frac{\gamma^2}{2}\notin\lbrace W_1,W_2,W_3,W+W_1,W+W_2\rbrace$. Let 
    \begin{equation}
        \alpha = \frac{W_3+W_2-W_1-2}{4\kappa}\left(W_3+W_1+2-W_2-\kappa\right).
    \end{equation}
    Then there exist some constant $C = C_{W,W_1,W_2}\in (0,\infty)$ such that
    \begin{equation}
        \qt(W+W_1,W+W_2,W_3)\otimes \widetilde{\sle_{\kappa}}(W-2;W_2-2,W_1-W_2;\alpha) = C\int_{0}^{\infty}\mathcal M^{\disk}_{0,2}(W;\ell)\times \qt(W_1,W_2,W_3;\ell)d\ell. 
    \end{equation}
\end{Th}
\subsection{Quantum disks with generic bulk and boundary insertions}
\begin{Def}[Special case of Definition \ref{Def one bulk one boundary}]\label{Def: one bulk one boundary}
	Let $\alpha,\beta\in \mathbb R$. Fix $p\in \mathbb R$ and $q\in \mathbb H$. Suppose $(h,\mathbf c)$ is sampled from $C_{\mathbb H}^{(\beta,p),(\alpha,q)} P_{\mathbb H}\times \left[e^{\left(\frac{1}{2}\beta+\alpha-Q\right)c}dc \right]$, where 
	\begin{equation*}
		C_{\mathbb H}^{(\beta,p),(\alpha,q)}=|p|_{+}^{-\beta(Q-\frac{\beta}{2})}(2\Im q)^{-\frac{\alpha^2}{2}}|q|_{+}^{-2\alpha(Q-\alpha)}.
	\end{equation*}
	Then the field $\phi(z)=h(z)-2Q\log|z|_{+}+\alpha G_{\mathbb H}(z,q)+\frac{\beta}{2}G_{\mathbb H}(z,p)+\mathbf c$ has the law of $\LF_{\mathbb H}^{(\beta,p),(\alpha,q)}$. Moreover, If $p=\infty$, let $(h,\mathbf c)$ be sampled from $C_{\mathbb H}^{(\beta,\infty),(\alpha,q)}P_{\mathbb H}\times \left[e^{\left(\frac{1}{2}\beta+\alpha-Q\right)c}dc \right]$, where 
	\begin{equation*}
		C_{\mathbb H}^{(\beta,\infty),(\alpha,q)}=(2\Im q)^{-\frac{\alpha^2}{2}}|q|_{+}^{-2\alpha(Q-\alpha)}.
	\end{equation*}
	Let $\phi_{\infty}(z)=h(z)+(\beta-2Q)\log|z|_{+}+\alpha G_{\mathbb H}(z,q)+\mathbf c$ and $\phi_{\infty}$ has the law of $\LF_{\mathbb H}^{(\beta,\infty),(\alpha,p)}$.
\end{Def}
\begin{Prop}[{\cite[Proposition 3.9]{FZZ}}]\label{LCFT des of QD_{1,1}}
    Suppose $(\mathbb H,\phi,i,0)$ is an embedding of $\qd_{1,1}$, then $\phi$ has the law of $C_{0}\cdot\LF_{\mathbb H}^{(\gamma,i),(\gamma,0)}$ for some fixed finite constant $C_{0}$.
\end{Prop}
\begin{Def}\label{def:general qd}  
Fix $\alpha,\beta\in \mathbb R$. Define the quantum surface $\qd_{1,1}(\alpha,\beta)$ as follows: suppose $(\mathbb H,\phi,i,0)$ is an embedding of $\qd_{1,1}(\alpha,\beta)$, then the law of $\phi$ is $\LF_{\mathbb H}^{(\alpha,i),(\beta,0)}$. Notice that $\qd_{1,1}(\gamma,\gamma) = C\cdot\qd_{1,1}$ for some finite constant $C$.
\end{Def}
\begin{Lemma}
    Fix $\alpha,\beta\in \mathbb R$ and let $h$ be sampled from $P_{\mathbb H}$. Let $\widetilde{h}(z) = h(z) + \alpha G_{\mathbb H}(z,i)+\frac{\beta}{2}G_{\mathbb H}(z,0)-2Q\log|z|_{+}$ and $L=\nu_{\widetilde{h}}(\mathbb R)$. Let $\LF_{\mathbb H}^{(\alpha,i),(\beta,0)}(\ell)$ be the law of $\widetilde{h}+\frac{2}{\gamma}\log\frac{\ell}{L}$ under the reweighted measure $2^{-\alpha^2/2}\frac{2}{\gamma}\frac{\ell^{\frac{2}{\gamma}(\alpha+\frac{\beta}{2} -Q)-1}}{L^{\frac{2}{\gamma}(\alpha+\frac{\beta}{2}-Q)}} P_{\mathbb H}$, and let $\qd_{1,1}(\alpha,\beta;\ell)$ be the measure on quantum surfaces $(\mathbb H,\phi,0,i)$ with $\phi$ being sampled from $\LF^{(\alpha,i),(\beta,0)}_{\mathbb H}(\ell)$. Then $\qd_{1,1}(\alpha,\beta;\ell)$ is a measure on quantum surfaces with (quantum) boundary length $\ell$, and 
    \begin{equation}\label{disint of general QD_{1,1}}
        \LF_{\mathbb H}^{(\alpha,i),(\beta,0)} = \int_{0}^{\infty}  \LF_{\mathbb H}^{(\alpha,i),(\beta,0)}(\ell)d\ell\qquad\text{and}\qquad \qd_{1,1}(\alpha,\beta) = \int_{0}^{\infty}\qd_{1,1}(\alpha,\beta;\ell)d\ell. 
    \end{equation}
\end{Lemma}
\begin{proof}
Suppose $\phi$ has the law of $\widetilde{h}+\frac{2}{\gamma}\log\frac{\ell}{L}$, then we have 
\begin{align}
    \begin{split}
        \nu_{\phi}(\mathbb R) = \int_{\mathbb R} e^{\frac{\gamma}{2}\phi(x)}dx = \frac{\ell}{L}\int_{\mathbb R} e^{\frac{\gamma}{2}\widetilde{h(x)}}dx = \ell.
    \end{split}
\end{align}
Therefore, we have $\nu_{\phi}(\mathbb R)=\ell$ almost surely under $\LF_{\mathbb H}^{(\alpha,i),(\beta,0)}(\ell)$. Moreover, for any non-negative measurable function $F$ on $H^{-1}(\mathbb H)$, we have 
    \begin{equation}
      \int_0^{\infty}\int F\left(\widetilde{h}+\frac{2}{\gamma}\log\frac{\ell}{L}\right)2^{-\alpha^2/2}\frac{2}{\gamma}\frac{\ell^{\frac{2}{\gamma}(\alpha+\frac{\beta}{2} -Q)-1}}{L^{\frac{2}{\gamma}(\alpha+\frac{\beta}{2}-Q)}} P_{\mathbb H}(dh)d\ell =\int\int_{-\infty}^{\infty} F(\widetilde{h}+c)2^{-\alpha^2/2}e^{(\alpha+\frac{\beta}{2}-Q)c}dc P_{\mathbb H}(dh)
    \end{equation}
    by Fubini's theorem and change of variable $c=\frac{2}{\gamma}\log\frac{\ell}{L}$. This matches the field law in Definition \ref{Def one bulk one boundary}. Hence (\ref{disint of general QD_{1,1}}) is proved. 
\end{proof}
\begin{Def}\label{Def: general QD}
    Fix $\alpha\in \mathbb R$ and let $(\mathbb H,\phi,i,0)$ be an embedding of $\qd_{1,1}(\gamma,\alpha)$. Let $L=\nu_{\phi}(\mathbb R)$ denote the total quantum boundary length and $A = \mu_{\phi}(\mathbb H)$ denote the total quantum area. Let $\qd_{0,1}(\gamma,\alpha)$ be the law of $(\mathbb H,\phi,0)$ under the reweighted measure $A^{-1}\qd_{1,1}(\gamma,\alpha)$. For integers $n\geq 0$ and $m\geq 1$, let $(\mathbb H, \phi)$ be sampled from the re-weighted measure $A^{n}L^{m-1}\qd_{0,1}(\gamma,\alpha)$, then independently sample $\omega_1,\ldots,\omega_{m-1}$ and $z_1,\ldots,z_n$ according to $\nu_{\phi}^{\#}$ and $\mu_{\phi}^{\#}$ respectively. Let $\qd_{n,m}(\gamma,\alpha)$ denote the law of $(\mathbb H,\phi,0,\omega_1,\ldots,\omega_{m-1},z_1,\ldots,z_n)$ viewed as a measure on equivalence class $\mathscr{D}_{n,m}$.
\end{Def}
\begin{figure}[H]
	\includegraphics[scale=0.6]{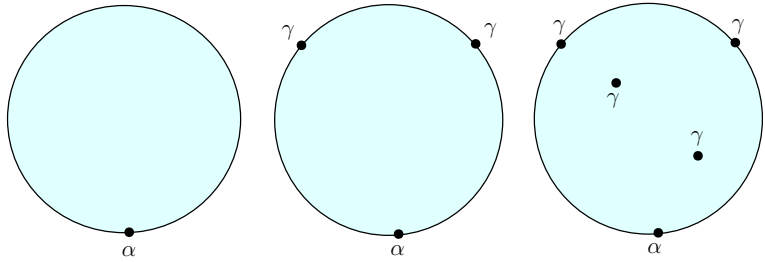}
	\centering
	\caption{Left: $\qd_{0,1}(\gamma,\alpha)$ Middle: $\qd_{0,3}(\gamma,\alpha)$ Right: $\qd_{2,3}(\gamma,\alpha)$}
	\label{fig:Def_general_qd}
\end{figure}
More generally, for fixed $\ell_1,\ldots,\ldots,\ell_m$, like in \cite[Section 2.6]{Weldisk}, we can define the measure $\qd_{1,m}(\gamma,\alpha)(\ell_1,\ell_2,\ldots,\ell_m)$ using disintegration and it satisfies
\begin{equation}
    \qd_{1,m}(\gamma,\alpha) =\int_0^\infty\ldots \int_{0}^\infty \qd_{1,m}(\gamma,\alpha;\ell_1,\ldots,\ell_m)d\ell_1\ldots d\ell_m.
\end{equation}
\subsection{Conformal welding of thin and thick disks}
\begin{Lemma}\label{connection between QT and QD alpha}
For $W>\frac{\gamma^2}{2}$, let $\beta_W = \gamma+\frac{2-W}{\gamma}<Q$. Then we have 
\begin{equation}
    \qd_{0,3}(\gamma,\beta_W) =C\cdot \qt(2,2,W)
\end{equation}
for some finite constant $C$.
\end{Lemma}
\begin{proof}
   After applying \cite[Lemma 2.31]{InteofSLE} twice, we have 
   \begin{equation}
       \LF_{\mathbb H}^{(\beta_W,0)}(d\phi)\nu_{\phi}(dx)\nu_{\phi}(dy) = \LF_{\mathbb H}^{(\beta_W,0),(\gamma,x),(\gamma,y)}(d\phi)dxdy.
   \end{equation}
   By disintegration, we can fix an embedding of $\qd_{0,3}(\gamma,\alpha)$ to be $(\mathbb H,\phi,-1,0,1)$ so that $\phi$ has the law of $C\cdot\LF_{\mathbb H}^{(\beta_W,0),(\gamma,-1),(\gamma,1)}$ for some finite constant $C$. Let $f:\mathbb H\to \mathcal S$ be the conformal map such that $f(-1)=-\infty,f(1)=\infty$ and $f(0)=0$. Therefore, by Definition \ref{Thick quantum triangle}, it has the law of $\qt(2,2,W)$ under push-forward of $f$. This completes the proof.
\end{proof}
\begin{Lemma}\label{Girsanov on bulk insertion}
	Recall $\LF_{\mathbb H}^{(\beta_i,z_i)_i}$ from Definition \ref{LCFT with boundary insertions}. We have
	\begin{equation}
		\LF_{\mathbb H}^{(\beta_i,z_i)_i}\left[f(\phi)\int_{\mathbb H}g(u)\mu_{\phi}(du) \right]=\LF_{\mathbb H}^{(\beta_i,z_i)_i,(\gamma,u)}[f(\phi)]g(u)d^2u
	\end{equation}
	for non-negative measurable functions $f$ and $g$. 
\end{Lemma}
\begin{proof}
	The proof is identical to that of \cite[Lemma 2.33]{InteofSLE} with $\widehat{\mathbb C}$ replaced by $\mathbb H$. 
\end{proof}
Next we recall the decomposition theorem of thin quantum disk with one additional typical boundary marked point that is crucial to our derivation of the field law. 
\begin{Lemma}[{\cite[Proposition 4.4]{Weldisk}}]\label{Lemma decomp of thin disk}
    For $W\in (0,\frac{\gamma^2}{2})$, we have
    \begin{equation}\label{decomp of thin disk with marked pt}
        \mathcal M^{\disk}_{2,\bullet}(W) = \left(1-\frac{2}{\gamma^2}W \right)^2\mathcal M^{\disk}_{0,2}(W)\times \mathcal M^{\disk}_{2,\bullet}(\gamma^2-W)\times\mathcal M^{\disk}_{0,2}(W).
    \end{equation}
\end{Lemma}
\begin{figure}[H]
	\includegraphics[scale=0.8]{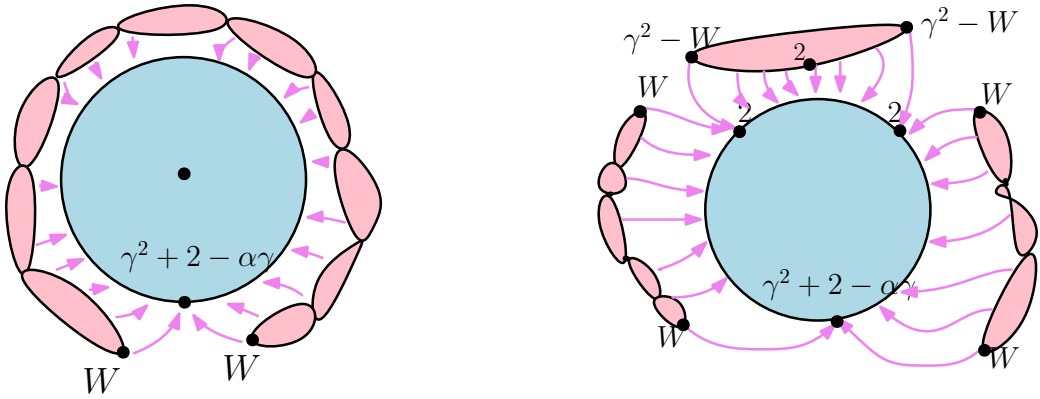}
	\centering
	\caption{When $0<W<\frac{\gamma^2}{2}$, welding $\qd_{1,1}(\gamma,\alpha)$ with $\mathcal M^{\disk}_{0,2}(W)$ is equivalent to first welding $\qd_{0,3}(\gamma,\alpha)$ with three independent quantum disks $\mathcal M^{\disk}_{0,2}(W)$, $\mathcal M^{\disk}_{2,\bullet}(\gamma^2-W)$ and $\mathcal M^{\disk}_{0,2}(W)$ separately then de-weighting all the three additional boundary marked points and sampling an bulk marked point in the blue region according to quantum area measure.}
	\label{fig:weld_thin_disks}
\end{figure}
\begin{Prop}\label{thin+qd11: field law}
    Fix $0<\gamma<2$ and $0<W<\frac{\gamma^2}{2}$. For $\alpha\leq \gamma<Q$, let $W_\alpha = 2-(\alpha-\gamma)\gamma\geq 2>\frac{\gamma^2}{2}$. Let $(\mathbb H,\phi,\eta, 0,i)$ be an embedding of 
    \begin{equation}
        \int_{0}^{\infty}\qd_{1,1}(\gamma,\alpha;\ell)\times \mathcal M_{0,2}^{\disk}(W;\ell)d\ell.
    \end{equation}
    Then $\phi$ has the law of $C\cdot\LF_{\mathbb H}^{(\beta_{2W+W_\alpha},0),(\gamma,i)}$ for some finite constant $C$. Notice that $\alpha=\beta_{W_\alpha}=Q+\frac{\gamma}{2}-\frac{W_\alpha}{\gamma}$. 
\end{Prop}
\begin{proof}
    Fix $0<W<\frac{\gamma^2}{2}$ and $\alpha\leq \gamma$. Start with the following four quantum surfaces: 
    \begin{equation}
        \qd_{0,3}(\gamma,\alpha),\mathcal M^{\disk}_{0,2}(W), \mathcal M^{\disk}_{0,2}(\gamma^2-W)\ \text{and} \ \mathcal M^{\disk}_{0,2}(W).
    \end{equation}
    Notice that $\qd_{0,3}(\gamma,\alpha)$ has one $\alpha$ insertion and two $\gamma$ insertions along its boundary. First, weld two $\mathcal{M}^{\disk}_{0,2}(W)$ disks along the boundaries of $\qd_{0,3}(\gamma,\alpha)$ with $\gamma$ and $\alpha$ insertions, then weld $\mathcal{M}^{\disk}_{0,2}(\gamma^2-W)$ along the boundary of $\qd_{0,3}(\gamma,\alpha)$ with two $\gamma$ insertions. Precisely, we consider
    \begin{align}\label{first step de-weighting}
        \begin{split}
        &\weld \left(\qd_{0,3}(\gamma,\alpha), \mathcal M^{\disk}_{0,2}(W)\times \mathcal M^{\disk}_{0,2}(\gamma^2-W)\times \mathcal M^{\disk}_{0,2}(W) \right)\\
            &:=\int_{0}^{\infty}\left(\iint_0^\infty \mathcal M^{\disk}_{0,2}(W;\ell_1) \qd_{0,3}(\gamma,\alpha;\ell_1,\ell_2,\ell_3)\mathcal M^{\disk}_{0,2}(W;\ell_3) d\ell_1 d\ell_3\right)\mathcal M^{\disk}_{0,2}(\gamma^2-W;\ell_2)d\ell_2 \\
            &= \int_{0}^{\infty}\left(\iint_0^\infty \mathcal M^{\disk}_{0,2}(W;\ell_1) \qd_{0,3}(\gamma,\alpha;\ell_1,\ell_2,\ell_3)\mathcal M^{\disk}_{0,2}(W;\ell_3) d\ell_1 d\ell_3\right)\ell_2^{-1} \mathcal M^{\disk}_{2,\bullet}(\gamma^2-W;\ell_2)d\ell_2\\
            &=L_2^{-1}\cdot\weld(\qd_{0,3}(\gamma,\alpha),\mathcal M^{\disk}_{0,2}(W)\times\mathcal M^{\disk}_{2,\bullet}(\gamma^2-W)\times\mathcal M^{\disk}_{0,2}(W))\\
            &= \left(1-\frac{2}{\gamma^2}W\right)^{-2}\cdot L_2^{-1}\cdot\weld\left(\qd_{0,3}(\gamma,\alpha),\mathcal M^{\disk}_{2,\bullet}(W) \right),
        \end{split}
    \end{align}
    where $L_2$ denotes the quantum length of welding interface between $\qd_{0,3}(\gamma,\alpha)$ and $\mathcal{M}^{\disk}_{0,2}(\gamma^2-W)$ and 
    \begin{equation}\label{03+2dot}
        \weld\left(\qd_{0,3}(\gamma,\alpha),\mathcal M^{\disk}_{2,\bullet}(W) \right) := \int_{0}^\infty \qd_{0,3}(\gamma,\alpha;\ell)\times \mathcal M^{\disk}_{2,\bullet}(W;\cdot,\ell)d\ell. 
    \end{equation}
    In (\ref{03+2dot}), $\qd_{0,3}(\gamma,\alpha;\ell)$ represents the $\qd_{0,3}(\gamma,\alpha)$ conditioning on having total boundary length $\ell$ and $\mathcal M^{\disk}_{2,\bullet}(W;\cdot,\ell)$ represents the $\mathcal M^{\disk}_{2,\bullet}(W)$ conditioning on having left boundary length $\ell$. By de-weighting all the three marked points on the welding interface and sampling an additional bulk marked points in the inner region of (\ref{03+2dot}), we have
    \begin{align}\label{second step de-weighting}
        \begin{split}
            \weld\left(\qd_{0,3}(\gamma,\alpha),\mathcal M^{\disk}_{2,\bullet}(W) \right) &= \int_{0}^\infty \qd_{0,3}(\gamma,\alpha;\ell)\times \mathcal M^{\disk}_{2,\bullet}(W;\ell)d\ell\\
            &=\int_{0}^\infty\ell^2\cdot \qd_{0,1}(\gamma,\alpha;\ell)\times \mathcal M^{\disk}_{2,\bullet}(W;\ell)d\ell\\
            &= \int_{0}^{\infty} \ell^3\cdot \qd_{0,1}(\gamma,\alpha;\ell)\times\mathcal M^{\disk}_{0,2}(W;\ell) d\ell\\
            & = L_T^3\cdot\weld(\qd_{0,1}(\gamma,\alpha),\mathcal M^{\disk}_{0,2}(W))\\
            &=L_T^3\cdot A_I^{-1}\cdot \weld(\qd_{1,1}(\gamma,\alpha),\mathcal M^{\disk}_{0,2}(W)),
        \end{split}
    \end{align}
    where $L_T$ denotes the quantum length of the total welding interface and $A_I$ denotes the quantum area of $\qd_{0,1}(\gamma,\alpha)$. Hence, by (\ref{first step de-weighting}), (\ref{second step de-weighting}), we have
    \begin{align}\label{final formula}
        \begin{split}
            &\weld \left(\qd_{0,3}(\gamma,\alpha), \mathcal M^{\disk}_{0,2}(W)\times \mathcal M^{\disk}_{0,2}(\gamma^2-W)\times \mathcal M^{\disk}_{0,2}(W) \right)\\
            &=\left(1-\frac{2}{\gamma^2}W\right)^{-2}\cdot L_2^{-1}\cdot L_T^3\cdot A_I^{-1}\cdot \weld(\qd_{1,1}(\gamma,\alpha),\mathcal M^{\disk}_{0,2}(W)).
        \end{split}
    \end{align}
    By applying Theorem \ref{qt+qd,v2} three times, we know that suppose $(\mathbb H,\phi,\eta_1,\eta_2,\eta_3,0,1,-1)$ is an embedding of 
    \begin{equation*}
        \weld \left(\qd_{0,3}(\gamma,\alpha), \mathcal M^{\disk}_{0,2}(W)\times \mathcal M^{\disk}_{0,2}(\gamma^2-W)\times \mathcal M^{\disk}_{0,2}(W) \right),
    \end{equation*}
    then $\phi$ is independent of $(\eta_1,\eta_2,\eta_3)$ and has the law of $C\cdot\LF_{\mathbb H}^{(\beta_{2W+W_\alpha},0),(0,-1),(0,1)}$ for some finite constant $C$. Here we emphasize the fact that weights of insertions $-1$ and $1$ are both zero due to the computation 
    \begin{equation*}
    	\beta_{2+W+(\gamma^2-W)}=\beta_{2+\gamma^2}=0,
    \end{equation*}
    where the $2$ comes from the insertion $\gamma$ on $\qd_{0,3}(\gamma,\alpha)$, the $W$ comes from $\mathcal M^{\disk}_{0,2}(W)$ and the $\gamma^2-W$ comes from  $\mathcal M_{0,2}^{\disk}(\gamma^2-W)$. Finally, by quantum surface relationship (\ref{final formula}) and Lemma \ref{Girsanov on bulk insertion}, we know that suppose $(\mathbb H,\phi, 0, i, \eta)$  is an embedding of $\weld\left(\qd_{1,1}(\gamma,\alpha),\mathcal M^{\disk}_{0,2}(W) \right)$, then $\phi$ has the law of $C\cdot\LF_{\mathbb H}^{(\beta_{2W+W_\alpha},0),(\gamma,i)}$ for some finite constant $C$. 
\end{proof}
\subsection{Proof of Theorem \ref{Conditional Welding}}
In this section, we prove Theorem \ref{Conditional Welding} by inductively welding thin disks along the $\qd_{1,1}$.
\begin{figure}[H]
	\includegraphics[scale=0.8]{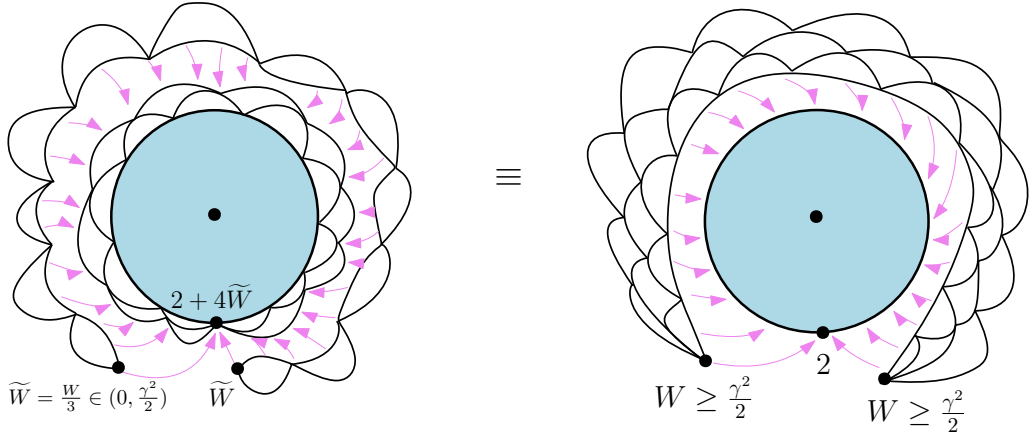}
	\centering
	\caption{Illustration of the induction procedure in the proof of Theorem \ref{Conditional Welding}: suppose $W\geq\frac{\gamma^2}{2}$ and $\widetilde{W}= \frac{W}{3}\in (0,\frac{\gamma^2}{2})$, then welding a thick quantum disk $\mathcal M^{\disk}_{0,2}(W)$ is equivalent to welding three thin quantum disks $\mathcal M^{\disk}_{0,2}(\widetilde{W})$. Notice that here we only care about the law of the underlying random field.}
	\label{fig:Induction}
\end{figure}
 
\begin{proof}[Proof of Theorem \ref{Conditional Welding}]
	By Proposition \ref{Curve Law Prop}, we have the correct curve law and know that the curve law is independent of the underlying random field. Therefore, it remains to derive the field law. Fix $0<\gamma<2$ and $0<W<\frac{\gamma^2}{2}$. For $\alpha\leq \gamma$, let $W_\alpha = 2-(\alpha-\gamma)\gamma\geq 2$. Let $(\mathbb H,\phi,\eta, 0,i)$ be an embedding of quantum surface
    \begin{equation}\label{4.4.1}
        \int_{0}^{\infty}\qd_{1,1}(\gamma,\alpha;\ell)\times \mathcal M_{0,2}^{\disk}(W;\ell)d\ell.
    \end{equation}
    By Proposition \ref{thin+qd11: field law}, $\phi$ has the law of $C\cdot\LF_{\mathbb H}^{(\beta_{2W+W_\alpha},0),(\gamma,i)}$ for some finite constant $C$. Therefore, in order to prove the Theorem \ref{Conditional Welding}, we only need to extend the range of $W$ from $(0,\frac{\gamma^2}{2})$ to $(0,\infty)$. For any $W\geq\frac{\gamma^2}{2}$, there exists some integer $n\geq 2$ such that $\widetilde{W}=\frac{W}{n}\in (0,\frac{\gamma^2}{2})$. Moreover, by Theorem \ref{Disk welding theorem}, we have
    \begin{align}
    	\begin{split}
    		&\int_{0}^{\infty}\qd_{1,1}(\ell)\times \mathcal M_{0,2}^{\disk}(W;\ell)d\ell\\
    		&=\int_{0}^{\infty}\qd_{1,1}(\ell)\underbrace{\mathcal M_{0,2}^{\disk}(\widetilde{W};\ell,\ell_1)\mathcal M_{0,2}^{\disk}(\widetilde{W};\ell_1,\ell_2) \ldots\mathcal M_{0,2}^{\disk}(\widetilde{W};\ell_{n-2},\ell_{n-1}) \mathcal M_{0,2}^{\disk}(\widetilde{W};\ell_{n-1},\ell_n)}_{n \ \text{thin disks}} d\vec{\ell}\\
    		&=\int_0^\infty\left(\left(\left(\left(\qd_{1,1}(\ell)\mathcal M_{0,2}^{\disk}(\widetilde{W};\ell,\ell_1)\right)\mathcal M_{0,2}^{\disk}(\widetilde{W};\ell_1,\ell_2)\right)\ldots  \mathcal M_{0,2}^{\disk}(\widetilde{W};\ell_{n-2},\ell_{n-1})\right) \mathcal M_{0,2}^{\disk}(\widetilde{W};\ell_{n-1},\ell_n)\right)d\vec{\ell},
    	\end{split}
    \end{align}
    where $d\vec{\ell}=d\ell d\ell_1,\ldots,d\ell_n$. Notice that $\qd_{1,1}=C\cdot\qd_{1,1}(\gamma,\gamma)$ by definition and $W_\gamma=2$. By applying Proposition \ref{thin+qd11: field law} $n$ times from the inner bracket to outer bracket, we have that suppose $(\mathbb H,\phi,\eta,0,i)$ is an embedding of $(\ref{4.4.1})$, then $\phi$ has the law of $C\cdot\LF_{\mathbb H}^{(\gamma,i),(\beta_{2+2n\widetilde{W}},0)}$, which is the same as $C\cdot\LF_{\mathbb H}^{(\gamma,i),(\beta_{2+2W},0)}$ for some finite constant $C$. This completes the proof. 
\end{proof}
\section{Proof of Theorem \ref{Main welding theorem} via uniform embeddings of quantum surfaces}\label{sec:proof of main based on conditional welding}
\subsection{Uniform embedding of quantum surfaces}
To start, recall that $\conf(\mathbb H)$ is the group of conformal automorphisms of $\mathbb H$ where group multiplication $\cdot$ is the function composition $f\cdot g = f\circ g$. Let $\mathbf m_{\mathbb H}$ be a Haar measure on $\conf(\mathbb H)$, which is both left and right invariant. Suppose $\mathfrak{f}$ is sampled from $\mathbf m_{\mathbb H}$ and $\phi\in H^{-1}(\mathbb H)$, then we call the random function $\mathfrak{f}\bullet_\gamma\phi=\phi\circ \mathfrak{f}^{-1}+Q|\log(\mathfrak{f}^{-1})'|$ the \textit{uniform embedding} of $(\mathbb H,\phi)$ via $\mathbf{m}_{\mathbb H}$. By invariance property of Haar measure, the law of $\mathfrak{f}\bullet_\gamma\phi$ only depends on $(\mathbb H,\phi)$ as quantum surface. Let $(z_i)_{1\leq i\leq n}\in \mathbb H, (s_j)_{1\leq j\leq m}\in \partial\mathbb H$ be groups of bulk and boundary marked points respectively. Suppose $(\mathbb H,h,z_1\ldots,z_n,s_1,\ldots,s_m)$ is a marked quantum surface, then we call $\mathbf{m}_{\mathbb H}\ltimes (\mathbb H,h,z_1\ldots,z_n,s_1,\ldots,s_m)$ the \textit{uniform embedding} of $(\mathbb H,h,z_1\ldots,z_n,s_1,\ldots,s_m)$ via $\mathbf{m}_{\mathbb H}$. 
\begin{Lemma}[{\cite[Lemma $3.7$]{FZZ}}]\label{decomp of Haar measure}
    Define three measures $A,N,K$ on the conformal automorphism group $\conf(\mathbb H)$ on $\mathbb H$ as follows. Sample $\mathbf t$ from $1_{t>0}\frac{1}{t}dt$ and let $a:z\mapsto \mathbf{t}z$. Sample $\mathbf s$ from Lebesgue measure on $\mathbb R$ and let $n:z\mapsto z+ \mathbf s$. Sample $\mathbf u$ from $\mathbbm{1}_{\lbrace -\frac{\pi}{2}<u<\frac{\pi}{2}\rbrace}du$ and let $k:z\mapsto \frac{z\cos\mathbf u-\sin\mathbf u}{z\sin\mathbf u+\cos\mathbf u}$. Let $A,N,K$ be the law of $a,n,k$ respectively, then the law of $a\circ n\circ k$ under $A\times N\times K$ is equal to $\mathbf{m}_{\mathbb H}$.
\end{Lemma}
\begin{Lemma}\label{Joint law}
	Suppose $\mathfrak{f}$ is sampled from $\mathbf m_{\mathbb H}$, then the joint law of $(\mathfrak{f}(0),\mathfrak{f}(i))$ is $\frac{1}{\Im q\cdot |p-q|^2}dpdq^2$.
\end{Lemma}
\begin{proof}
	By the definition of $A,N$ and $K$ in Lemma \ref{decomp of Haar measure}, the $\mathfrak{f}(i)$ and $\mathfrak{f}(0)$ has the marginal law of $\mathbf t\mathbf s+\mathbf t i$ and $\mathbf t\tan\mathbf u+\mathbf t\mathbf s$ respectively, where $\mathbf t$ is sampled from $\mathbf{1}_{\lbrace t>0\rbrace} \frac{1}{t}dt,\mathbf s$ is sampled from $ds$, and $\mathbf u$ is sampled from $\mathbbm{1}_{\lbrace -\frac{\pi}{2}<u<\frac{\pi}{2}\rbrace}du$. Let $x=ts,y=t$ and $z=t\tan u+st$, then we have
	\begin{equation*}
		\frac{1}{t}dsdtdu=\left(\frac{1}{y^2}dxdy \right)\left(\frac{y}{y^2+(z-x)^2}dz \right)=\frac{1}{\Im q\cdot|p-q|^2}dpd^2q.
	\end{equation*}
	Therefore the joint law of $(\mathfrak{f}(0),\mathfrak{f}(i))$ is equal to $\frac{1}{\Im q\cdot |p-q|^2}dpdq^2$.
\end{proof}
\begin{Lemma}\label{derivative}
	Let $f\in \conf(\mathbb H)$ be such that $f(0)=p\in \mathbb R$ and $f(i)=q\in \mathbb H$, then we have that
	\begin{equation}
		|f'(i)|=\Im q\qquad\text{and}\qquad f'(0)=\frac{|q-p|^2}{\Im q}.
	\end{equation}
	\begin{proof}
		Write $f(z)=\frac{az+b}{cz+d}$ with $ad-bc=1$. Since $f(0)=p$ and $f(i)=q$, we have that
		\begin{equation*}
			\begin{cases}
				\Re(q) &=\frac{ac+bd}{c^2+d^2},\\
				\Im(q) &=\frac{1}{c^2+d^2},\\
				p &=\frac{b}{d},\\
				ad-bc&=1.
			\end{cases}
		\end{equation*}
		Furthermore, we have $|f'(i)|=\frac{1}{c^2+d^2}$ and $f'(0)=\frac{1}{d^2}$. Since $\frac{c^2}{d^2}=\frac{|\Re(q-p)|^2}{|\Im q|^2}$ and $c^2+d^2=\frac{1}{\Im q}$, $f'(0)=\frac{|\Re(q-p)|^2}{\Im q}+\Im q=\frac{|q-p|^2}{\Im q}$ and $|f'(i)|=\Im q$. This completes the proof.
	\end{proof}
\end{Lemma}
\subsection{Proof of Theorem \ref{Main welding theorem}}
Fix $p\in \mathbb R$ and $\gamma\in (0,2)$. Recall that for any $\eta\in \bubble_{\mathbb H}(p)$, $D_\eta(p)$ denotes the component of $\mathbb H\backslash\eta$ which is encircled by $\eta$. Let $|D_\eta(p)|$ denote the euclidean area of $D_\eta(p)$. For $W>0$, let $\rho = W-2$. Define 
\begin{align}\label{SLE bubble tilt}
\begin{split}
    \widetilde{\sle_{\kappa,p}^{\bub}(\rho)}:= \frac{1}{|D_\eta(p)|}\int_{\mathbb H} |q-p|^{W-\frac{2W(W+2)}{\gamma^2}}\left(\Im q\right)^{\frac{W(W+2)}{\gamma^2}-\frac{W}{2}}\sle_{\kappa,p}^{\bub}(\rho)[d\eta|{q\in D_\eta(p)}] d^2q.
\end{split}
\end{align}
\begin{Lemma}\label{inter Lemma on SLE tilt}
For $W>0$, let $\beta_{2W+2}=\gamma-\frac{2W}{\gamma}$. There exists some constant $C\in (0,\infty)$ such that 
\begin{equation}
    \mathbf{m}_{\mathbb H}\ltimes\left(\int_0^\infty \mathcal M^{\disk}_{0,2}(W;\cdot,\ell)\times \qd_{0,1}(\ell) d\ell\right) = C\cdot \LF_{\mathbb H}^{(\beta_{2W+2},p)}(d\phi)\times\widetilde{\sle^{\bub}_{\kappa,p}(\rho)(d\eta)}dp.
\end{equation}
Furthermore, there exists some constant $C\in (0,\infty)$
\begin{equation}
    \mathbf{m}_{\mathbb H,0}\ltimes\left(\int_0^\infty \mathcal M^{\disk}_{0,2}(W;\cdot,\ell)\times \qd_{0,1}(\ell) d\ell\right) = C\cdot \LF_{\mathbb H}^{(\beta_{2W+2},0)}(d\phi)\times\widetilde{\sle^{\bub}_{\kappa,0}(\rho)(d\eta)},
\end{equation}
where recall that $\mathbf{m}_{\mathbb H,0}$ is a Haar measure on $\conf(\mathbb H,0)$, i.e., the group of conformal automorphisms of $\mathbb H$ fixing $0$.
\end{Lemma}
\begin{proof}
    By Theorem \ref{Conditional Welding}, suppose $(\mathbb H,\phi,\eta,0,i)$ is an embedding of quantum surface $\int_0^\infty \mathcal M^{\disk}_{0,2}(W;\cdot,\ell)\times \qd_{1,1}(\ell) d\ell$, then $(\phi,\eta)$ has the law of 
    \begin{equation}\label{LCFT rep of 11}
        C\cdot\LF_{\mathbb H}^{(\gamma,i),\beta_{2W+2},0)}(d\phi)\times \sle_{\kappa,0}^{\bub}(\rho)[d\eta|{i\in D_\eta(0)}]
    \end{equation}
    for some constant $C\in (0,\infty)$. By Proposition \ref{pushforward of LCFT} and Lemma \ref{derivative}, for any $f\in \conf(\mathbb H)$ with $f(0)=p\in \mathbb R$ and $f(i) = q\in \mathbb H$, we have
    \begin{align}
        \begin{split}
            f_{*}\LF_{\mathbb H}^{(\gamma,i),(\beta_{2W+2},0)} &= |f'(0)|^{\Delta_{\beta_{2W+2}}}|f'(i)|^{2\Delta_{\gamma}} \LF_{\mathbb H}^{(\gamma,q),(\beta_{2W+2},p)}\\
            &=\left(\frac{|q-p|^2}{\Im q}\right)^{\Delta_{\beta_{2W+2}}}\cdot\left(\Im q\right)^{2\Delta_{\gamma}}\cdot \LF_{\mathbb H}^{(\gamma,q),(\beta_{2W+2},p)}.
        \end{split}
    \end{align}
    Recall that for $\alpha\in \mathbb R$, $\Delta_{\alpha} = \frac{\alpha}{2}(Q-\frac{\alpha}{2})$. By Lemma \ref{Joint law}, if $f$ is sampled from a $\mathbf{m}_{\mathbb H}$, then the joint law of $(f(0),f(i))$ is $\frac{1}{\Im q\cdot |p-q|^2}dpdq^2$. Therefore, suppose $f$ is sampled from a $\mathbf{m}_{\mathbb H}$, then $f_{*}\LF_{\mathbb H}^{(\gamma,i),(\beta_{2W+2},0)}$ has the law of
    \begin{align}\label{uniform pw of LCFT}
        \begin{split}
             &\frac{1}{\Im q\cdot |p-q|^2}\cdot \left(\frac{|q-p|^2}{\Im q}\right)^{\Delta_{\beta_{2W+2}}}\cdot\left(\Im q\right)^{2\Delta_{\gamma}}\cdot \LF_{\mathbb H}^{(\gamma,q),(\beta_{2W+2},p)} dpdq^2\\
             &= |q-p|^{W-\frac{2W(W+2)}{\gamma^2}}\left(\Im q\right)^{\frac{W(W+2)}{\gamma^2}-\frac{W}{2}}\LF_{\mathbb H}^{(\gamma,q),(\beta_{2W+2},p)}dp d^2q.
        \end{split}
    \end{align}
    Moreover, since $\sle_{\kappa,0}^{\bub}(\rho)[d\eta|{i\in D_\eta(0)}]$ is a probability measure, for fixed $f\in \conf(\mathbb H)$ with $f(0)=p$ and $f(i)=q$, we have 
    \begin{equation}\label{pw of conditional bubble}
        f_* \sle_{\kappa,0}^{\bub}(\rho)[d\eta|{i\in D_\eta(0)}] = \sle_{\kappa,0}^{\bub}(\rho)[d\eta|{q\in D_\eta(p)}].
    \end{equation}
    Combining (\ref{LCFT rep of 11}) ,(\ref{uniform pw of LCFT}) and (\ref{pw of conditional bubble}), we have
    \begin{align}\label{ue of 11}
    \begin{split}
        &\mathbf{m}_{\mathbb H}\ltimes\left(\int_0^\infty \mathcal M^{\disk}_{0,2}(W;\cdot,\ell)\times \qd_{1,1}(\ell) d\ell\right)\\ 
        &= C\cdot |q-p|^{W-\frac{2W(W+2)}{\gamma^2}}\left(\Im q\right)^{\frac{W(W+2)}{\gamma^2}-\frac{W}{2}}\LF_{\mathbb H}^{(\gamma,q),(\beta_{2W+2},p)}\times \sle_{\kappa,0}^{\bub}(\rho)[d\eta|{q\in D_\eta(p)}] dp d^2q.
    \end{split}
    \end{align}
    On the other hand, by \cite[Lemma 2.32]{InteofSLE} (the proof is identical with the domain replaced by $\mathbb H$),
    \begin{align}\label{decomp of sle bubble tilt}
    \begin{split}
	&\LF_{\mathbb H}^{(\beta_{2W+2},p)}(d\phi)\times\widetilde{\sle^{\bub}_{\kappa,p}(\rho)(d\eta)}\mathbbm{1}_{D_\eta(p)}(\mu_{\phi}(d^2q))dp\\
		&=\LF_{\mathbb H}^{(\beta_{2W+2},p),(\gamma,q)}(d\phi)\times\widetilde{\sle_{\kappa,p}^{\bub}(\rho)(d\eta)}\mathbbm{1}_{D_{\eta}(p)}(d^2q)dp\\
		&=\LF_{\mathbb H}^{(\beta_{2W+2},p),(\gamma,q)}(d\phi)\cdot|q-p|^{W-\frac{2W(W+2)}{\gamma^2}}\left(\Im q\right)^{\frac{W(W+2)}{\gamma^2}-\frac{W}{2}}\sle_{\kappa,p}^{\bub}(\rho)[d\eta|{q\in D_\eta(p)}]d^2q dp.
   \end{split}
  \end{align}
  Hence, by (\ref{ue of 11}) and (\ref{decomp of sle bubble tilt}), we have
  \begin{align}\label{inter of eq}
      \begin{split}
          \mathbf{m}_{\mathbb H}\ltimes\left(\int_0^\infty \mathcal M^{\disk}_{0,2}(W;\cdot,\ell)\times \qd_{1,1}(\ell) d\ell\right) = C\cdot \LF_{\mathbb H}^{(\beta_{2W+2},p)}(d\phi)\times\widetilde{\sle^{\bub}_{\kappa,p}(\rho)(d\eta)}\mu_{\phi}(\mathbbm{1}_{D_\eta(p)} d^2q)dp
      \end{split}
  \end{align}
  for some constant $C\in (0,\infty)$. After de-weighting both sides of (\ref{inter of eq}) by the quantum area of $D_\eta(p)$ and forgetting the bulk marked point, we have
  \begin{equation}\label{eqn:ue of 01}
      \mathbf{m}_{\mathbb H}\ltimes\left(\int_0^\infty \mathcal M^{\disk}_{0,2}(W;\cdot,\ell)\times \qd_{0,1}(\ell) d\ell\right) = C\cdot \LF_{\mathbb H}^{(\beta_{2W+2},p)}(d\phi)\times\widetilde{\sle^{\bub}_{\kappa,p}(\rho)(d\eta)}dp.
  \end{equation}
Furthermore, if we consider the $\mathbf{m}_{\mathbb H,0}$, which is a Haar measure on the subgroup of $\conf(\mathbb H)$ fixing $0$, i.e., $\conf(\mathbb H,0)$, then we have 
\begin{equation}\label{eqn:ue of 01 fixing 0}
	\mathbf{m}_{\mathbb H,0}\ltimes\left(\int_0^\infty \mathcal M^{\disk}_{0,2}(W;\cdot,\ell)\times \qd_{0,1}(\ell) d\ell\right) = C\cdot \LF_{\mathbb H}^{(\beta_{2W+2},0)}(d\phi)\times\widetilde{\sle^{\bub}_{\kappa,0}(\rho)(d\eta)}.
\end{equation}
Note that equation (\ref{eqn:ue of 01 fixing 0}) should be viewed as the disintegration of equation (\ref{eqn:ue of 01}) over its boundary marked point. This finishes the proof.
\end{proof}
\begin{Lemma}\label{rela bet SLE and SLE tilt}
    Fix $\rho>-2$. Then there exists some constant $C\in (0,\infty)$ such that
    \begin{equation}
		\sle^{\bub}_{\kappa,0}(\rho) =C \cdot \widetilde{\sle^{\bub}_{\kappa,0}(\rho)},
	\end{equation}
 where the constant $C$ equals to $\sle^{\bub}_{\kappa,0}(\rho)[i\in D_\eta(0)]$.
\end{Lemma}
\begin{proof}
    Notice that 
	\begin{align*}
		|D_\eta(0)|\cdot\sle^{\bub}_{\kappa,0}(\rho)(d\eta) &= \int_{\mathbb H}\mathbbm{1}_{q\in D_\eta(0)}\sle^{\bub}_{\kappa,0}(\rho)(d\eta) d^2q\\
		&= \int_{\mathbb H}\sle^{\bub}_{\kappa,0}(\rho)[q\in D_\eta(0)]\cdot\sle^{\bub}_{\kappa,0}(\rho)[d\eta|q\in D_\eta(0)]d^2q.
	\end{align*}
 Let $\psi\in \conf(\mathbb H)$ be such that $\psi(i) = q$ and $\psi(0) = 0$ and it is easy to show that $\psi'(0)=\frac{|q|^2}{\Im q}$. By \cite[Theorem $3.16$]{SLEbubble}, we have
 \begin{equation}
     \sle^{\bub}_{\kappa,0}(\rho)[q\in D_\eta(0)] = \psi'(0)^{-\alpha}\cdot\sle^{\bub}_{\kappa,0}(\rho)[i\in D_\eta(0)],
 \end{equation}
 where $\alpha = \frac{(\rho+2)(2\rho+8-\kappa)}{2\kappa}$. Since $W=\rho+2$, we have
 \begin{align}
		\psi'(0)^{-\alpha} = |q|^{W-\frac{2W(W+2)}{\kappa}}(\Im q)^{-\frac{W}{2}+\frac{W(W+2)}{\kappa}}.
	\end{align}
 Hence, 
 \begin{equation}
		\sle^{\bub}_{\kappa,0}(\rho)[q\in D_\eta(0)] = C\cdot |q|^{W-\frac{2W(W+2)}{\kappa}}(\Im q)^{-\frac{W}{2}+\frac{W(W+2)}{\kappa}},
\end{equation}
where $C=\sle^{\bub}_{\kappa,0}(\rho)[i\in D_\eta(0)]\in (0,\infty)$ by Corollary \ref{conditional convergence lemma}. Therefore, by (\ref{SLE bubble tilt}),
\begin{align}
\begin{split}
    \sle^{\bub}_{\kappa,0}(\rho) &= C\cdot \frac{1}{|D_\eta(0)|}\int_{\mathbb H} |q|^{W-\frac{2W(W+2)}{\kappa}}(\Im q)^{-\frac{W}{2}+\frac{W(W+2)}{\kappa}}\sle^{\bub}_{\kappa,0}(\rho)[d\eta|q\in D_\eta(0)]d^2q\\
    &=C\cdot \widetilde{\sle_{\kappa,0}^{\bub}(\rho)}.
\end{split}
\end{align}
This completes the proof.
\end{proof}
\begin{Cor}\label{Cor:root p equiv}
	Fix $\rho>-2$ and $p\in \mathbb R$. Then there exists some constant $C\in (0,\infty)$ such that
    \begin{equation}\label{eqn:SLE equiv with root p}
		\sle^{\bub}_{\kappa,p}(\rho) =C \cdot \widetilde{\sle^{\bub}_{\kappa,p}(\rho)}.
	\end{equation}
\end{Cor}
\begin{proof}
	Fix $p\in \mathbb R$. Let $f_p\in \conf(\mathbb H)$ be such that $f_p(z)=z+p$. Since $\sle_{\kappa,0}^{\bub}[d\eta|q\in D_\eta(0)]$ is a probability measure on $\bubble_{\mathbb H}(0,q)$ for all $q$, 
	\begin{equation}
		f_p\left(\sle_{\kappa,0}^{\bub}[d\eta|q\in D_\eta(0)]\right) =  \sle_{\kappa,p}^{\bub}[d\eta|q\in D_\eta(p)].
	\end{equation}
	Hence,  
	\begin{align}
		\begin{split}
			&f_p\left(\widetilde{\sle_{\kappa,0}^{\bub}(\rho)}\right) \\
			&=f_p\left(\frac{1}{|D_\eta(0)|}\int_{\mathbb H} |q|^{W-\frac{2W(W+2)}{\kappa}}(\Im q)^{-\frac{W}{2}+\frac{W(W+2)}{\kappa}}\sle^{\bub}_{\kappa,0}(\rho)[d\eta|q\in D_\eta(0)]d^2q \right)\\
			&=\frac{1}{|D_\eta(p)|}\int_{\mathbb H}|q-p|^{W-\frac{2W(W+2)}{\kappa}}(\Im (q-p))^{-\frac{W}{2}+\frac{W(W+2)}{\kappa}}\sle^{\bub}_{\kappa,p}(\rho)[d\eta|q\in D_\eta(p)]d^2q\\
			&= \frac{1}{|D_\eta(p)|}\int_{\mathbb H}|q-p|^{W-\frac{2W(W+2)}{\kappa}}(\Im q)^{-\frac{W}{2}+\frac{W(W+2)}{\kappa}}\sle^{\bub}_{\kappa,p}(\rho)[d\eta|q\in D_\eta(p)]d^2q\\
			&=\widetilde{\sle_{\kappa,p}^{\bub}(\rho)}.
		\end{split}
	\end{align}
	By Lemma \ref{rela bet SLE and SLE tilt}, we have
	\begin{equation}\label{eqn:SLE equiv with root 0}
		\sle^{\bub}_{\kappa,0}(\rho) =C \cdot \widetilde{\sle^{\bub}_{\kappa,0}.(\rho)}
	\end{equation}
	The (\ref{eqn:SLE equiv with root p}) follows from applying $f_p$ on both sides of (\ref{eqn:SLE equiv with root 0}). 
\end{proof}

\begin{proof}[Proof of Theorem \ref{Main welding theorem}]
    Theorem \ref{Main welding theorem} follows immediately from Lemma \ref{inter Lemma on SLE tilt}, Lemma \ref{rela bet SLE and SLE tilt} and Corollary \ref{Cor:root p equiv}.
\end{proof}
\section{SLE bubble zippers with a generic insertion}\label{sec:bulk insertion of general weight}
\subsection{SLE bubble zippers with a generic bulk insertion}
\begin{Def}[Definition \ref{Def one bulk one boundary}]   
For $\alpha,\beta\in \mathbb R$, let $\phi$ be sampled from $\LF_{\mathbb H}^{(\alpha,i),(\beta,0)}$. We denote $\qd_{1,1}(\alpha,\beta)$ the infinite measure describing the law of quantum surface $(\mathbb H,\phi,0,i)$.
\end{Def}
\begin{Lemma}\label{transformation rule}
	Fix $\alpha,\beta\in\mathbb R$ and $q\in \mathbb H$, and we have
	\begin{equation}
		\left(\Im q\right)^{2\Delta_\alpha-\Delta_\beta}|q|^{2\Delta_\beta}\LF_{\mathbb H}^{(\alpha,q)(\beta,0)}=(f_{\infty})_{*} \LF_{\mathbb H}^{(\beta,\infty),(\alpha,i)},
	\end{equation}
	where $f_{\infty}\in\conf(\mathbb H)$ is the conformal map with $f_{\infty}(\infty)=0$ and $f_{\infty}(i)=q$.
\end{Lemma}
\begin{proof}
	For each $r>0$, let $f_r\in \conf(\mathbb H)$ be a conformal map such that $f_r(r)=0$ and $f_r(i)=q$. By Proposition \ref{pushforward of LCFT}, we have 
	\begin{equation}\label{pr LCFT}
		\LF_{\mathbb H}^{(\beta,0),(\alpha,q)}=|f'_r(i)|^{-2\Delta_\alpha}|f_r'(r)|^{-\Delta_\beta}(f_r)_{*}\LF_{\mathbb H}^{(\beta,r),(\alpha,i)}.
	\end{equation}
	Assume $f_r(z)=\frac{a_rz+b_r}{c_rz+d_r}$, where $a_r d_r-b_r c_r=1$.  Trivially, we have $|f_r'(z)|=\frac{1}{(c_rz+d_r)^2}$. Since $f_r(r)=0$ and $f_r(i)=q$, we have 
	\begin{equation*}
		\begin{cases}
			\Im q=\frac{1}{c_r^2+d_r^2},\\
			\Re q=\frac{a_r c_r+b_r d_r}{c_r^2+d_r^2},\\
			a_r r+b_r=0.
		\end{cases}
	\end{equation*}
	After solving the above equations, we have 
	\begin{equation*}
		\begin{cases}
			|f'_r(r)|=a_r^2=\frac{|q|^2}{(r^2+1)\Im q},\\
			|f_r'(i)|=\Im q,
		\end{cases}
		\qquad\text{and}\qquad
		\begin{cases}
			a_r=\frac{|q|}{\sqrt{r^2+1}\sqrt{\Im q}},\\
			b_r=-\frac{r|q|}{\sqrt{r^2+1}\sqrt{\Im q}},\\
			c_r=\frac{\sqrt{|q|^2(r^2+1)-(\Im q-r\Re q)^2}}{\sqrt{r^2+1}\sqrt{\Im q}|q|},\\
			d_r=\frac{\sqrt{\Im q}(1-r\frac{\Re q}{\Im q})}{\sqrt{r^2+1}|q|}.
		\end{cases}
	\end{equation*}
	After multiplying $r^{\beta(Q-\frac{\beta}{2})}$ on both sides of $(\ref{pr LCFT})$, we have
	\begin{equation*}
		\left(\frac{r^{\beta(Q-\frac{\beta}{2})}}{(r^2+1)^{\frac{\beta}{2}(Q-\frac{\beta}{2})}}\right)(\Im q)^{2\Delta_\alpha-\Delta_\beta}|q|^{2\Delta_\beta}\LF_{\mathbb H}^{(\alpha,q),(\beta,0)}=r^{\beta(Q-\frac{\beta}{2})} (f_r)_{*}\left[\LF_{\mathbb H}^{(\beta,r),(\alpha,i)} \right].
	\end{equation*}
    As $r\to \infty$, the left hand side becomes $(\Im q)^{2\Delta_\alpha-\Delta_\beta}|q|^{2\Delta_\beta}\LF_{\mathbb H}^{(\alpha,q),(\beta,0)}$. The right hand side converges in vague topology to  $(f_{\infty})_{*}\LF_{\mathbb H}^{(\beta, \infty),(\alpha,i)}$ follows from the facts that $f_r\to f_{\infty}$ in the topology of uniform convergence of analytic function and its derivatives on all compact sets and \cite[Lemma $2.18$]{InteofSLE}. This completes the proof.
\end{proof}
\begin{Lemma}
	Let $\alpha_1,\alpha_2, \beta \in \mathbb R$ and $\ell>0$. For $\varepsilon>0$, we define the measure $\LF_{\mathbb H,\varepsilon}^{(\alpha_2,i),(\beta,0)}$ through the Radon-Nikodym derivative as follows: 
	\begin{equation*}
		\frac{\LF_{\mathbb H,\varepsilon}^{(\alpha_2,i),(\beta,0)}(\ell)}{\LF_{\mathbb H}^{(\alpha_1,i),(\beta,0)}(\ell)}(\phi):=\varepsilon^{\frac{1}{2}(\alpha_2^2-\alpha_1^2)}e^{(\alpha_2-\alpha_1)\phi_\varepsilon(i)}.
	\end{equation*}
	Furthermore, we have the weak convergence of measures 
	\begin{equation*}
		\lim_{\varepsilon\to 0} \LF_{\mathbb H,\varepsilon}^{(\alpha_2,i),(\beta,0)}(\ell)=\LF_{\mathbb H}^{(\alpha_2,i),(\beta,0)}(\ell).
	\end{equation*}
\end{Lemma}
\begin{proof}
	We know that if $\phi$ is sampled from $\LF_{\mathbb H}^{(\alpha,i),(\beta,0)}(1)^{\#}$, then $\phi+\frac{2}{\gamma}\log\ell$ has the law of $\LF_{\mathbb H}^{(\alpha,i),(\beta,0)}(\ell)^{\#}$. Moreover, we have 
	\begin{equation*}
		\frac{|\LF_{\mathbb H,\varepsilon}^{(\alpha,i),(\beta,0)}(\ell) |}{|\LF_{\mathbb H,\varepsilon}^{(\alpha,i),(\beta,0)}(1)|}=\frac{|\LF_{\mathbb H}^{(\alpha,i),(\beta,0)}(\ell) |}{|\LF_{\mathbb H}^{(\alpha,i),(\beta,0)}(1)|}=\ell^{\frac{2}{\gamma}\left(\frac{\beta}{2}+\alpha-Q\right)-1}. 
	\end{equation*}
	Let 
	\begin{equation*}
		\widetilde{h^j}=h-2Q\log|\cdot|_{+}+\alpha_j G_{\mathbb H}(\cdot,i)+\frac{\beta}{2} G_{\mathbb H}(\cdot,0),\qquad j=1,2
	\end{equation*}
	and $\widetilde{h^{2,\varepsilon}}=\widetilde{h^1}+(\alpha_2-\alpha_1) G_{\mathbb H,\varepsilon}(\cdot,i)$, where $G_{\mathbb H,\varepsilon}(z,i)$ is the average of Green function $G_{\mathbb H}(z,\cdot)$ over $\partial B(i,\varepsilon)$. Notice that $\var\left(h_\varepsilon(i)\right)=-\log\varepsilon-\log 2+o_\varepsilon(1)$ and $\mathbb E\left[e^{(\alpha_2-\alpha_1)h_\varepsilon(i)}\right]=(1+o_\varepsilon(1)) \left(2\varepsilon\right)^{-\frac{1}{2}(\alpha_2-\alpha_1)^2}$. Furthermore, the average of $-2Q\log|\cdot|+\alpha G_\mathbb H(\cdot,i)+\frac{\beta}{2} G_{\mathbb H}(\cdot,0)$ over $\partial B(i,\varepsilon)$ is $-\alpha\log(2\varepsilon)+o_\varepsilon(1)$. Let $L_1=\nu_{\widetilde h^1}(\mathbb R)$, $L_2=\nu_{\widetilde h^2}(\mathbb R)$ and $L_{2,\varepsilon}=\nu_{\widetilde{ h^{2,\varepsilon}}}(\mathbb R)$. For any bounded continuous function $F$ on $H^{-1}(\mathbb H)$, we have 
	\begin{align*}
		&\int\varepsilon^{\frac{1}{2}(\alpha_2^2-\alpha_1^2)}e^{(\alpha_2-\alpha_1)(\widetilde{h_\varepsilon^1}(i)-\frac{2}{\gamma}\log L_1)}F(\widetilde{h^1}-\frac{2}{\gamma}\log L_1)\cdot 2^{-\frac{\alpha_1^2}{2}}\cdot\frac{2}{\gamma}L_1^{-\frac{2}{\gamma}(\frac{\beta}{2}+\alpha_1-Q)}dh\\
		&=\int\frac{(1+o_\varepsilon(1))}{\mathbb E[e^{(\alpha_2-\alpha_1)h_\varepsilon(i)}]} e^{(\alpha_2-\alpha_1){h_\varepsilon}(i)}F(\widetilde{h^1}-\frac{2}{\gamma}\log L_1)\cdot 2^{-\frac{\alpha_2^2}{2}}\cdot\frac{2}{\gamma}L_1^{-\frac{2}{\gamma}(\frac{\beta}{2}+\alpha_2-Q)}dh\\
		&=\int(1+o_\varepsilon(1))F(\widetilde{h^{2,\varepsilon}}-\frac{2}{\gamma}\log L_{2,\varepsilon})\cdot 2^{-\frac{\alpha_2^2}{2}}\cdot\frac{2}{\gamma}L_{2,\varepsilon}^{-\frac{2}{\gamma}(\frac{\beta}{2}+\alpha_2-Q)}dh\\
		&\xrightarrow{\varepsilon\to 0}  \int F(\widetilde{h^{2}}-\frac{2}{\gamma}\log L_{2})\cdot 2^{-\frac{\alpha_2^2}{2}}\cdot\frac{2}{\gamma}L_{2}^{-\frac{2}{\gamma}(\frac{\beta}{2}+\alpha_2-Q)}dh.
	\end{align*}
	The second equality follows from the Girsanov's Theorem. Since $L_2=(1+o_\varepsilon(1))L_{2,\varepsilon}$ and $\sup_{x\in \mathbb R}| G_{\mathbb H}(x,i)-G_{\mathbb H,\varepsilon}(x,i)|=o_\varepsilon(1)$, the final $\varepsilon$ limit follows from the the Dominated Convergence Theorem.  
\end{proof}
\subsection{Proof of Theorem $\ref{welding of generalized weight}$}
\begin{proof}[Proof of Theorem $\ref{welding of generalized weight}$]
	By Theorem \ref{Conditional Welding}, we have 
	\begin{equation*}
		\LF_{\mathbb H}^{(\beta_{2W+2},0),(\gamma,i)}(1)\times \mathbf m=C_W\cdot\int_0^\infty \qd_{1,1}(\ell)\times \mathcal M^{\disk}_{0,2}(W;1,\ell)d\ell. 
	\end{equation*}
	Let $(Y,\eta)$ be sampled from the left hand side. Let $\psi_\eta: \mathbb H\to D_\eta(i)$ be the conformal map fixing $0$ and $i$ and $\xi_\eta:\mathbb H\to D_\eta(\infty)$ be such that $\xi_\eta(0)=0^-$, $\xi_\eta(1)=0^+$ and $\xi_\eta(\infty)=\infty$. Let $X,Z\in H^{-1}(\mathbb H)$ be such that 
	\begin{equation*}
		X=Y\circ \psi_\eta+Q\log|\psi_\eta'|\qquad\text{and}\qquad Z=Y\circ \xi_\eta+Q\log|\xi_\eta'|.
	\end{equation*}
	Notice that $\qd_{1,1}(\ell)$ embedded in $(\mathbb H,0,i)$ has the law of $C\cdot\LF_{\mathbb H}^{(\gamma,i),(\gamma,0)}(r)$. Therefore, the $X$ has the law of 
	\begin{equation*}
		C_W\int_0^\infty |\mathcal M^{\disk}_{0,2}(W;1,\ell)|\cdot\LF_{\mathbb H}^{(\gamma,i),(\gamma,0)}(\ell)d\ell.
	\end{equation*}
	The conditional law of marked quantum surface $(\mathbb H,Z,0,1)$ given $X$ is $\mathcal M^{\disk}_{0,2}(\beta_W;1,\nu_X(\mathbb R))^{\#}$. Next, if we re-weight $X$ by $\varepsilon^{\frac{1}{2}(\alpha^2-\gamma^2)}e^{(\alpha-\gamma)X_\varepsilon(i)}$ and send $\varepsilon$ to $0$, the law of $X$ converges weakly to 
	\begin{equation*}
		C_W\int_0^\infty |\mathcal M^{\disk}_{0,2}(W;1,\ell)|\cdot\LF_{\mathbb H}^{(\alpha,i),(\gamma,0)}(\ell)d\ell.
	\end{equation*} 
	Consequently, the law of $Z$ conditioned on re-weighted $X$ is $\mathcal M^{\disk}_{0,2}(W;1,\nu_X(\mathbb R))^{\#}$.
	\par
	Next, let $\theta_{i,\varepsilon}$ be the uniform probability measure on $\partial B(i,\varepsilon)$ for sufficiently small $\varepsilon$. Let $\theta^\eta_{i,\varepsilon}=(\psi_\eta)_{*}(\theta_{i,\varepsilon})$ be the push-forward of $\theta_{i,\varepsilon}$ under $\psi_\eta$. Since $\psi'_\eta$ is holomorphic and $\log|\psi'_\eta|$ is harmonic, 
	\begin{equation*}
		X_\varepsilon(i)=(X,\theta_{i,\varepsilon})=(Y\circ \psi_\eta+Q\log|\psi_\eta'|,\theta^\eta_{i,\varepsilon})=(Y,\theta^\eta_{i,|\psi'_\eta(i)|\varepsilon})+Q\log|\psi'_\eta(i)|. 
	\end{equation*}
	Therefore, re-weighting by $\varepsilon^{\frac{1}{2}(\alpha^2-\gamma^2)}e^{(\alpha-\gamma)X_\varepsilon(i)}$ is equivalent to re-weighting by 
	\begin{align*}
		&\varepsilon^{\frac{1}{2}(\alpha^2-\gamma^2)}e^{(\alpha-\gamma)[(Y,\theta^\eta_{i,|\psi'_\eta(i)|\varepsilon})+Q\log|\psi'_\eta(i)|]}\\
		&=\left(\varepsilon|\psi_\eta'(i)| \right)^{\frac{1}{2}(\alpha^2-\gamma^2)}e^{(\alpha-\gamma)(Y,\theta^\eta_{i,|\psi'_\eta(i)|\varepsilon})}|\psi_\eta'(i)|^{-\frac{1}{2}\alpha^2+Q\alpha-2}.\\
	\end{align*}
     Hence, we conclude that for any bounded continuous $F$ on $H^{-1}(\mathbb H)^3$ and bounded continuous function $g$ on $\bubble_{\mathbb H}(0,i)$ equipped with Hausdorff topology, 
	\begin{align*}
		&\lim_{\varepsilon\to 0}\int\int C_{W}\left(\varepsilon|\psi_\eta'(i)| \right)^{\frac{1}{2}(\alpha^2-\gamma^2)}e^{(\alpha-\gamma)(Y,\theta^\eta_{i,|\psi'_\eta(i)|\varepsilon})}F(X,Y,Z)\LF_{\mathbb H}^{(\beta_{2W+2},0),(\gamma,i)}(1)(dY)g(\eta)\mathbf m_{\alpha}(d\eta)\\
		&=\int\int C_{W}F(\widetilde X,\widetilde Y,\widetilde Z)\LF_{\mathbb H}^{(\beta_{2W+2},0),(\gamma,i)}(1)(d\widetilde Y)g(\widetilde \eta)\mathbf m_{\alpha}(d\widetilde\eta).
	\end{align*}
	By conformal welding, $(X,Z)$ is uniquely determined by $(Y,\eta)$. Similarly, $(\widetilde Y,\widetilde \eta)$ is uniquely determined by $(\widetilde X,\widetilde Z)$. Therefore, when $(\widetilde Y,\widetilde\eta)$ is sampled from $\LF_{\mathbb H}^{(\beta_{2W+2},0),(\gamma,i)}\times \mathbf m_{\alpha}$, $\widetilde X$ has the law of 
	\begin{equation*}
		C_W\cdot\int_{0}^{\infty}|\mathcal M^{\disk}_{0,2}(W;1,\ell)|\cdot \LF_{\mathbb H}^{(\alpha,i),(\gamma,0)}(\ell) d\ell
	\end{equation*}
	and the conditional law of marked quantum surface $(\mathbb H,\widetilde Z,0,i)$ given $\widetilde X$ is $\mathcal M^{\disk}_{0,2}(W;1,\nu_{\widetilde X}(\mathbb R))$. This finishes the proof. 
\end{proof}
\section{Applications}\label{sec:computation of conformal radius}
\subsection{Preliminary results on integrabilities of the LCFT}
First, we recall the double gamma function $\Gamma_b(z)$. For $b$ such that $\Re(b)>0$, the $\Gamma_b(z)$ is a meromorphic function on $\mathbb C$ such that 
\begin{equation*}
	\ln \Gamma_b(z)=\int_0^\infty\frac{1}{t}\left(\frac{e^{-zt}-e^{-(b+\frac{1}{b})t/2}}{(1-e^{-bt})(1-e^{-\frac{1}{b}t})}-\frac{\left(\frac{1}{2}(b+\frac{1}{b})-z\right)^2}{2}e^{-t}+\frac{z-\frac{1}{2}(b+\frac{1}{b})}{t}\right)dt
\end{equation*}
for $\Re(z)>0$ and it satisfies the following two shift equations:
\begin{equation}
	\frac{\Gamma_b(z)}{\Gamma_b(z+b)}=\frac{1}{\sqrt{2\pi}}\Gamma(bz)b^{-bz+\frac{1}{2}}\qquad\text{and}\qquad \frac{\Gamma_b(z)}{\Gamma_b(z+\frac{1}{b})}=\frac{1}{\sqrt{2\pi}}\Gamma\left(\frac{1}{b}z \right)\left(\frac{1}{b}\right)^{-\frac{1}{b}z+\frac{1}{2}}.
\end{equation}
The above two shift equations allow us to extend $\Gamma_b(z)$ meromorphically from $\Re(z)>0$ to the entire complex plane $\mathbb C$. It has simple poles at $-nb-m\frac{1}{b}$ for nonnegative integers $m,n$. The double sine function is defined as
\begin{equation}
	S_b(z):=\frac{\Gamma_b(z)}{\Gamma_b(b+\frac{1}{b}-z)}.
\end{equation}
We can now define the Liouville reflection coefficient $R$. For fixed $\mu_1,\mu_2>0$, let $\sigma_j\in \mathbb C$ satisfy $\mu_j=e^{i\pi\gamma(\sigma_j-\frac{Q}{2})}$ and $\Re\sigma_j=\frac{Q}{2}$ for $j=1,2$ and define the following two meromorphic functions for $\beta\in \mathbb C$ as belows: 
\begin{align}
	\overline R(\beta,\mu_1,\mu_2) &=\frac{(2\pi)^{\frac{2}{\gamma}(Q-\beta)-\frac{1}{2}}(\frac{2}{\gamma})^{\frac{\gamma}{2}(Q-\beta)-\frac{1}{2}}}{(Q-\beta)\Gamma(1-\frac{\gamma^2}{4})^{\frac{2}{\gamma}(Q-\beta)}}\frac{\Gamma_{\frac{\gamma}{2}}(\beta-\frac{\gamma}{2})e^{i\pi(\sigma_1+\sigma_2-Q)(Q-\beta)}}{\Gamma_{\frac{\gamma}{2}}(Q-\beta)S_{\frac{\gamma}{2}}(\frac{\beta}{2}+\sigma_2-\sigma_1)S_{\frac{\gamma}{2}}(\frac{\beta}{2}+\sigma_1-\sigma_2)},\\
	R(\beta,\mu_1,\mu_2) &=-\Gamma\left(1-\frac{2}{\gamma}(Q-\beta)\right)\overline R(\beta,\mu_1,\mu_2).
\end{align}
\begin{Prop}[{\cite[Theorem $1.7$]{RZ22}}]
	Let $\beta_W=Q=\frac{\gamma}{2}-\frac{W}{\gamma} \in (\frac{\gamma}{2},Q)$. Let $\mu_1,\mu_2\geq 0$ not both be zero. Recall random field $\widehat h$  defined in Definition $\ref{Def: thick quantum disk}$ of $\mathcal M^{\disk}_{0,2}(W)$. We have that 
	\begin{equation}
		\mathbb E\left[\left(\mu_1\nu_{\widehat h}(\mathbb R)+\mu_2\nu_{\widehat h}(\mathbb R+\pi i) \right)^{\frac{2}{\gamma}(Q-\beta_W)}\right]=\overline R(\beta_W,\mu_1,\mu_2).
	\end{equation}
\end{Prop}
\begin{Lemma}[{\cite[Lemma $3.3$]{InteofSLE}}]
	For $W\in [\frac{\gamma^2}{2},\gamma Q)$ and $\beta_W=Q+\frac{\gamma}{2}-\frac{W}{\gamma}$, let $L_1,L_2$ denote the left and right boundary length of weight $W$ quantum disk $\mathcal M^{\disk}_{0,2}(W)$, then the law of $\mu_1 L_1+\mu_2 L_2$ is 
	\begin{equation*}
		\mathbbm{1}_{\ell>0} \overline R(\beta_W,\mu_1,\mu_2)\ell^{-\frac{2}{\gamma^2}W}d\ell.
	\end{equation*}
\end{Lemma}
Let $W=2, \mu_1=\mu_2=1$ and by independent sampling property of $\mathcal M^{\disk}_{0,2}(2)$, we have the following results on the joint law of left and right boundary length.  
\begin{Prop}[\cite{DMS14}, Proposition $5.1$]
	For $\ell,\gamma>0$, we have 
	\begin{equation}
		|\mathcal M^{\disk}_{0,2}(2;\ell,r)|=\frac{(2\pi)^{\frac{4}{\gamma^2}-1}}{(1-\frac{\gamma^2}{4})\Gamma(1-\frac{\gamma^2}{4})^{\frac{4}{\gamma^2}}}(\ell+r)^{-\frac{4}{\gamma^2}-1}.
	\end{equation}
\end{Prop}
\begin{Prop}[{\cite[Proposition $3.6$]{InteofSLE}}]
	For $W\in(0,\frac{\gamma^2}{2})$ and $\beta_W=Q+\frac{\gamma}{2}-\frac{W}{\gamma}\in (Q,Q+\frac{\gamma}{2})$. Let $L_W$ and $R_W$ be the left and right quantum boundary lengths of weight-$W$ thin quantum disk $\mathcal M^{\disk}_{0,2}(W)$, and we have
	\begin{equation}
		\mathcal M_{0,2}^{\disk}(W)\left[e^{-\mu_1 L_W-\mu_2 R_W} \right]=-\frac{\gamma}{2(\beta_W-Q)}R(\beta_W;\mu_1,\mu_2).
	\end{equation} 
\end{Prop}
Next, we recall the two-pointed correlation function of the Liouville theory on $\mathbb H$ that was introduced in Section \ref{subsubsec: two-pointed correlation function} when $\mu=0,\mu_\partial>0$. For bulk insertions $z_i$ with weights $\alpha_i$ and boundary insertions $s_j$ with weights $\beta_j$, the correlation function of LCFT at these points is defined using the following formal path integral: 
\begin{equation}
	\left\langle\prod_{i=1}^N e^{\alpha_i\phi(z_i)}\prod_{j=1}^M e^{\frac{\beta_j}{2}\phi(s_j)}\right\rangle_{\mu_\partial}=\int_{X:\mathbb H\to \mathbb R} DX\prod_{i=1}^N e^{\alpha_iX(z_i)}\prod_{j=1}^M e^{\frac{\beta_j}{2}X(s_j)}e^{-S_{\mu_\partial}^L(X)}.
\end{equation}
In the above formula, $DX$ is the formal uniform measure on infinite dimensional function space and $S^{\mu_\partial}_L(X)$ is the \textit{Liouville action functional} given by 
\begin{equation}
	S_{\mu_\partial}^L(X):=\frac{1}{4\pi}\int_{\mathbb H} \left(|\nabla_g X|^2+Q R_g X\right)d\lambda_g+\frac{1}{2\pi}\int_{\mathbb R}\left(QK_g X+2\pi\mu_\partial e^{\frac{\gamma}{2}X}\right)d\lambda_{\partial g}.
\end{equation}
For background Riemannian metric $g$ on $\mathbb H$, $\nabla_g, R_g, K_g,d\lambda_g,d\lambda_{\partial g}$ stand for the gradient, Ricci curvature, Geodesic curvature, volume form and line segment respectively. The subscript $\mu_\partial$ emphases the fact that we are considering the case when $\mu=0,\mu_\partial>0$. For $z\in \mathbb H$ and $s\in \mathbb R$, the \textit{bulk-boundary correlator} is 
\begin{equation}
	\left\langle e^{\alpha\phi(z)}e^{\frac{\beta}{2}\phi(s)}\right\rangle_{\mu_\partial}=\frac{G(\alpha,\beta)}{|z-\overline z|^{2\Delta_\alpha-\Delta_\beta}|z-s|^{2\Delta_\beta}}.
\end{equation}
Next, we introduce the rigorous mathematical definition of $G(\alpha,\beta)$. 
\begin{Def}[{\cite[Definition $1.5$]{RZ22}}]
	The function $G(\alpha,\beta)=\frac{2}{\gamma}\Gamma\left(\frac{2\alpha+\beta-2Q}{\gamma}\right)\left(\mu_B^{\frac{2Q-2\alpha-\beta}{\gamma}}\right)\overline G(\alpha,\beta)$, where for $\beta<Q$ and $\frac{\gamma}{2}-\alpha<\frac{\beta}{2}<\alpha$:
	\begin{equation}
		\overline G(\alpha,\beta)=\mathbb E\left[\left(\int_{\mathbb R}\frac{g(x)^{\frac{\gamma}{4}(\frac{2}{\gamma}-\alpha-\frac{\beta}{2})}}{|x-i|^{\gamma\alpha}}e^{\frac{\gamma}{2}h(x)}dx\right)^{\frac{2}{\gamma}(Q-\alpha-\frac{\beta}{2})}\right].
	\end{equation}
	In the above formula, $g(x)=\frac{1}{|x|_+^4}, |x|_+=\max(|x|,1)$ and $h(x)$ is sampled from $P_{\mathbb H}$.
\end{Def}
\begin{Th}[{\cite[Theorem $1.7$]{RZ22}}]
For $\gamma\in (0,2),\beta<Q$ and $\frac{\gamma}{2}-\alpha<\frac{\beta}{2}<\alpha$, 
	\begin{equation}
	\overline G(\alpha,\beta)=\left(\frac{2^{\frac{\gamma}{2}(\frac{\beta}{2}-\alpha)}2\pi}{\Gamma(1-\frac{\gamma^2}{4})}\right)^{\frac{2}{\gamma}(Q-\alpha-\frac{\beta}{2})}\frac{\Gamma(\frac{\gamma\alpha}{2}+\frac{\gamma\beta}{4}-\frac{\gamma^2}{4})\Gamma_{\frac{\gamma}{2}}(\alpha-\frac{\beta}{2})\Gamma_{\frac{\gamma}{2}}(\alpha+\frac{\beta}{2})\Gamma_{\frac{\gamma}{2}}(Q-\frac{\beta}{2})^2}{\Gamma_{\frac{\gamma}{2}}(Q-\frac{\beta}{2})\Gamma_{\frac{\gamma}{2}}(\alpha)^2\Gamma_{\frac{\gamma}{2}}(Q)}.
\end{equation}
\end{Th}
\begin{Lemma}\label{quantum length distribution of LCFT}
	Fix $\ell>0$. Let $\gamma,\beta,\alpha$ be such that $\gamma\in (0,2),\beta<Q,\frac{\gamma}{2}-\alpha<\frac{\beta}{2}<\alpha$. Let $h$ be sampled from $P_{\mathbb H}$ and let $h_{\infty}(z)=h(z)+(\beta-2Q)\log|z|_{+}+\alpha G_{\mathbb H}(z,i)$. Let $\phi$ be sampled from $\LF_{\mathbb H}^{(\beta,\infty),(\alpha,i)}(d\phi)$ and for each bounded non-negative measurable function $f$ on $(0,\infty)$, we have 
	\begin{equation*}
		\LF_{\mathbb H}^{(\beta,\infty ),(\alpha,i)}[f(\nu_{\phi}(\mathbb R))]=\int_{0}^{\infty}f(\ell) 2^{-\frac{\alpha^2}{2}}\ell^{\frac{2}{\gamma}(\frac{1}{2}\beta+\alpha-Q)-1}\cdot \frac{2}{\gamma}\cdot \overline G(\alpha,\beta)d\ell,
	\end{equation*}
	where $\overline G(\alpha,\beta)$ is the two point (one bulk, one boundary) correlation function of Liouville theory on $\mathbb H$. 
\end{Lemma}
\begin{proof}
	It suffices to consider the case when $f(\ell)=\mathbbm{1}_{a<\ell<b}(\ell)$. By direct computation, 
	\begin{align*}
	\LF_{\mathbb H}^{(\beta,\infty),(\alpha,i)}[f(\nu_{\phi}(\mathbb R))]&=\mathbb E\left[\int_{\mathbb R}\mathbbm{1}_{\lbrace e^{\frac{\gamma}{2}c}\nu_{h_\infty(\mathbb R)}\in(a,b) \rbrace} 2^{-\frac{\alpha^2}{2}}e^{(\frac{1}{2}\beta+\alpha-Q)c}dc\right] \\
	 &=\mathbb E\left[\int_a^b \nu_{h_\infty}(\mathbb R)^{\frac{2}{\gamma}(Q-\alpha-\frac{1}{2}\beta)}2^{-\frac{\alpha^2}{2}}\ell^{\frac{2}{\gamma}(\frac{1}{2}\beta+\alpha-Q)-1}\cdot\frac{2}{\gamma}d\ell \right] \\
		&=\int_a^b 2^{-\frac{\alpha^2}{2}}\ell^{\frac{2}{\gamma}(\frac{1}{2}\beta+\alpha-Q)-1}\cdot \frac{2}{\gamma}\cdot \mathbb E\left[\nu_{h_{\infty}(\mathbb R)}^{\frac{2}{\gamma}(Q-\alpha-\frac{1}{2}\beta)} \right]d\ell.
	\end{align*}
	The second line follows from the change of variable $\ell=e^{\frac{\gamma}{2}c}\nu_{h_\infty}(\mathbb R)$. The third line follows from the finiteness of $\mathbb E\left[\nu_{h_{\infty}(\mathbb R)}^{\frac{2}{\gamma}(Q-\alpha-\frac{1}{2}\beta)} \right]$ and Fubini's theorem. The finiteness of $\mathbb E\left[\nu_{h_{\infty}(\mathbb R)}^{\frac{2}{\gamma}(Q-\alpha-\frac{1}{2}\beta)} \right]$ is proved in \cite[Proposition $5.1$]{RZ22}. Furthermore, 
	\begin{align*}
		\mathbb E\left[\nu_{h_{\infty}(\mathbb R)}^{\frac{2}{\gamma}(Q-\alpha-\frac{1}{2}\beta)} \right]&=\lim_{\varepsilon\downarrow 0}\mathbb E\left[\left(\int_{\mathbb R}\varepsilon^{\frac{\gamma^2}{4}}e^{\frac{\gamma}{2}h_\varepsilon(x)}e^{\frac{\gamma}{2}[(\beta-\frac{4}{\gamma})\log|x|_{+}+\alpha\widetilde G_{\mathbb H}(x,i)]}dx\right)^{\frac{2}{\gamma}(Q-\alpha-\frac{1}{2}\beta)}\right]\\
		&=\lim_{\varepsilon\downarrow 0}\mathbb E\left[\left(\int_{\mathbb R} e^{\frac{\gamma}{2}h_\varepsilon(x)-\frac{\gamma^2}{2}\mathbb E[h_\varepsilon(x)^2]} \frac{|x|_{+}^{\gamma\alpha+\frac{\gamma\beta}{2}-2}}{|x-i|^{\gamma\alpha}}dx \right)^{\frac{2}{\gamma}(Q-\alpha-\frac{1}{2}\beta)}\right]\\
		&=\mathbb E\left[\left(\int_{\mathbb R} e^{\frac{\gamma}{2}h_\varepsilon(x)-\frac{\gamma^2}{2}\mathbb E[h_\varepsilon(x)^2]} \frac{|x|_{+}^{\gamma\alpha+\frac{\gamma\beta}{2}-2}}{|x-i|^{\gamma\alpha}}dx \right)^{\frac{2}{\gamma}(Q-\alpha-\frac{1}{2}\beta)}\right]\\
		&=\mathbb E\left[\left(\int_{\mathbb R} e^{\frac{\gamma}{2}h(x)} \frac{|x|_{+}^{\gamma\alpha+\frac{\gamma\beta}{2}-2}}{|x-i|^{\gamma\alpha}}dx \right)^{\frac{2}{\gamma}(Q-\alpha-\frac{1}{2}\beta)}\right]\\
		&=\overline G(\alpha,\beta).
	\end{align*}
	This completes the proof. 
\end{proof}
\subsection{Moments of the conformal radius of $\sle_\kappa(\rho)$ bubbles}
By (\ref{general welding equation}) in Theorem \ref{welding of generalized weight}, 
\begin{equation}\label{CR welding equation}
	\LF_{\mathbb H}^{(\beta_{2W+2},0),(\alpha,i)}(1)\times \mathbf m_\alpha=C_{W}\cdot\int_0^\infty \qd_{1,1}(\alpha,\gamma;\ell)\times \mathcal M^{\disk}_{0,2}(W;1,\ell)d\ell
\end{equation}
for $W>0$ and $\alpha\in \mathbb R$. By definition of $\mathbf m_\alpha$ (\ref{twisted SLE bubble: RN derivative}), 
\begin{equation}
	|\mathbf m_\alpha|=\mathbb E\left[|\psi'_\eta(i)|^{2\Delta_\alpha-2} \right]
\end{equation}
since $\mathbf m$ is a probability measure. Therefore, taking mass on both sides of (\ref{CR welding equation}) yields
\begin{equation}\label{mass equation; general case}
	\left|\LF_{\mathbb H}^{(\beta_{2W+2},0),(\alpha,i)}(1)\right|\cdot \mathbb E\left[|\psi'_\eta(i)|^{2\Delta_\alpha-2} \right] = C_{W}\cdot\int_0^\infty \left|\qd_{1,1}(\alpha,\gamma;\ell)\right|\left| \mathcal M^{\disk}_{0,2}(W;1,\ell)\right|d\ell.
\end{equation}
\begin{Lemma}\label{LCFT 11 total length }
	Fix $\ell>0$ and $q\in \mathbb H$. Let $\gamma,\beta,\alpha$ be such that $\gamma\in (0,2),\beta<Q$ and $\frac{\gamma}{2}-\alpha<\frac{\beta}{2}<\alpha$. Then we have 
	\begin{equation}\label{mass at length l}
		|\LF_{\mathbb H}^{(\beta,0),(\alpha,q)}(\ell)|=|q|^{-2\Delta_\beta}(\Im q)^{\Delta_\beta-2\Delta_\alpha}\cdot 2^{-\frac{\alpha^2}{2}}\ell^{\frac{2}{\gamma}(\frac{1}{2}\beta+\alpha-Q)-1}\cdot \frac{2}{\gamma}\cdot \overline G(\alpha,\beta).
	\end{equation}
	Moreover, for $\mu>0$, $\beta<Q$ and $Q-\alpha <\frac{\beta}{2}<\alpha$, we have
	\begin{equation}\label{laplace trans of outer boundary}
		\LF_{\mathbb H}^{(\beta,0),(\alpha,q)}\left[e^{-\mu\nu_{\phi}(\mathbb R)} \right]=|q|^{-2\Delta_\beta}(\Im q)^{\Delta_\beta-2\Delta_\alpha}2^{-\frac{\alpha^2}{2}} \frac{2}{\gamma}\cdot \overline G(\alpha,\beta)\mu^{\frac{2}{\gamma}(Q-\alpha-\frac{1}{2}\beta)}\Gamma\left(\frac{2}{\gamma}\left(\frac{1}{2}\beta+\alpha-Q \right) \right).	
	\end{equation}
\end{Lemma}
\begin{proof}
	By Lemma \ref{quantum length distribution of LCFT} and Lemma \ref{transformation rule}, for bounded continuous function $f$ on $(0,\infty)$, $\beta<Q$ and $\frac{\gamma}{2}-\alpha<\frac{\beta}{2}<\alpha$, 
\begin{align*}
	\LF_{\mathbb H}^{(\beta,0),(\alpha,q)}[f(\nu_{\phi}(\mathbb R))] &=|q|^{-2\Delta_\beta}(\Im q)^{\Delta_\beta-2\Delta_\alpha}\LF_{\mathbb H}^{(\beta,\infty), (\alpha,i)}[f(\nu_{\phi}(\mathbb R))]  \\
	&=|q|^{-2\Delta_\beta}(\Im q)^{\Delta_\beta-2\Delta_\alpha}\cdot 2^{-\frac{\alpha^2}{2}}\int_0^\infty f(\ell) \ell^{\frac{2}{\gamma}(\frac{1}{2}\beta+\alpha-Q)-1}\cdot \frac{2}{\gamma}\cdot \overline G(\alpha,\beta)d\ell.
\end{align*}
When $f(\ell)=e^{-\mu \ell}$, for $\beta<Q$ and $Q-\alpha<\frac{\beta}{2}<\alpha$, 
\begin{align*}
	\int_{0}^\infty e^{-\mu \ell}\ell^{\frac{2}{\gamma}(\frac{1}{2}\beta+\alpha-Q)-1}d\ell=\mu^{\frac{2}{\gamma}(Q-\alpha-\frac{1}{2}\beta)}\Gamma\left(\frac{2}{\gamma}\left(\frac{1}{2}\beta+\alpha-Q \right) \right).
\end{align*}
This completes the proof. 
\end{proof}
\subsubsection{Special Case: $W=2$}
When $W=2$, we have that $\Delta_{\beta_6}=\Delta_{\gamma-\frac{4}{\gamma}}=2-\frac{8}{\gamma^2}$. By (\ref{mass equation; general case}), 
\begin{equation}
	\left|\LF_{\mathbb H}^{(\gamma-\frac{4}{\gamma} ,0),(\alpha ,i)}(1)\right|\cdot \mathbb E\left[|\psi'_\eta(i)|^{2\Delta_\alpha-2} \right]=C_{2}\cdot \int_0^\infty |\qd_{1,1}(\alpha,\gamma;\ell)|| \mathcal M^{\disk}_{0,2}(2;1,\ell)|d\ell.
\end{equation}
Furthermore, we renormalize the moments of the conformal radius of $\sle_\kappa$ bubbles so that there is no additional multiplicative constant on the right hand side. More specifically, we define the \textit{renormalized moments of the conformal radius} to be 
\begin{equation*}
	\C_2(\alpha): =\frac{\mathbb E\left[|\psi'_\eta(i)|^{2\Delta_\alpha-2} \right]}{C_{2}}
\end{equation*}
and therefore have
\begin{equation*}
	\left|\LF_{\mathbb H}^{(\gamma-\frac{4}{\gamma} ,0),(\alpha ,i)}(1)\right|\cdot \C_2(\alpha)=\int_0^\infty |\qd_{1,1}(\alpha,\gamma; \ell)|| \mathcal M^{\disk}_{0,2}(2;1,\ell)|d\ell.
\end{equation*}
\begin{Prop}[Moments of the conformal radius of $\sle_\kappa$ bubbles, same as  Proposition \ref{main prop: conformal radius}]\label{Prop: conformal radius W=2}
	Fix $W=2,\rho=0$ and $\frac{\gamma}{2}<\alpha<Q+\frac{2}{\gamma}$. Suppose $\eta$ is sampled from $\sle_{\kappa,0}^{\bub}[d\eta|i\in D_\eta(0)]$, then we have 
	\begin{equation}
		\mathbb E\left[|\psi'_\eta(i)|^{2\Delta_\alpha-2} \right]=\frac{\Gamma(\frac{2\alpha}{\gamma})\Gamma(\frac{8}{\kappa}-\frac{2\alpha}{\gamma}+1)}{\Gamma(\frac{8}{\kappa}-1)}.
	\end{equation}
 Consequently, 
 \begin{equation}
     \mathbb E\left[\Rad(D_\eta(0),i)^{2\Delta_\alpha-2} \right] = 2^{2\Delta_\alpha-2}\cdot \frac{\Gamma(\frac{2\alpha}{\gamma})\Gamma(\frac{8}{\kappa}-\frac{2\alpha}{\gamma}+1)}{\Gamma(\frac{8}{\kappa}-1)}.
\end{equation}
\end{Prop}
\begin{proof}
By Lemma \ref{LCFT 11 total length }, when $\alpha>\frac{2}{\gamma}$,
\begin{align}
	\left|\LF_{\mathbb H}^{(\gamma-\frac{4}{\gamma} ,0),(\alpha ,i)}(1)\right|=2^{-\frac{\alpha^2}{2}}\cdot \frac{2}{\gamma}\cdot\overline G(\alpha,\gamma-\frac{4}{\gamma})
\end{align}
and when $\alpha>\frac{\gamma}{2}$, 
\begin{align}
	|\qd_{1,1}(\alpha,\gamma; r)|=2^{-\frac{\alpha^2}{2}}r^{\frac{2}{\gamma}(\alpha-Q)}\frac{2}{\gamma}\overline G(\alpha,\gamma).
\end{align}
By \cite[Proposition $5.1$]{InteofSLE},
\begin{equation*}
	|\mathcal M^{\disk}_{0,2}(2;1,r)|=\frac{(2\pi)^{\frac{4}{\gamma^2}-1}}{(1-\frac{\gamma^2}{4})\Gamma(1-\frac{\gamma^2}{4})^{\frac{4}{\gamma^2}}}(1+r)^{-\frac{4}{\gamma^2}-1}.
\end{equation*} 
Notice that when $\frac{2}{\gamma}<\alpha<Q+\frac{2}{\gamma}$,  
\begin{align*}
	\int_0^\infty \frac{r^{\frac{2}{\gamma}\alpha}}{[(1+r)r]^{\frac{4}{\gamma^2}+1}}dr= B\left(\frac{2}{\gamma}(\alpha-Q)+1,\frac{8}{\gamma^2}-\frac{2\alpha}{\gamma}+1 \right)&=\frac{\Gamma(\frac{2}{\gamma}(\alpha-Q)+1)\Gamma(\frac{8}{\gamma^2}-\frac{2\alpha}{\gamma}+1)}{\Gamma(\frac{4}{\gamma^2}+1)}\\
	&=\frac{\Gamma(\frac{2\alpha}{\gamma}-\frac{4}{\kappa})\Gamma(\frac{8}{\kappa}-\frac{2\alpha}{\gamma}+1)}{\Gamma(\frac{4}{\kappa}+1)},
\end{align*}
where $B(x,y)$ is the Beta function with parameter $x,y$. Therefore, when $\max\lbrace\frac{\gamma}{2},\frac{2}{\gamma} \rbrace <\alpha<Q+\frac{2}{\gamma}$, we have
\begin{align}\label{pre conformal derivative}
	\C_2(\alpha)=\frac{\overline G(\alpha,\gamma)}{\overline G(\alpha,\gamma-\frac{4}{\gamma})}\cdot \frac{(2\pi)^{\frac{4}{\kappa}-1}}{(1-\frac{\kappa}{4})\Gamma(1-\frac{\kappa}{4})^{\frac{4}{\kappa}}}\frac{\Gamma(\frac{2\alpha}{\gamma}-\frac{4}{\kappa})\Gamma(\frac{8}{\kappa}-\frac{2\alpha}{\gamma}+1)}{\Gamma(\frac{4}{\kappa}+1)}.
\end{align}
By shifting relation $(2.30)$ in \cite{RZ22}, 
\begin{align*}
	\frac{\overline G(\alpha,\gamma)}{\overline G(\alpha,\gamma-\frac{4}{\gamma})}=\frac{\frac{\kappa}{4}\Gamma(1-\frac{\kappa}{4})^{\frac{4}{\kappa}}}{(2\pi)^{\frac{4}{\kappa}}2^{1-\frac{8}{\kappa}}}\cdot\frac{\Gamma(\frac{2\alpha}{\gamma}-1)\Gamma(\frac{4}{\kappa})^2}{\Gamma(\frac{8}{\kappa}-1)\Gamma(\frac{4}{\kappa}-1)\Gamma(\frac{2\alpha}{\gamma}-\frac{4}{\kappa})}.
\end{align*}
Therefore, when $\frac{2}{\gamma} <\alpha<Q+\frac{2}{\gamma}$, the renormalized moments of the conformal radius is equal to 
\begin{align}\label{comp of C2}
\begin{split}
	\C_2(\alpha) &=\frac{1}{\pi}\frac{\kappa}{4-\kappa}2^{\frac{8}{\kappa}-2} \frac{\Gamma(\frac{2\alpha}{\gamma}-1)\Gamma(\frac{4}{\kappa})^2}{\Gamma(\frac{8}{\kappa}-1)\Gamma(\frac{4}{\kappa}-1)\Gamma(\frac{2\alpha}{\gamma}-\frac{4}{\kappa})}\frac{\Gamma(\frac{2\alpha}{\gamma}-\frac{4}{\kappa})\Gamma(\frac{8}{\kappa}-\frac{2\alpha}{\gamma}+1)}{\Gamma(\frac{4}{\kappa}+1)}\\
	&= \frac{2^{\frac{8}{\kappa}-2}}{\pi}\frac{\Gamma(\frac{2\alpha}{\gamma}-1)\Gamma(\frac{4}{\kappa})}{\Gamma(\frac{8}{\kappa}-1)\Gamma(\frac{2\alpha}{\gamma}-\frac{4}{\kappa})}\frac{\Gamma(\frac{2\alpha}{\gamma}-\frac{4}{\kappa})\Gamma(\frac{8}{\kappa}-\frac{2\alpha}{\gamma}+1)}{\Gamma(\frac{4}{\kappa}+1)}\\
	&= \frac{2^{\frac{8}{\kappa}-2}}{\pi}\cdot\frac{\kappa}{4}\cdot \frac{\Gamma(\frac{2\alpha}{\gamma}-1)\Gamma(\frac{8}{\kappa}-\frac{2\alpha}{\gamma}+1)}{\Gamma(\frac{8}{\kappa}-1)} \\
	&= \frac{\kappa}{4\sqrt{\pi}}\cdot\frac{\Gamma(\frac{2\alpha}{\gamma}-1)\Gamma(\frac{8}{\kappa}-\frac{2\alpha}{\gamma}+1)}{\Gamma(\frac{4}{\kappa})\Gamma(\frac{4}{\kappa}-\frac{1}{2})}.\\
\end{split}
\end{align}
Notice that the lower bound $\alpha>\frac{2}{\gamma}$ comes from $\Gamma(\frac{2\alpha}{\gamma}-\frac{4}{\kappa})$. However, this term is transitory and will be canceled with a term in $\frac{\overline G(\alpha,\gamma)}{\overline G(\alpha,\gamma-\frac{4}{\gamma})}$. Therefore, by analytic continuation of Gamma function, (\ref{comp of C2}) holds when $\frac{\gamma}{2}<\alpha<Q+\frac{2}{\gamma}$. Therefore, when $\alpha=\gamma$,
\begin{equation*}
	\C_2(\gamma)=\frac{1}{C_{2}}=\frac{\kappa}{4\sqrt{\pi}}\cdot\frac{\Gamma(\frac{8}{\kappa}-1)}{\Gamma(\frac{4}{\kappa})\Gamma(\frac{4}{\kappa}-\frac{1}{2})}.
\end{equation*}
Hence, when $\frac{\gamma}{2}<\alpha<Q+\frac{2}{\gamma}$, 
\begin{equation}
	\mathbb E\left[|\psi'_\eta(i)|^{2\Delta_\alpha-2} \right]=\frac{\C_2(\alpha)}{\C_2(\gamma)}=\frac{\Gamma(\frac{2\alpha}{\gamma})\Gamma(\frac{8}{\kappa}-\frac{2\alpha}{\gamma}+1)}{\Gamma(\frac{8}{\kappa}-1)}. 
\end{equation}
\end{proof}
Next, we verify the Proposition \ref{Prop: conformal radius W=2} by using the Laplace transform of total boundary length $\nu_\phi(\mathbb R)$. As we will see, it will produce the exact same formula. We mention this computation to motivate our calculation of general weight-$W$ case. From now on, let $L_W$ and $R_W$ denote the left and right quantum boundary length of $\mathcal M^{\disk}_{0,2}(W)$ respectively. 
\begin{Lemma}\label{welding lemma laplace}
	Let $\mu>0$ and we have  
\begin{align*}
	\LF_{\mathbb H}^{(\gamma-\frac{4}{\gamma},0), (\alpha,i)}[e^{-\mu\nu_{\phi}(\mathbb R)}]\cdot \C_2(\alpha)= \mathcal M^{\disk}_{0,2}(2)[e^{-\mu R_2}|\qd_{1,1}(\alpha,\gamma;L_2)|]
\end{align*}
\end{Lemma}
	\begin{proof}
	By definition of the conformal welding, the $L_2$ is also equal to outer boundary of $\qd_{1,1}(\alpha,\gamma)$. Therefore, we have
	\begin{align}\label{7.2.12}
	\begin{split}
		\LF_{\mathbb H}^{(\gamma-\frac{4}{\gamma},0),(\alpha,i)}\left[e^{-\mu\nu_{\phi}(\mathbb R)}\right]\cdot \mathbb E\left[|\psi'_\eta(i)|^{2\Delta_\alpha-2} \right] &=C_{2}\cdot \int_0^\infty \mathcal M^{\disk}_{0,2}(2;\ell)[e^{-\mu R_2}]\qd_{1,1}(\alpha,\gamma ;\ell)d\ell\\
		&=C_{2}\cdot \int_0^\infty \mathcal M^{\disk}_{0,2}(2;\ell)\left[e^{-\mu R_2}|\qd_{1,1}(\alpha,\gamma;\ell)|\right]d\ell\\
		&=C_{2}\cdot \mathcal M^{\disk}_{0,2}(2)\left[e^{-\mu R_2}\cdot |\qd_{1,1}(\alpha,\gamma;L_2)| \right]. 
	\end{split}
	\end{align} 
\end{proof}
\begin{proof}[Proof of Proposition \ref{Prop: conformal radius W=2} using Laplace transform]
	We first simplify last line of (\ref{7.2.12}). By (\ref{laplace trans of outer boundary}), when $Q-\alpha<\frac{\gamma}{2}<\alpha$ and $\gamma<Q$, i.e.,  $\alpha>\frac{2}{\gamma}$, 
\begin{align*}
	\mathcal M^{\disk}_{0,2}(2)\left[e^{-\mu R_2}|\qd_{1,1}(\alpha,\gamma ;L_2)|\right]&=2^{-\frac{\alpha^2}{2}}\frac{2}{\gamma}\overline G(\alpha,\gamma)\mathcal M^{\disk}_{0,2}[e^{-\mu R_2}L_2^{\frac{2}{\gamma}(\alpha-Q)}]\\
	&=2^{-\frac{\alpha^2}{2}}\frac{2}{\gamma}\overline G(\alpha,\gamma)\frac{(2\pi)^{\frac{4}{\gamma^2}-1}}{(1-\frac{\gamma^2}{4})\Gamma(1-\frac{\gamma^2}{4})^{\frac{4}{\gamma^2}}}\iint_{0}^{\infty}e^{-\mu \ell}r^{\frac{2}{\gamma}(\alpha-Q)}(\ell+r)^{-\frac{4}{\gamma^2}-1}d\ell dr.
\end{align*}
Let $r=\ell\cdot t$ and $dr=dt\cdot\ell$. We have 
\begin{align*}
	\iint_{0}^{\infty}e^{-\mu \ell}r^{\frac{2}{\gamma}(\alpha-Q)}(\ell+r)^{-\frac{4}{\gamma^2}-1}d\ell dr &=\iint_0^\infty e^{-\mu\ell}(\ell t)^{\frac{2}{\gamma}(\alpha-Q)}(\ell+\ell\cdot t)^{-\frac{4}{\gamma^2}-1} \ell dtd\ell\\
	&=\left(\int_0^\infty \frac{\ell^{\frac{2}{\gamma}(\alpha-Q)} e^{-\mu\ell}}{\ell^{\frac{4}{\gamma^2}}}d\ell\right)\left(\int_0^\infty \frac{t^{\frac{2}{\gamma}(\alpha-Q)}}{(1+t)^{\frac{4}{\gamma^2}+1}}dt\right).
\end{align*}
When $\frac{2}{\gamma}(\alpha-Q)-\frac{4}{\gamma^2}>-1$,  i.e., $\alpha>\frac{4}{\gamma}$, 
\begin{align*}
	\int_0^\infty \frac{\ell^{\frac{2}{\gamma}(\alpha-Q)} e^{-\mu\ell}}{\ell^{\frac{4}{\gamma^2}}}d\ell=\mu^{\frac{8}{\gamma^2}-\frac{2\alpha}{\gamma}}\Gamma\left(\frac{2\alpha}{\gamma}-\frac{8}{\gamma^2} \right).
\end{align*}
Furthermore, when $\frac{2}{\gamma}<\alpha<Q+\frac{2}{\gamma}$,  
\begin{align*}
	\int_0^\infty \frac{t^{\frac{2}{\gamma}(\alpha-Q)}}{(1+t)^{\frac{4}{\gamma^2}+1}}dr = B\left(\frac{2}{\gamma}(\alpha-Q)+1,\frac{8}{\gamma^2}-\frac{2\alpha}{\gamma}+1 \right) &=\frac{\Gamma(\frac{2}{\gamma}(\alpha-Q)+1)\Gamma(\frac{8}{\gamma^2}-\frac{2\alpha}{\gamma}+1)}{\Gamma(\frac{4}{\gamma^2}+1)}\\
	&=\frac{\Gamma(\frac{2\alpha}{\gamma}-\frac{4}{\kappa})\Gamma(\frac{8}{\kappa}-\frac{2\alpha}{\gamma}+1)}{\Gamma(\frac{4}{\kappa}+1)},
\end{align*}
where $B(x,y)$ is the Beta function with parameter $x,y$. To conclude, when $\frac{4}{\gamma} <\alpha<Q+\frac{2}{\gamma}$, 
\begin{align*}
	&\mathcal M^{\disk}_{0,2}(2)[e^{-\mu R_2}|\qd_{1,1}(\alpha,\gamma;L_2)|]\\
	&=2^{-\frac{\alpha^2}{2}}\frac{2}{\gamma}\overline G(\alpha,\gamma)\frac{(2\pi)^{\frac{4}{\kappa}-1}}{(1-\frac{\kappa}{4})\Gamma(1-\frac{\kappa}{4})^{\frac{4}{\kappa}}}\mu^{\frac{8}{\kappa}-\frac{2\alpha}{\gamma}}\Gamma\left(\frac{2\alpha}{\gamma}-\frac{8}{\kappa} \right)\frac{\Gamma(\frac{2\alpha}{\gamma}-\frac{4}{\kappa})\Gamma(\frac{8}{\kappa}-\frac{2\alpha}{\gamma}+1)}{\Gamma(\frac{4}{\kappa}+1)}.
\end{align*}
On the other hand, when $\gamma-\frac{4}{\gamma}<Q$ and $Q-\alpha<\frac{\gamma}{2}-\frac{2}{\gamma}<\alpha$, i.e., $\alpha>\frac{4}{\gamma}$, 
\begin{align*}
	\LF_{\mathbb H}^{(\gamma-\frac{4}{\gamma},0), (\alpha,i)}[e^{-\mu\nu_{\phi}(\mathbb R)}]=2^{-\frac{\alpha^2}{2}}\frac{2}{\gamma}\overline G(\alpha ,\gamma-\frac{4}{\gamma})\mu^{\frac{8}{\kappa}-\frac{2\alpha}{\gamma}}\Gamma\left(\frac{2\alpha}{\gamma}-\frac{8}{\kappa} \right).
\end{align*}
Therefore, when $\frac{4}{\gamma}<\alpha<Q+\frac{2}{\gamma}$, we have 
\begin{align*}
	\C(\alpha)=\frac{\overline G(\alpha,\gamma)}{\overline G(\alpha,\gamma-\frac{4}{\gamma})}\frac{(2\pi)^{\frac{4}{\kappa}-1}}{(1-\frac{\kappa}{4})\Gamma(1-\frac{\kappa}{4})^{\frac{4}{\kappa}}}\frac{\Gamma(\frac{2\alpha}{\gamma}-\frac{4}{\kappa})\Gamma(\frac{8}{\kappa}-\frac{2\alpha}{\gamma}+1)}{\Gamma(\frac{4}{\kappa}+1)},
\end{align*}
which is identical to our previous calculation $(\ref{pre conformal derivative})$. Notice that by analytic continuation, we can again extend the range of $\alpha$ to $(\frac{\gamma}{2},Q+\frac{2}{\gamma})$ in the end.
\end{proof}
\subsubsection{General weight-$W$ case}\label{subsubsec: general weight CR}
In this section, we compute the moments of the conformal radius of $\sle^{\bub}_{\kappa,0}(W-2)[d\eta|i\in D_\eta(0)]$ for general $W>0$. 
\begin{Lemma}
	Let $\mu>0$ and we have  
	\begin{equation}\label{laplace welding}
		\LF_{\mathbb H}^{(\beta_{2W+2},0),(\alpha,i)}\left[e^{-\mu\nu_{\phi}(\mathbb R)}\right]\mathbb E\left[|\psi'_\eta(i)|^{2\Delta_\alpha-2} \right] =C_{W}\cdot\mathcal M^{\disk}_{0,2}(W)\left[e^{-\mu R_W}\cdot |\qd_{1,1}(\alpha,\gamma;L_W)|\right].
	\end{equation}
\end{Lemma}
\begin{proof}
The proof is identical to that of Lemma \ref{welding lemma laplace}. 
\end{proof}
Similarly as before, define the \textit{generalized renormalized moments of the conformal radius} $\C(\alpha,W)$ to be the following: 
\begin{equation}\label{def:renormalized CR}
	\C(\alpha,W):=\frac{\mathbb E\left[|\psi'_\eta(i)|^{2\Delta_\alpha-2} \right]}{C_{W}}.
\end{equation}
Therefore, we have 
\begin{equation}
    \LF_{\mathbb H}^{(\beta_{2W+2},0),(\alpha,i)}\left[e^{-\mu\nu_{\phi}(\mathbb R)}\right]\cdot\C(\alpha,W) = \mathcal M^{\disk}_{0,2}(W)\left[e^{-\mu R_W}\cdot |\qd_{1,1}(\alpha,\gamma;L_W)|\right].
\end{equation}

\begin{Prop}\label{conformal radius: general case}
	Fix $\gamma\in (0,2)$. When $\beta_{2W+2}$ and $\alpha$ satisfy $0<\beta_{2W+2}<\gamma$ and $Q-\frac{\beta_{2W+2}}{2}<\alpha<Q+\frac{\gamma}{2}$,  we have
	\begin{align}
 \begin{split}
 	&\mathbb E\left[|\psi'_\eta(i)|^{2\Delta_\alpha-2} \right]\\
 	 &=\frac{\overline G(\alpha,\gamma)}{\overline G(\alpha, \gamma-\frac{2W}{\gamma})}\frac{\overline G(\gamma,\gamma-\frac{2W}{\gamma})}{\overline G(\gamma , \gamma)}\frac{\int_0^\infty\mu_1^{\frac{2}{\gamma}(Q-\alpha)}\left(\frac{\partial}{\partial \mu_1} R(\beta_W;\mu_1,1)\right)d\mu_1}{\Gamma(\frac{2}{\gamma}(Q-\alpha)+1)\Gamma(\frac{2}{\gamma}(\alpha-\frac{W+2}{\gamma}))}\cdot \frac{\Gamma(\frac{2}{\gamma}(Q-\gamma)+1)\Gamma(\frac{2}{\gamma}(\gamma-\frac{W+2}{\gamma}))}{\int_0^\infty\mu_1^{\frac{2}{\gamma}(Q-\gamma)}\left(\frac{\partial}{\partial \mu_1} R(\beta_W;\mu_1,1)\right)d\mu_1}.
 \end{split}
\end{align}
\end{Prop}
\begin{Cor}\label{intermediate conformal derivative equation}
	Let $\gamma,\alpha$ be such that $\alpha>\frac{\gamma}{2}$, and we have
	\begin{align*}
		\C(\alpha,W)\cdot \LF_{\mathbb H}^{(\beta_{2W+2},0),(\alpha,i)}\left[e^{-\nu_{\phi}(\mathbb R)}\right]=2^{-\frac{\alpha^2}{2}}\frac{2}{\gamma}\overline G(\alpha,\gamma)\mathcal M^{\disk}_{0,2}(W)\left[e^{- R_W}\cdot L_W^{\frac{2}{\gamma}(\alpha-Q)}\right].
    \end{align*}
\end{Cor}
\begin{proof}
	By $(\ref{mass at length l})$ and definition of $\qd_{1,1}(\alpha,\gamma;\ell)$, when $\frac{\gamma}{2}-\alpha<\frac{\gamma}{2}<\alpha$ and $\gamma<Q$, i.e., $\alpha>\frac{\gamma}{2}$, we have
	\begin{align*}
		|\qd_{1,1}(\alpha,\gamma;L_W)|=2^{-\frac{\alpha^2}{2}}L_W^{\frac{2}{\gamma}(\frac{1}{2}\gamma +\alpha-Q)-1}\frac{2}{\gamma}\overline G(\alpha,\gamma).
	\end{align*}
	The statement then follows directly from $(\ref{laplace welding})$.  
\end{proof}
\begin{Lemma}\label{Lemma: laplace of M02W}
	When $W\in (0,\frac{\gamma^2}{2})$ and $\alpha<Q+\frac{\gamma}{2}$, we have 
	\begin{align*}
		\mathcal M^{\disk}_{0,2}(W)\left[L_W^{\frac{2}{\gamma}(\alpha- Q)}e^{-R_W} \right]=\frac{\gamma}{2(\beta_W-Q) \Gamma\left(\frac{2}{\gamma}(Q-\alpha)+1\right)}\int_0^\infty\mu_1^{\frac{2}{\gamma}(Q-\alpha)}\left( \frac{\partial}{\partial \mu_1} R(\beta_W;\mu_1,1)\right)d\mu_1. 
	\end{align*}
\end{Lemma}
\begin{proof}
	By \cite[Proposition $3.6$]{InteofSLE}, when $W\in (0,\frac{\gamma^2}{2})$ and $\beta_W=Q+\frac{\gamma}{2}-\frac{W}{\gamma}\in (Q,Q+\frac{\gamma}{2})$, 
	\begin{equation*}
		\mathcal M^{\disk}_{0,2}(W)\left[e^{-\mu_1 L_W-R_W} \right]=-\frac{\gamma}{2(\beta_W-Q)}R(\beta_W;\mu_1,1). 
	\end{equation*}
	Taking partial derivatives on both sides with respect to $\mu_1$ and we have  
	\begin{align*}
		\mathcal M^{\disk}_{0,2}(W)\left[L e^{-\mu_1 L_W-R_W}\right]=\frac{\gamma}{2(\beta_W-Q)}\left(\frac{\partial}{\partial \mu_1} R(\beta_W;\mu_1,1)\right). 
	\end{align*}
	Next, for fixed real number $a>-1$, we integrate the above equation against $\mu_1^a$ on both sides. By Fubini's theorem, 
	\begin{align*}
		\int_0^\infty \mu_1^a \mathcal M^{\disk}_{0,2}(W)\left[L_W e^{-\mu_1 L_W- R_W}\right]d\mu_1 &=\mathcal M^{\disk}_{0,2}(W)\left[L_W e^{-R_W}\int_0^\infty \mu_1^a e^{-\mu_1 L_W}d\mu_1 \right]\\
		&=\Gamma(a+1)\mathcal M^{\disk}_{0,2}(W)\left[L_W^{-a}e^{-R_W} \right]. 
	\end{align*}
	Let $a=\frac{2}{\gamma}(\alpha-Q)$. When  $\alpha<Q+\frac{\gamma}{2}$, i.e., $\frac{2}{\gamma}(Q-\alpha)>-1$, we have 
	\begin{align*}
		\int_0^\infty \mu_1^{\frac{2}{\gamma}(Q-\alpha)} \mathcal M^{\disk}_{0,2}(W)\left[L e^{-\mu_1 L_W-R_W}\right]d\mu_1=\Gamma\left(\frac{2}{\gamma}(Q-\alpha)+1\right)\mathcal M^{\disk}_{0,2}(W)\left[L_W^{\frac{2}{\gamma}(\alpha- Q)}e^{-R_W} \right]. 
	\end{align*}
	Therefore, when $\alpha<Q+\frac{\gamma}{2}$ and $0<W<\frac{\gamma^2}{2}$, we have
	\begin{align*}
		\mathcal M^{\disk}_{0,2}(W)\left[L_W^{\frac{2}{\gamma}(\alpha- Q)}e^{-R_W} \right]=\frac{\gamma}{2(\beta_W-Q) \Gamma\left(\frac{2}{\gamma}(Q-\alpha)+1\right)}\int_0^\infty\mu_1^{\frac{2}{\gamma}(Q-\alpha)}\left( \frac{\partial}{\partial \mu_1} R(\beta_W;\mu_1,1)\right)d\mu_1. 
	\end{align*}
\end{proof}
\begin{Lemma}\label{Lemma:pre C}
	Fix $\gamma\in (0,2)$. When $\beta_{2W+2}$ and $\alpha$ satisfy $0<\beta_{2W+2}<\gamma$ and $Q-\frac{\beta_{2W+2}}{2}<\alpha<Q+\frac{\gamma}{2}$, we have 
\begin{align}\label{formula of C(alpha, W)}
	\C(\alpha ,W)=\frac{\overline G(\alpha,\gamma)}{\overline G(\alpha, \gamma-\frac{2W}{\gamma})}\frac{\gamma\int_0^\infty\mu_1^{\frac{2}{\gamma}(Q-\alpha)}\left( \frac{\partial}{\partial \mu_1} R(\beta_W;\mu_1,1)\right)d\mu_1}{(\gamma-\frac{2W}{\gamma})\Gamma(\frac{2}{\gamma}(Q-\alpha)+1)\Gamma(\frac{2}{\gamma}(\alpha-\frac{W+2}{\gamma}))}.
\end{align}
\end{Lemma}
\begin{proof}
	By Lemma \ref{LCFT 11 total length }, Corollary \ref{intermediate conformal derivative equation} and Lemma \ref{Lemma: laplace of M02W}, when
\begin{equation}\label{cases:cr}
\begin{cases}
			 \frac{\gamma}{2}<\alpha<Q+\frac{\gamma}{2},\\
			 0<W<\frac{\gamma^2}{2}, i.e., 0<\beta_{2W+2}<\gamma,\\
            \beta_{2W+2}<Q,\\
			Q-\alpha<\frac{\beta_{2W+2}}{2}<\alpha,
		\end{cases}
\end{equation}
we have
	\begin{align}
	\begin{split}
		\C(\alpha ,W)&=\frac{\overline G(\alpha,\gamma)}{\overline G(\alpha,\beta_{2W+2})}\frac{\gamma }{2(\beta_W-Q)\Gamma(\frac{2}{\gamma}(Q-\alpha)+1)\Gamma(\frac{2}{\gamma}(\frac{1}{2}\beta_{2W+2}+\alpha-Q))}\cdot\int_0^\infty\mu_1^{\frac{2}{\gamma}(Q-\alpha)}\left( \frac{\partial}{\partial \mu_1} R(\beta_W;\mu_1,1)\right)d\mu_1\\
		&=\frac{\overline G(\alpha,\gamma)}{\overline G(\alpha, \gamma-\frac{2W}{\gamma})}\frac{\gamma\int_0^\infty\mu_1^{\frac{2}{\gamma}(Q-\alpha)}\left( \frac{\partial}{\partial \mu_1} R(\beta_W;\mu_1,1)\right)d\mu_1}{(\gamma-\frac{2W}{\gamma})\Gamma(\frac{2}{\gamma}(Q-\alpha)+1)\Gamma(\frac{2}{\gamma}(\alpha-\frac{W+2}{\gamma}))}.
	\end{split}
	\end{align}
 Notice that (\ref{cases:cr}) implies $0<\beta_{2W+2}<\gamma$ and $\frac{Q}{2}<\alpha<Q+\frac{\gamma}{2}$. Since $\frac{W+2}{\gamma}=Q-\frac{
 \beta_{2W+2}}{2}$, by analytic continuation of $\Gamma(\frac{2}{\gamma}(\alpha-\frac{W+2}{\gamma}))$, the lower bound of $\alpha$ can be extended to $\alpha>Q-\frac{\beta_{2W+2}}{2}$. Therefore, the statement is proved.
\end{proof}
\begin{proof}[Proof of Proposition \ref{conformal radius: general case}]
By analytic continuation of $\Gamma(\frac{2}{\gamma}(\alpha-\frac{W+2}{\gamma}))$ , we can further relax the range of $\alpha$ and $\beta_{2W+2}$ to $\alpha\in (\frac{\gamma}{2}, Q+\frac{\gamma}{2})$ and $\beta_{2W+2}\in (0,\gamma)$ as long as $\frac{2}{\gamma}(\alpha-\frac{W+2}{\gamma})\in\bigcup_{n\geq 0,n\in \mathbb Z} (-2n-2,-2n-1)$. Here, we extend to the range of $\alpha$ so that it contains the point $\gamma$. Therefore, by simple computation, 
\begin{align}
 \begin{split}
 	&\mathbb E\left[|\psi'_\eta(i)|^{2\Delta_\alpha-2} \right]\\ &=\frac{\C(\alpha ,W)}{\C(\gamma ,W)}\\
  &=\frac{\overline G(\alpha,\gamma)}{\overline G(\alpha, \gamma-\frac{2W}{\gamma})}\frac{\overline G(\gamma,\gamma-\frac{2W}{\gamma})}{\overline G(\gamma , \gamma)}\frac{\int_0^\infty\mu_1^{\frac{2}{\gamma}(Q-\alpha)}\left(\frac{\partial}{\partial \mu_1} R(\beta_W;\mu_1,1)\right)d\mu_1}{\Gamma(\frac{2}{\gamma}(Q-\alpha)+1)\Gamma(\frac{2}{\gamma}(\alpha-\frac{W+2}{\gamma}))}\cdot \frac{\Gamma(\frac{2}{\gamma}(Q-\gamma)+1)\Gamma(\frac{2}{\gamma}(\gamma-\frac{W+2}{\gamma}))}{\int_0^\infty\mu_1^{\frac{2}{\gamma}(Q-\gamma)}\left(\frac{\partial}{\partial \mu_1} R(\beta_W;\mu_1,1)\right)d\mu_1}.
 \end{split}
\end{align}
Again, by analytic continuation of Gamma function, we see that the above equation holds as long as  $0<\beta_{2W+2}<\gamma$ and $Q-\frac{\beta_{2W+2}}{2}<\alpha<Q+\frac{\gamma}{2}$.
\end{proof}
\subsection{The bulk-boundary correlation function in the LCFT}
In this section, we derive a formula linking the two-pointed correlation function in the LCFT to the joint law of left, right quantum boundary length and total quantum area of $\mathcal M^{\disk}_{0,2}(W)$. First, we recall the definition of the quantum disk with only one bulk insertion point.  
\begin{Def}[{\cite[Definition 4.2]{FZZ}}]
	For $\alpha\in \mathbb R$, let $\phi$ be sampled from $\LF_{\mathbb H}^{(\alpha,i)}$. We denote $\mathcal M^{\disk}_{1,0}(\alpha)$ as the infinite measure described the law of quantum surface $(\mathbb H,\phi,i)$. 
\end{Def}
\begin{Th}[{\cite[Proposition 2.8]{FZZ},\cite{Rem20}}]\label{Th: U0alpha}
	For $\alpha>\frac{\gamma}{2}$, let $h$ be sampled from $P_{\mathbb H}$ and let $\widetilde{\phi}(z)=h(z)-2Q\log|z|_{+}+\alpha G_{\mathbb H}(z,i)$. Let $\overline U_0(\alpha):=\mathbb E\left[\nu_{\widetilde \phi}(\mathbb R)^{\frac{2}{\gamma}(Q-\alpha)}\right]$ where the expectation is taken over $P_{\mathbb H}$. Then we have 
	\begin{align}
		\overline U_0(\alpha)=\left(\frac{2^{-\frac{\gamma\alpha}{2}}2\pi}{\Gamma(1-\frac{\gamma^2}{4})}\right)^{\frac{2}{\gamma}(Q-\alpha)}\Gamma\left(\frac{\gamma\alpha}{2}-\frac{\gamma^2}{4}\right)\qquad\text{for all $\alpha>\frac{\gamma}{2}$.}
	\end{align} 
\end{Th}
\begin{Prop}[Same as Proposition \ref{Prop:bulk-boundary}]\label{Prop:two-pointed cor}
	Fix $\gamma\in (0,2)$ and $\mu,\mu_\partial>0$. When $\beta_{2W+2}$ and $\alpha$ satisfy $0<\beta_{2W+2}<\gamma$ and $Q-\frac{\beta_{2W+2}}{2}<\alpha<Q$, we have
	\begin{align}
	\begin{split}
     G_{\mu,\mu_\partial}(\alpha,\beta_{2W+2})&=\LF_{\mathbb H}^{(\beta_{2W+2},0),(\alpha,i)}\left[e^{-\mu_\partial \nu_{\phi}(\mathbb R) -\mu \mu_{\phi}(\mathbb H)}\right]\\ &=\C(\alpha,W)^{-1}\frac{2}{\gamma}2^{-\frac{\alpha^2}{2}}\overline U_0(\alpha)\frac{2}{\Gamma(\frac{2}{\gamma}(Q-\alpha))}\left(\frac{1}{2}\sqrt{\frac{\mu}{\sin(\pi\gamma^2/4)}}\right)^{\frac{2}{\gamma}(Q-\alpha)}\times\\
	&\mathcal M_{0,2}^{\disk}(W)\left[e^{-\mu_\partial R_W-\mu A_W}\cdot K_{\frac{2}{\gamma}(Q-\alpha)}\left(L_W\sqrt{\frac{\mu}{\sin(\pi\gamma^2/4)}} \right) \right],
	\end{split}
\end{align}
where $L_W,R_W$ and $A_W$ denote the left, right (quantum) boundary length and total quantum area of $\mathcal M^{\disk}_{0,2}(W)$ respectively and $\C(\alpha,W)$ is the renormalized moments of the conformal radius taking formula (\ref{formula of C(alpha, W)}).   
\end{Prop}
\begin{proof}
	For $\mu_\partial,\mu>0$, we have that
	\begin{align}
		\begin{split}
			\LF_{\mathbb H}^{(\beta_{2W+2},0),(\alpha,i)}[e^{-\mu_\partial \nu_{\phi}(\mathbb R)-\mu\mu_{\phi}(\mathbb H)}]\cdot \C(\alpha,W)&=\int_0^\infty \mathcal M_{0,2}^{\disk}(W;\ell)\left[e^{-\mu_\partial R_W-\mu A_W}\right]\cdot \qd_{1,1}(\alpha,\gamma ;\ell)\left[e^{-\mu A_{1,1}} \right]d\ell\\
			&=\int_0^\infty \mathcal M_{0,2}^{\disk}(W;\ell)\left[e^{-\mu_\partial R_W-\mu A_W}\cdot \qd_{1,1}(\alpha,\gamma;\ell)\left[e^{-\mu A_{1,1}} \right] \right]d\ell\\
			&= \mathcal M_{0,2}^{\disk}(W)\left[e^{-\mu_\partial R_W-\mu A_W}\cdot\qd_{1,1}(\alpha,\gamma ;L_W)\left[e^{-\mu A_{1,1}} \right] \right],
		\end{split}
	\end{align}
	where $A_{1,1}$ is the total quantum area of $\qd_{1,1}(\alpha,\gamma,\ell)$. Next, notice that 
	\begin{align}\label{eqn: qd11 calc}
		\begin{split}
			\qd_{1,1}(\alpha,\gamma;\ell)\left[e^{-\mu A_{1,1}} \right] &=\left|\qd_{1,1}(\alpha,\gamma ;\ell)\right|\cdot \qd_{1,1}(\alpha,\gamma ;\ell)^{\#}\left[e^{-\mu A_{1,1}} \right]\\
	&=\ell\cdot \left|\mathcal M^{\disk}_{1,0}(\alpha ;\ell)\right|\cdot \mathcal M^{\disk}_{1,0}(\alpha;\ell)^{\#}\left[e^{-\mu A_{1,0}} \right]\\
	&=\ell\cdot \mathcal M^{\disk}_{1,0}(\alpha ;\ell)\left[e^{-\mu A_{1,0}}\right],\\
		\end{split}
	\end{align}
	where $A_{1,0}$ is the total quantum area of $\mathcal M^{\disk}_{1,0}(\alpha;\ell)$. The (\ref{eqn: qd11 calc}) follows from the fact that $\qd_{1,1}(\alpha,\gamma ;\ell)^{\#}$ and $\mathcal M^{\disk}_{1,0}(\alpha ;\ell)^{\#}$ are the same probability measure if we ignore the boundary marked point. By \cite[Proposition 4.20]{FZZ}, when $\alpha\in \left(\frac{\gamma}{2},Q \right)$,
	\begin{align}
		\begin{split}
			\mathcal M^{\disk}_{1,0}(\gamma,\alpha;\ell)\left[e^{-\mu A_{1,0}}\right]=\frac{2}{\gamma}2^{-\frac{\alpha^2}{2}}\overline U_0(\alpha)\ell^{-1}\frac{2}{\Gamma(\frac{2}{\gamma}(Q-\alpha))}\left(\frac{1}{2}\sqrt{\frac{\mu}{\sin(\pi\gamma^2/4)}}\right)^{\frac{2}{\gamma}(Q-\alpha)}K_{\frac{2}{\gamma}(Q-\alpha)}\left(\ell\sqrt{\frac{\mu}{\sin(\pi\gamma^2/4)}} \right),
		\end{split}
	\end{align}
	where $K_\nu(x)$ is the modified Bessel function of second kind. Precisely, 
\begin{equation}\label{eqn:Bessel of second kind}
	K_\nu(x):=\int_0^\infty e^{-x\cosh t}\cosh(\nu t)dt\qquad\text{for $x>0$ and $\nu\in \mathbb R$}.
\end{equation}
Therefore, when $\alpha\in \left(\frac{\gamma}{2},Q \right)$ and $\mu>0$,
\begin{align}
	\qd_{1,1}(\alpha,\gamma;\ell)\left[e^{-\mu A_{1,1}}\right]=\frac{2}{\gamma}2^{-\frac{\alpha^2}{2}}\overline U_0(\alpha)\frac{2}{\Gamma(\frac{2}{\gamma}(Q-\alpha))}\left(\frac{1}{2}\sqrt{\frac{\mu}{\sin(\pi\gamma^2/4)}}\right)^{\frac{2}{\gamma}(Q-\alpha)}K_{\frac{2}{\gamma}(Q-\alpha)}\left(\ell\sqrt{\frac{\mu}{\sin(\pi\gamma^2/4)}} \right).
\end{align}
Finally, together with Corollary \ref{conformal radius: general case}, we see that when $\beta_{2W+2}$ and $\alpha$ satisfy $0<\beta_{2W+2}<\gamma$ and $Q-\frac{\beta_{2W+2}}{2}<\alpha<Q$, 
\begin{align}
	\begin{split}
		\LF_{\mathbb H}^{(\beta_{2W+2},0),(\alpha,i)}\left[e^{-\mu_\partial \nu_{\phi}(\mathbb R) -\mu \mu_{\phi}(\mathbb H)}\right]&=\C(\alpha,W)^{-1}\frac{2}{\gamma}2^{-\frac{\alpha^2}{2}}\overline U_0(\alpha)\frac{2}{\Gamma(\frac{2}{\gamma}(Q-\alpha))}\left(\frac{1}{2}\sqrt{\frac{\mu}{\sin(\pi\gamma^2/4)}}\right)^{\frac{2}{\gamma}(Q-\alpha)}\times\\
	&\mathcal M_{0,2}^{\disk}(W)\left[e^{-\mu_\partial R_W-\mu A_{W}}K_{\frac{2}{\gamma}(Q-\alpha)}\left(L\sqrt{\frac{\mu}{\sin(\pi\gamma^2/4)}} \right) \right]. 
	\end{split}
\end{align}
This finishes the proof. 
\end{proof}
\begin{rem}\label{remark:R} 
For $\beta_W\in (\frac{2}{\gamma},Q)$ and $W=\gamma(Q+\frac{\gamma}{2}-\beta_W)$, with $A_W$, $L_W$ and $R_W$ being the total area, left boundary and right boundary of the corresponding weight-$W$, two-pointed quantum disk $\mathcal{M}^{\disk}_{0,2}(W)$ respectively, define
\begin{equation*}
R_{\mathrm{bulk}}(\beta_W;\mu_1,\mu_2):=\frac{2(Q-\beta_W)}{\gamma}\mathcal M^{\disk}_{0,2}(W)\left[e^{-A_W-\mu_1 L_W-\mu_2 R_W} -1 \right],
	\end{equation*}
which is the same as \cite[(1.14)]{ARSZ23}. Using the exact same argument as in \cite[Proposition $3.6$]{InteofSLE}, when $W\in (0,\frac{\gamma^2}{2})$ and $\beta_W = Q+\frac{\gamma}{2}-\frac{W}{\gamma}\in (Q,Q+\frac{\gamma}{2})$, 
we have 
\[\mathcal M^{\disk}_{0,2}(\gamma^2-W)\left[e^{-A_{\gamma^2-W}-\mu_1 L_{\gamma^2-W}-\mu_2 R_{\gamma^2-W}} -1 \right] \mathcal M^{\disk}_{0,2}(W)\left[e^{-A_W -\mu_1 L_W-\mu_2 R_W} \right] =-\frac{\gamma^2}{4(\beta_W-Q)^2}.\]
Therefore, when $W\in (0,\frac{\gamma^2}{2})$ and $\beta_W\in (Q,Q+\frac{\gamma}{2})$, we have 
\begin{equation*}  
		\mathcal M^{\disk}_{0,2}(W)\left[e^{-A_W-\mu_1 L_W-\mu_2 R_W} \right]=\frac{\gamma}{2(Q-\beta_W)}
  R_{\mathrm {bulk} }(2Q-\beta_W;\mu_1,\mu_2)^{-1}. 
	\end{equation*}
Notice that $2Q-\beta_W = \beta_{\gamma^2-W}$. The exact formula of $R_{\mathrm {bulk} }$ is obtained in~\cite[Theorem 1.3]{ARSZ23}, which in turn yields the exact formula for $G_{\mu,\mu_\partial}(\alpha,\beta_{2W+2})$ in ~\cite[Section 4.3]{ARSZ23}.
\end{rem}
\section{Outlook and Future Research}\label{sec:open problems and future research}
In the last section, we discuss several conjectures that arise naturally from the contexts of this paper. 
\subsection{Generalized $\sle$ bubbles on $\mathbb H$: single case}
As natural generalizations of Theorem \ref{Conditional Welding} and Theorem \ref{Main welding theorem}, we can consider the case when $\qd_{0,1}$ has one general boundary insertion, i.e., $\qd_{0,1}(\gamma,\alpha)$ in Definition \ref{def:general qd}. For the sake of completeness, we provide two conjectures: one with the bulk insertion and one without. Although our discussion will be centered around the Conjecture \ref{conj:gen_SLE_bubble}.
\begin{Conj}\label{conj:cond_gen_SLE_bubble}
	Fix $W_1\geq \frac{\gamma^2}{2}$ and $W>2$. There exist a $\sigma$-finite infinite measure $\sle^{\bub}_{\kappa,0}(W,W_1)$ on $\bubble_{\mathbb H}(0)$ and some constant $C\in (0,\infty)$ such that suppose $\phi\times \eta_{W,W_1}$ is sampled from
	\begin{equation}
		C\cdot\LF_{\mathbb H}^{(\beta_{2W_1+W},0)}(d\phi) \times \sle^{\bub}_{\kappa,0}(W,W_1)[d\eta_{W,W_1}|i\in D_{\eta_{W,W_1}(0)}],
	\end{equation}
	then the law of $(D_{\eta_{W,W_1}}(0),\phi, 0)$ and $(D_{\eta_{W,W_1}}(\infty),\phi, 0^-,0^+)$ viewed as a pair of marked quantum surface is equal to 
	\begin{equation}
		\int_0^\infty \mathcal M^{\disk}_{0,2}(W_1;\cdot,\ell)\times\qd_{1,1}(\gamma,\beta_{W};\ell)d\ell. 
	\end{equation} 
\end{Conj}
\begin{figure}[H]\
	\includegraphics[scale=0.9]{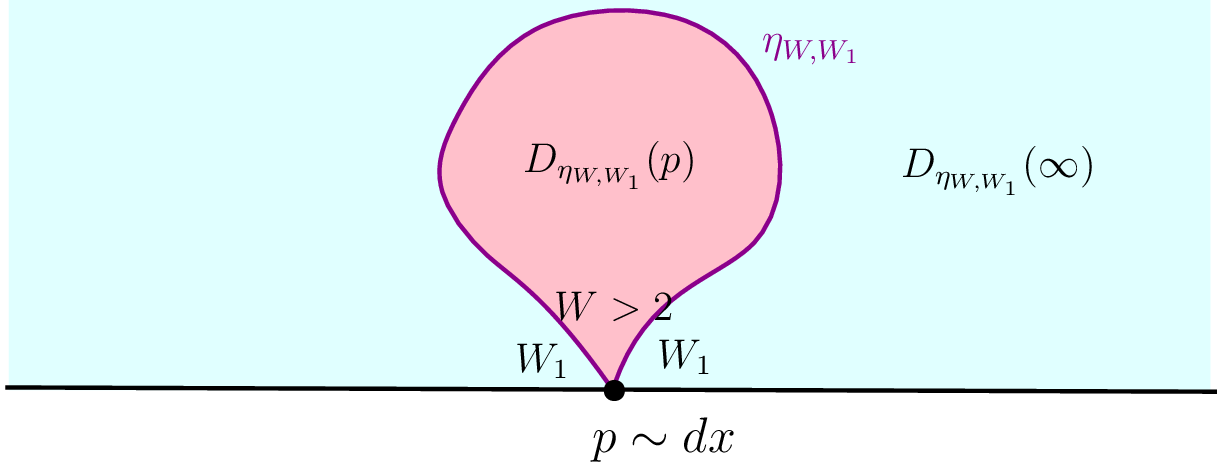}
	\centering
	\caption{Illustration of welding equation (\ref{eqn:ue in conj}) in Conjecture \ref{conj:gen_SLE_bubble}: first sample a root point $p$ according to Lebesgue measure $dx$ on $\mathbb R$, then sample $(\phi,\eta)$ according to the product measure $\LF^{(\beta_{2W_1+W},p)}_{\mathbb H}(d\phi)\times\sle_{\kappa,p}^{\bub}(W,W_1)(d\eta)$. The resulting quantum surface $(\mathbb H,\phi,\eta,p)/\sim_\gamma$ has the law of $C\int_0^\infty\mathcal{M}^{\disk}_{0,2}(W_1;\cdot,\ell)\times \qd_{0,1}(\gamma,\beta_W;\ell)d\ell$ after uniform embedding.}
	\label{fig:Conj_gen_weight}
\end{figure}
\begin{Conj}\label{conj:gen_SLE_bubble}
	Fix $W_1\geq \frac{\gamma^2}{2}$ and $W>2$. There exist a $\sigma$-finite infinite measure $\sle^{\bub}_{\kappa,p}(W,W_1)$ on $\bubble_{\mathbb H}(p)$ and some constant $C\in (0,\infty)$ such that
\begin{equation}\label{eqn:ue in conj}
    \mathbf{m}_{\mathbb H}\ltimes \left(\int_0^\infty\mathcal{M}^{\disk}_{0,2}(W_1;\cdot,\ell)\times \qd_{0,1}(\gamma,\beta_W;\ell)d\ell\right) = C\cdot\LF^{(\beta_{2W_1+W},p)}_{\mathbb H}\times\sle_{\kappa,p}^{\bub}(W,W_1)dp.
 \end{equation}
 Furthermore, there exists some constant $C\in (0,\infty)$ such that 
\begin{equation}\label{eqn:ue in conj fixed root}
    \mathbf{m}_{\mathbb H,0}\ltimes\left(\int_0^\infty \mathcal M^{\disk}_{0,2}(W_1;\cdot,\ell)\times \qd_{0,1}(\gamma,\beta_W;\ell) d\ell\right) = C\cdot \LF_{\mathbb H}^{(\beta_{2W_1+W},0)}(d\phi)\times\sle_{\kappa,0}^{\bub}(W,W_1),
\end{equation}
where $\mathbf{m}_{\mathbb H,0}$ is a Haar measure on $\conf(\mathbb H,0)$, i.e., the group of conformal automorphisms of $\mathbb H$ fixing $0$.
\end{Conj}
In Conjecture \ref{conj:gen_SLE_bubble}, by the quantum triangle welding and the induction techniques developed in Section \ref{sec:field law}, we can show that 1) $\phi$ has the law of $C\cdot\LF_{\mathbb H}^{(\beta_{2W_1+W},0)}$, and 2) the welding interface $\eta_{W,W_1}$ is independent of $\phi$. 
\par
However, we have almost zero understanding on the law of $\eta_{W,W_1}$, i.e., $\sle^{\bub}_{\kappa,0}(W,W_1)$. Recall that in Zhan's limiting constructions of $\sle_\kappa(\rho)$ bubbles, one takes the weak limit of chordal $\sle_\kappa(\rho)$ under suitable rescaling. Therefore, in LQG frameworks, we take ``quantum version'' of the limit by 1) conditioning on the (one-side) quantum boundary length of $\mathcal M^{\disk}_{1,2}(2)$ goes to zero 2) constructing a coupling with the limiting picture so that, with high probability, the random domains match. 
\par
Nonetheless, this technique will not work in the case of Conjecture \ref{conj:gen_SLE_bubble}, or in a more straightforward way, $\eta_{W,W_1}$ is not the weak limit of chordal $\sle_\kappa(W-2,W_1-2)$ under suitable rescaling. Suppose one takes $\mathcal M^{\disk}_{1,2}(W)$ and then conditioning on the (one-side) quantum boundary length goes to zero, the limiting quantum surface will always be the same; the boundary marked point is always quantum typical (cf. \cite[Appendix A]{MSW21}). In other words, we will always get $\sle_{\kappa,0}^{\bub}(W_1-2)$. Therefore, shrinking (one-side) quantum boundary length and coupling will only work for $\mathcal M^{\disk}_{1,2}(2)$. 
\begin{figure}
	\includegraphics[scale=0.8]{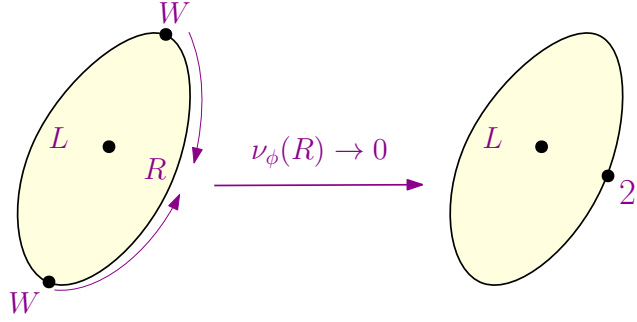}
	\centering
	\caption{On the LHS, we have $\mathcal M^{\disk}_{1,2}(W)$. Conditioning on the quantum boundary length of the right arc $R$ shrinks to zero, we will get a $\mathcal M^{\disk}_{1,1}(2)$, i.e., $\qd_{1,1}$. Notice that in \cite[Appendix A]{MSW21}, the weight $W$ is in the restricted range. However, we believe that this is only a technical barrier and will not affect the overall outcome.}
	\label{fig:shrink_left_boundary}
\end{figure}
Hence, one interesting question is that \emph{how to describe the law of $\eta_{W,W_1}$ in Conjecture \ref{conj:gen_SLE_bubble}?} If better, \emph{what is its corresponding Lowener evolution (driving function)?} 
\par
Also, going back to the Euclidean settings, in Zhan's constructions of $\sle_\kappa(\rho)$ bubbles, one takes the weak limit of $\sle^{\mathbb H}_{\kappa,(\varepsilon;\varepsilon^+)\to 0}(\rho)$ or $\sle^{\mathbb H}_{\kappa,(0;0^-)\to \varepsilon}(\rho)$ under suitable rescaling. Either way, that single force point of $\sle_\kappa$ is on the outside (see Figure \ref{fig:SLEbubble_convergence}).
\par
Hence, \emph{what if you have two force points?} In other words, what if we take the weak limit of $\sle^{\mathbb H}_{\kappa,(0;0^-,0^+)\to \varepsilon}(\rho_-,\rho_+)?$ I conjecture that it is the $\sle^{\bub}_{\kappa,0}(\rho_-)$. Similarly, if we take the weak limit of $\sle^{\mathbb H}_{\kappa,(\varepsilon;\varepsilon^-,\varepsilon^+)\to 0}(\rho_-,\rho_+)$, then it is $\sle^{\bub}_{\kappa,0}(\rho_+)$.  
\par 
A somewhat similar question as above is \emph{what happens to the inner force point after collapsing the $\varepsilon$ with $0$. Do they vanish?} I conjecture that yes, the inner force point vanishes once collapsed. 
\subsection{Generalized $\sle$ bubbles on $\mathbb H$: multiple case}
Going one step further, motivated by the induction procedure described in Figure \ref{fig:Induction}, we are also interested in understanding the multiple $\sle$ bubbles on $\mathbb H$. Specifically, consider welding of three quantum disks 
\begin{equation}\label{eqn:weld three disks}
	\int_0^\infty \int_0^\infty \qd_{0,1}(\gamma,\beta_W;\ell)\times \mathcal M^{\disk}_{0,2}(W_1,\ell,r)\times\mathcal M^{\disk}_{0,2}(W_2;r,\cdot)drd\ell
\end{equation}
for $W\geq 2,W_1>0$ and $W_2>0$. 
\begin{figure}
	\includegraphics[scale=0.8]{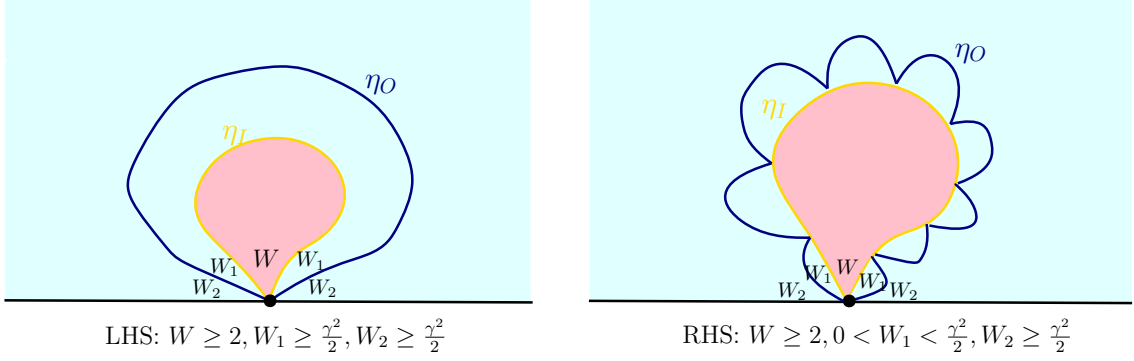}
	\centering
	\caption{Illustration of quantum surface (\ref{eqn:weld three disks}) when embedded in $(\mathbb H,\phi,\eta_I,\eta_O)$.}
	\label{fig:Conj_multi_bubble}
\end{figure}
Let $(\mathbb H,\phi,0,\eta_I,\eta_O)$ be an particular embedding of (\ref{eqn:weld three disks}) (see Figure \ref{fig:Conj_multi_bubble}), then it is not hard to show that the joint law of $(\eta_I,\eta_O)$ is independent of $\phi$. Moreover, the condition law $(\eta_O|\eta_I)$ should equal to $\sle_{\kappa,0}^{\bub}(W_1,W_2)$ and the law of $(\eta_I|\eta_O)$ should equal to $\sle_{\kappa,0}^{\bub}(W,W_1)$. Recall that $\sle_{\kappa,0}^{\bub}(\cdot,\cdot)$ is the welding interface in Conjecture \ref{conj:gen_SLE_bubble}. 
\par
The interesting questions to the SLE research communities are \emph{what is the marginal law of $\eta_{\bullet},\bullet\in \lbrace I,O\rbrace$. Moreover, what is the Loewner evolution  (driving function) of $\eta_{\bullet},\bullet\in \lbrace I,O\rbrace$?}
\subsection{Scaling limits of bubble-decorated quadrangulated disks}\label{sec:conj on scaling limits}
Recall that in the $\sle$ loop case \cite{SLEloopviawelding}, $\ms^n\otimes\saw^n$ is the measure on pairs $(M,\eta)$, where $M$ is a quadrangulation, $\eta$ is a self-avoiding loop on $M$, and each $(M,\eta)$ has weight $n^{5/2}12^{-\#\mathcal F(M)}54^{-\#\eta}$, where $\#\mathcal F(M)$ denotes the number of faces of $M$ and $\#\eta$ is the number of edges of $\eta$. It is proved that the following convergence result holds. 
	\begin{Th}[{\cite[Theorem 1.2]{SLEloopviawelding}}]
		There exists constant $c_0>0$ and for all $c\in (0,1)$, 
		\begin{equation}
			\ms^n\otimes\saw^n|_{A(c)} \xrightarrow{w} c_0\cdot \qs\otimes \sle_{8/3}^{\lo}|_{A(c)},
		\end{equation}
		where $A(c)$ is the event that the length of the loop is in $[c,c^{-1}]$.
	\end{Th}
	In the disk case, we say a planar map $D$ is a \emph{quadrangulated disk} if it is a planar map where all faces have four edges except for the exterior face, which has arbitrary degree and simple boundary. Let $\partial D$ denote the edges on the boundary of the exterior face, and we denote $\#\partial D$ the boundary length of $D$. Let $\md^n$ be the measure on the quadrangulated disks such that each disk $D$ has weight $n^{5/2}12^{-\#\mathcal F(D)}54^{-\#\partial D}$, which has the same scaling as $\ms^n$ above. Note that here if $D$ is sampled from $\md^n$, then $D$ is viewed as a metric measure space by considering the graph metric rescaled by $2^{-1/2}n^{-1/4}$ and giving each vertex mass $2(9n)^{-1}$. 
	\par 
	If $D$ is a quadrangulated disk, then we say $\eta$ is a \emph{self-avoiding bubble} on $D$ rooted at $e_r\in \partial D$ if $\eta$ is an orderer set of edges $e_1,\ldots,e_{2k}\in \mathcal E(D)$ with $r\in \lbrace 1,\ldots,2k\rbrace$ and $e_j$ and $e_i$ share an end-point if and only if $|i-j|\leq 1$  or $(i,j)\in \lbrace (1,2k),(2k,1)\rbrace$. 
	\par
	Let $\md^n\otimes \partial \md^n\otimes\sab^n$ denote the measure on pairs $(D,e,\eta)$ where $\eta$ is a self-avoiding bubble on $D$ rooted at edge $e\in \partial D$ and the pair $(D,\eta)$ has weight $\#\partial D^{-1}\cdot n^{5/2}12^{-\#\mathcal F(D)}54^{-\#\eta}$. For $(D,e,\eta)$ sampled from $\md^n\otimes \partial \md^n\otimes\sab^n$, we view $D$ as a metric measure space and view $\eta$ as a bubble on this metric measure space rooted at edge $e$ so that the time it takes to traverse each edge on the loop is $2^{-1}n^{-1/2}$. 
	\begin{Conj}
		There exists some $c_0>0$ such that for all $c\in (0,1)$, 
		\begin{equation}
			\md^n\otimes \partial \md^n\otimes\sab^n|_{A(c)} \xrightarrow{w} c_0 \cdot \LF_{\mathbb H}^{(\beta_6,p)}\times \sle_{8/3,p}^{\bub}\times dp|_{A(c)},
		\end{equation}
		in Gromov-Hausdorff-Prokhorov-uniform topology, where $A(c)$ is event that the length of the bubble is in $[c,c^{-1}]$. 
	\end{Conj}
	We can also understand the measure $\md^n\otimes \partial \md^n\otimes\sab^n$ from the welding perspective. Suppose $\overline{\md^n}$ is a measure on qudrangulated disks such that each disk $\overline D$ has weight $n^{5/2}12^{-\#\mathcal F(\overline D)}54^{-2\#\partial \overline D}$ and $\underline{\md^n}$ is a measure on qudrangulated disks with each disk $\underline{D}$ has weight $n^{5/2}12^{-\#\mathcal F(\underline D)}54^{-\#\partial \underline D}$. Let $\overline{\md_{0,2}^n}$ be the measure on $(\overline D, e_1, e_2)$ such that we first sample $\overline D$ from reweighted measure $(\#\partial \overline{\md^n})^2\overline{\md^n}$ and then sample two edges $e_1,e_2$ uniformly on $\partial \overline D$.
	Similarly, let $\underline{\md^n_{0,1}}$ be the measure on $(\underline{D},e)$ such that we first sample $\underline D$ from reweighted measure $(\#\partial \underline{\md^n})\cdot \underline{\md^n}$ and then sample an edge $e$ from $\partial\underline D$ uniformly.
	\par
	For $k\in \mathbb N$, let $\overline{\md_{0,2}^n}(\cdot,k)$ denote the restriction of $\overline{\md_{0,2}^n}$ to the event that right boundary has length $2k$ and let $\underline{\md_{0,1}^n}(k)$ denote the restriction of $\underline{\md_{0,1}^n}$ to the event that the total boundary has length $2k$. Let $\overline{\md_{0,2}^n}(\cdot,k)^\#$ and $\underline{\md_{0,1}^n}(k)^\#$ denote the corresponding probability measure respectively. 
	\par
	Suppose $(\overline D, e_1,e_2)$ is sampled from $\overline{\md_{0,2}^n}(\cdot,k)^\#$ and $(\underline D,e)$ is sampled from $\underline{\md_{0,1}^n}(k)^\#$, then we can do the ``discrete conformal welding'' by identifying the right boundary of $\overline D$ to the total boundary of $\underline D$ such that $e_1,e_2$ and $e$ are identified. The self-avoiding bubble on the discrete disk represents the welding interface of $\overline D$ and $\underline D$. We parametrize the bubble so that each edge on the bubble has length $2^{-1}n^{-1/2}$ just like the sphere case. Suppose $(\overline D, e_1,e_2)$ is sampled from $\overline{\md_{0,2}^n}(\cdot,k)^\#$ and $(\underline D,e)$ is sampled from $\underline{\md_{0,1}^n}(k)^\#$, then we denote the measure on the disks decorated with a self-avoiding bubble sampled in this way by $\weld_d^{\bub}(\overline{\md_{0,2}^n}(\cdot,k)^\#, \underline{\md_{0,1}^n}(k)^\#)$. Similarly, let $\weld_c^{\bub}(\qd_{0,2}(\cdot,\ell)^\#,\qd_{0,1}(\ell)^\#)$ denote the measure on bubble-decorated quantum disk obtained by identifying the right boundary of the disk sampled from $\qd_{0,2}(\cdot,\ell)^{\#}$ and the total boundary of the disk sampled from $\qd_{0,1}(\ell)^{\#}$.
	\begin{Conj}
		For any $\ell>0$, we have 
		\begin{equation}
			\weld_d^{\bub}(\overline{\md_{0,2}^n}(\cdot,[\ell n^{1/2}] )^\#, \underline{\md_{0,1}^n}([\ell n^{1/2}])^\#)\xrightarrow{w} \weld_c^{\bub}(\qd_{0,2}(\cdot,\ell)^{\#},\qd_{0,1}(\ell)^{\#})
		\end{equation}
		in Gromov-Hausdorff-Prokhorov-uniform topology. 
	\end{Conj} 

\bibliographystyle{alpha}
\bibliography{SLEbubble}

\end{document}